\documentclass[3p]{elsarticle}

\usepackage{lineno,hyperref}
\modulolinenumbers[5]
\usepackage{amsmath,amssymb,amsthm}
\usepackage{fixmath}
\usepackage{xcolor}
\journal{Applied and Computational Harmonic Analysis}

\allowdisplaybreaks
\DeclareMathOperator{\dvol}{dvol}
\DeclareMathOperator{\vol}{vol}
\DeclareMathOperator{\Real}{Re}
\DeclareMathOperator{\Imag}{Im}
\DeclareMathOperator{\divr}{div}
\DeclareMathOperator{\var}{var}
\DeclareMathOperator{\grad}{grad}
\DeclareMathOperator{\spn}{span}
\DeclareMathOperator{\diag}{diag}
\newcommand{\ii}{\mathrm{i}}
\newtheorem{thm}{Theorem}
\newtheorem{prop}[thm]{Proposition}
\newtheorem{lemma}[thm]{Lemma}
\newdefinition{defn}[thm]{Definition}
\newdefinition{rk}[thm]{Remark}
\newdefinition{alg}{Algorithm}
\newdefinition{cor}[thm]{Corollary}

\biboptions{sort&compress}

\begin{document}

\begin{frontmatter}

\title{Data-driven spectral decomposition and forecasting of ergodic dynamical systems}

\author{Dimitrios Giannakis\corref{mycorrespondingauthor}}
\address{Courant Institute of Mathematical Sciences, New York University, New York, NY 10012, USA}
\cortext[mycorrespondingauthor]{Corresponding author \ead{dimitris@cims.nyu.edu}}

\begin{abstract}
  We develop a framework for dimension reduction, mode decomposition, and nonparametric forecasting of data generated by ergodic dynamical systems. This framework is based on a representation of the Koopman and Perron-Frobenius groups of unitary operators in a smooth orthonormal basis of the $ L^2 $ space of the dynamical system, acquired from time-ordered data through the diffusion maps algorithm. Using this representation, we compute Koopman eigenfunctions through a regularized advection-diffusion operator, and employ these eigenfunctions in dimension reduction maps with projectible dynamics and high smoothness for the given observation modality. In systems with pure point spectra, we construct a decomposition of the generator of the Koopman group into mutually commuting vector fields that transform naturally under changes of observation modality, which we reconstruct in data space through a representation of the pushforward map in the Koopman eigenfunction basis. We also establish a correspondence between Koopman operators and Laplace-Beltrami operators constructed from data in Takens delay-coordinate space, and use this correspondence to provide an interpretation of diffusion-mapped delay coordinates for this class of systems. Moreover, we take advantage of a special property of the Koopman eigenfunction basis, namely that the basis elements evolve as simple harmonic oscillators, to build nonparametric forecast models for probability densities and observables. In systems with more complex spectral behavior, including mixing systems, we develop a method inspired from time change in dynamical systems to transform the generator to a new operator with potentially improved spectral properties, and use that operator for vector field decomposition and nonparametric forecasting.           
\end{abstract}


\end{frontmatter}


\section{Introduction}
  
\subsection{\label{secBackground}Background and motivation}
 
In many branches of science and engineering, one is faced with the problems of dimension reduction and forecasting of high-dimensional time series. When these time series are generated by ergodic dynamical systems (or by ergodic components of non-ergodic systems), they posses an important special property, namely that long-time averages are equivalent to expectation values with respect to an invariant measure of the dynamics. This property enables sampling of the full phase space from long time series and sufficiently high-dimensional observables, and data analysis algorithms can be employed to perform dimension reduction and forecasting of these systems with no a priori knowledge of the equations of motion. This nonparametric approach is useful in a variety of contexts, such as when the equations of motion are unknown or partially known, when the system cannot be feasibly simulated, or when the solutions of the equations of motion are complicated and a decomposition into simpler components (modes) is desired.      

The vast array of methods developed to address these goals can be broadly categorized as  state-space or operator-theoretic methods \citep{BudisicEtAl12}. A popular state-space approach is to approximate a nonlinear dynamical system by a collection of local linear models on the tangent planes of the attractor \citep[][]{FarmerSidorovich87,Sauer92,KugiumtzisEtAl98}; other approaches construct global nonlinear models for the dynamical evolution map \cite{BroomheadLowe88}, or nonlinearly project the attractor to lower-dimensional Euclidean spaces and construct reduced models operating in those spaces \cite{KevrekidisEtAl04,TalmonCoifman13,CrosskeyMaggioni17}. A common element of these techniques is that the forward operators of the reduced models are defined in state space; i.e., they map the state at a given time to another state in the future. On the other hand, operator-theoretic techniques \citep[][and others]{MezicBanaszuk04,Mezic05,Mezic13,SchmidSesterhenn08,RowleyEtAl09,ChenEtAl12,Schmid10,BudisicEtAl12,JovanovicEtAl14,TuEtAl14,HematiEtAl15,WilliamsEtAl15,KutzEtAl16,BruntonEtAl17,ArbabiMezic16,ProctorEtAl16,BudisicMezic12,DellnitzJunge99,DellnitzEtAl00,FroylandDellnitz03,Froyland07,Froyland08,FroylandEtAl14b,Froyland05,FroylandPadberg09,FroylandEtAl10,FroylandEtAl10b,SchutteEtAl10,FroylandEtAl14,BerryEtAl15} shift attention away from the state-space perspective, and focus instead on the action of dynamical systems on spaces of observables or measures.

Remarkably, the action of a nonlinear dynamical system on appropriately constructed linear spaces of observables and measures can be characterized without approximation by groups (or semigroups) of \emph{linear} operators, known as Koopman \citep[][]{Koopman31} or Perron-Frobenius operators, respectively. In particular, Koopman operators (also known as composition operators) act on observables by composition with the dynamical flow map, while Perron-Frobenius operators (also known as Ruelle \cite{Ruelle68} or transfer operators) act on measures by pullbacks. Typically, the spaces of observables and measures in question are infinite-dimensional, so one can think of a tradeoff between a finite-dimensional nonlinear system and a group of linear operators acting on an infinite-dimensional space. Nevertheless, the intrinsically linear structure of these spaces makes Koopman and Perron-Frobenius operators amenable to treatment through the full machinery of linear operator theory and the associated finite-dimensional approximation schemes (e.g., Galerkin methods). Data-driven operator-theoretic techniques exploit this structure to perform tasks such as spectral analysis of complex systems \cite{MezicBanaszuk04,Mezic05,Mezic13}, identification of coherent sets (e.g., invariant, almost-invariant, periodic, and almost-periodic sets)  \cite{DellnitzJunge99,DellnitzEtAl00,FroylandDellnitz03,Froyland05,Froyland07,Froyland08,FroylandPadberg09,FroylandEtAl10,FroylandEtAl10b,FroylandEtAl14b}, identification of dynamic modes \cite{SchmidSesterhenn08,Schmid10,RowleyEtAl09,ChenEtAl12,JovanovicEtAl14,TuEtAl14,WilliamsEtAl15,HematiEtAl15,KutzEtAl16,BruntonEtAl17,ArbabiMezic16}, computation of ergodic and harmonic quotients \cite{BudisicEtAl12,BudisicMezic12}, nonparametric prediction \citep[][]{BerryEtAl15}, modeling of metastable \cite{SchutteEtAl10} and slow-fast systems \cite{FroylandEtAl14}, control \cite{ProctorEtAl16}, and other applications. In particular, \cite{BudisicMezic12,BudisicEtAl12} developed a method for analyzing ergodic and harmonic quotients based on the diffusion maps algorithm \cite{CoifmanLafon06}; in what follows, we will use diffusion maps to learn smooth orthonormal basis sets of functions from the data, and employ these bases in various Galerkin and spectral methods for dimension reduction, mode decomposition, and nonparametric forecasting based on Koopman and Perron-Frobenius operators.   

For appropriately chosen spaces of observables and measures, the Koopman and Perron-Frobenius operators are dual pairs, and therefore provide theoretically equivalent information. For instance, in the setting of an ergodic dynamical system, natural spaces of observables are  $L^2 $ spaces of complex-valued scalar functions associated with invariant probability measures, and natural spaces of measures are complex measures with $ L^2 $ densities. Yet, at an applied level, the fact that Koopman and Perron-Frobenius operators act on fundamentally different entities appears to have led to the development of two fairly distinct families of approximation techniques (though the dichotomy is not rigid, and there are references in the literature utilizing results from both the Koopman and Perron-Frobenius perspectives; e.g, \cite{WilliamsEtAl15}).

Historically, data-driven techniques based on Perron-Frobenius operators \cite{DellnitzJunge99,DellnitzEtAl00,FroylandDellnitz03,Froyland05,Froyland07,Froyland08,FroylandPadberg09,FroylandEtAl10,FroylandEtAl10b,SchutteEtAl10,FroylandEtAl14,FroylandEtAl14b} began with the work of Dellnitz and Junge \cite{DellnitzJunge99} in 1999. A common approach in these techniques is to approximate the spectrum of the Perron-Frobenius operator through a scheme known as Ulam's method \cite{Ulam64}. Essentially, this involves partitioning the state space into a finite collection of disjoint subsets, and estimating the transition probabilities between these subsets by counting the corresponding transitions in a large ensemble of simulations or experiments. The resulting transition probability matrix can be interpreted as a Galerkin projection of a smoothed (compact) transfer operator for the system perturbed by a small amount of noise, and the eigenvectors of that matrix at eigenvalues on or near the unit circle are used to identify coherent sets. In subsequent work, Dellnitz, Froyland, et al.~\cite{DellnitzEtAl00,FroylandDellnitz03,Froyland07,Froyland08,FroylandEtAl14b} studied certain classes of systems possessing quasi-compact Perron-Frobenius operators, where it was rigorously established that Ulam's method is able to provide accurate approximations of isolated Perron-Frobenius eigenvalues and their corresponding eigenfunctions.  

Data-driven techniques based on Koopman operators were first developed by Mezi\'c and Banaszuk~\cite{MezicBanaszuk04} and Mezi\'c~\cite{Mezic05} in 2004 and 2005, respectively. In particular, \cite{Mezic05} established the Koopman mode expansion for measure-preserving systems and $ L^2 $ observables. In this setting, the Koopman and Perron-Frobenius operators are unitary and adjoint to one another, and have a well-defined spectral expansion consisting in general of both discrete (pure point) and continuous parts. In \cite{MezicBanaszuk04,Mezic05}, a method for estimating the point spectrum of Koopman operators was developed using generalized Laplace analysis. This approach involves computing harmonic averages (Fourier transforms) of time series of observables for different initial conditions, and identifying the frequencies leading to non-vanishing averages. For those frequencies, the harmonic average corresponds to a projection of the observable onto an eigenspace of the Koopman operator, allowing one to identify Koopman eigenvalues and their corresponding periodic eigenfunctions, as well as spatial patterns called Koopman modes.

In later work, Rowley et al.~\cite{RowleyEtAl09} established that the Koopman mode expansion has close connections with the dynamic mode decomposition (DMD) technique of Schmid and Sesterhenn \cite{SchmidSesterhenn08} and Schmid \cite{Schmid10} for decomposing time-ordered spatiotemporal datasets into so-called dynamic modes. Similarly to the proper orthogonal decomposition (POD)  \citep[][]{AubryEtAl91,HolmesEtAl96}, DMD extracts modes by solving an eigenvalue problem for a matrix $ A $ constructed from experimental snapshots, but instead of the covariance matrix used in POD, DMD employs a time-shifted cross-correlation matrix $A$ capturing the linear dependence of the snapshots at the next time step on the snapshots at the current time step. To address the computational cost associated with the high dimensionality of typical spatiotemporal data, the DMD algorithm in~\cite{RowleyEtAl09} employs an Arnoldi-type iterative method which does not require explicit formation of $ A $. In particular, it was shown that the eigenvalues and eigenvectors obtained via this algorithm approximate the eigenvalues and modes in the Koopman mode expansion in \cite{Mezic05}. Tu et al.~\cite{TuEtAl14} subsequently developed an alternative formulation of DMD based on matrix pseudoinverses (or truncated pseudoinverses) that does not require the data to be in the form of a single time series, and moreover returns estimates of the Koopman eigenfunctions in addition to the eigenvalues and Koopman modes.  

The connections between the Koopman mode expansion and DMD were further studied  by Williams et al.~\cite{WilliamsEtAl15}, who  developed an extended DMD (EDMD) framework. In EDMD, the eigenvalue problem for the Koopman operator is projected to a finite-dimensional matrix eigenvalue problem associated with the action of the Koopman operator on a general dictionary of observables. In particular, standard DMD  can be viewed as a particular instance of EDMD for the dictionary formed by the components of the state vector, but the general approach advocated in EDMD is to employ richer dictionaries in order to improve the accuracy of the computed eigenvalues and Koopman modes. The dictionaries proposed in~\cite{WilliamsEtAl15} include Hermite polynomials in ambient data space, radial basis functions, and discontinuous spectral elements analogous to the bases used in Ulam approximations of Perron-Frobenius operators. When these dictionaries are sufficiently rich, EDMD converges to a Galerkin method for the Koopman eigenvalue problem, although issues related to non-compactness and potentially continuous spectrum of the Koopman operator were not addressed in \cite{WilliamsEtAl15}. More recently, Brunton et al.\ \cite{BruntonEtAl17} and Arbabi and Mezi\'c \cite{ArbabiMezic16} have developed Hankel matrix techniques for Koopman eigenfunction approximation using dictionaries of delay-coordinate mapped observables.
        
Irrespective of the method employed to compute them, the triplets of eigenvalues, eigenfunctions, and Koopman modes produce a decomposition a complex spatiotemporal signal into simpler ``building blocks'', which are intrinsic to the dynamical system generating the data. In particular, the Koopman eigenvalues and eigenfunctions are independent of observation modality, meaning that data acquired from different sensors would yield the same results for the eigenvalues and eigenfunctions, so long as the data contain sufficiently rich information.  In contrast, the results of POD-type techniques and many nonlinear dimension reduction techniques depend on the observation modality. The issue of invariance under changes of observation modality has been at the focus of several recent data analysis techniques \citep[][]{SingerCoifman08,TalmonCoifman13,BerryEtAl13,Giannakis15,BerrySauer16b,DsilvaEtAl16,YairEtAl16}. Of course, the efficacy of the Koopman and Perron-Frobenius methods in real-world applications depends strongly on a number of factors, including the properties of the underlying dynamical system, the discretization scheme, and the measurement and data acquisition apparatus. To motivate our work, we end this section by summarizing some of the challenges encountered in data-driven operator-theoretic techniques for dynamical systems.

First, at a fundamental level, for systems of high complexity the spectra of the Koopman and Perron-Frobenius operators are not amenable to eigendecomposition. For instance, a necessary and sufficient condition for a dynamical system to be weak-mixing (a weak form of chaos, which implies ergodicity) is that its Koopman operators have no $ L^1 $  eigenfunctions other than the constant function \cite{KornfeldSinai00,HasselblattKatok02}. While such cases with continuous spectra can be theoretically handled through the Koopman mode expansion \cite{Mezic05}, in numerical implementations involving matrix algebra one always obtains eigenvectors, whose properties become difficult to analyze. These observations suggests that numerical methods for Koopman eigenfunctions must involve, either implicitly or explicitly, some form of regularization or annealing. For instance, in dictionary or Galerkin discretizations regularization is implicitly provided through projection on a finite-dimensional solution space. In fact, regularization is warranted even in simple cases involving  systems with pure point spectra and complete bases of Koopman eigenfunctions. For example, in irrational flows on tori (arguably, among the simplest ergodic dynamical systems)  and at dimension greater than 1, the frequencies associated with the Koopman eigenvalues are dense on the real line, and this can lead to poorly conditioned numerical schemes unless care is taken to filter out highly oscillatory solutions. 

A second important issue concerns the choice of dictionary or basis employed for operator approximation. This issue becomes especially pertinent in situations where the ambient space dimension is high, and/or the data lies on an a priori unknown subset of ambient data space of small or zero measure. As a concrete example, consider again an irrational flow on an $m$-torus embedded in $ \mathbb{ R }^n $. In this case, computing Koopman eigenfunctions in a smooth function basis for $ \mathbb{ R }^n $ (e.g., Hermite polynomials) can be problematic due to computational cost (which increases exponentially with $ n$ independently of the intrinsic dimension $ m $), but more importantly due to the fact that the eigenfunctions are supported on a set of Lebesgue measure zero, and therefore cannot be well represented in a smooth basis for functions on $ \mathbb{ R }^n $. In such situations, which are quite prevalent in real-world applications, it is preferable to work in a basis which is adapted to the intrinsic geometry of the data.

A third issue concerns the identification of spatial modes. Since the Koopman and Perron-Frobenius eigenfunctions are intrinsic to the dynamical system generating the data, one would like that the associated modes are also intrinsic; i.e., that they transform naturally as vector-valued functions under changes of observation modality. However, in the standard DMD and EDMD construction, the Koopman modes are given by global averages of the data weighted by the corresponding Koopman eigenfunctions, and such averages do not transform naturally under general nonlinear transformations of the data.
     
\subsection{Contributions of this work}

In this paper, we develop operator-theoretic techniques for dimension reduction and mode decomposition of data generated by ergodic dynamical systems. Building on these methods, we also develop nonparametric techniques for forecasting observables and probability measures. A key ingredient of our approach is a data-driven orthonormal basis for the $ L^2 $ space of the dynamical system constructed using the diffusion maps algorithm \cite{CoifmanLafon06} in conjunction with variable-bandwidth kernels \cite{BerryHarlim16}. The construction of this basis follows the approach of Berry et al.~\citep[][]{BerryEtAl15} for approximating Kolmogorov and Fokker-Planck operators (the stochastic analogs of the Koopman and Perron-Frobenius operators, respectively) of stochastic dynamical systems on manifolds. Here, we employ this basis in a family of spectral and  Galerkin methods for the eigenvalue problem for the generator of the Koopman group, reconstruction of vector fields (which can be thought of as a generalization of the  notion of Koopman modes), and nonparametric forecasting. Ergodicity enables the implementation of these techniques from a single time series of measurements of the state vector without requiring (potentially costly  or infeasible) ensembles of experiments with different initial conditions. The techniques are also applicable in the case of partial observations, so long as the observation map is such that one of the variants of Takens' delay embedding theorem applies \cite{PackardEtAl80,Takens81,SauerEtAl91,Robinson05,DeyleSugihara11}. We demonstrate our approach through analytical and numerical applications to ergodic dynamical systems on the 2- and 3-torus, including systems with mixing \cite{Fayad02} and fixed points \cite{Oxtoby53}. 

The data-driven basis from diffusion maps has several attractive properties that help address the challenges in operator-theoretic techniques outlined in Section~\ref{secBackground}:
\begin{enumerate}
\item The basis functions are naturally supported on the subset of ambient data space on which the system evolves (e.g., a low-dimensional attractor), and their computation requires no a priori knowledge of that subset. In particular, under relatively mild assumptions, the diffusion maps basis converges in the limit of large data to a complete orthonormal basis of the $ L^2 $ space of the dynamical system. By orthonormality of the basis, passing from a Koopman to a Perron-Frobenius operator and vice versa can be accomplished by a standard matrix transpose.  
\item The cost of computing the basis functions scales linearly with the ambient space dimension $ d $, and after the basis has been computed, the cost of operator approximation is independent of $ d$. This property is particularly important in high-dimensional applications such as analysis of engineering and geophysical fluid flows (e.g., \cite{GiannakisEtAl15,SlawinskaGiannakis16}). 
\item Associated with the diffusion eigenfunctions is a Dirichlet energy (given by the corresponding eigenvalues), which measures the roughness of the eigenfunctions in a Riemannian geometry that depends on the diffusion maps kernel. Here, the kernel is constructed explicitly so that the Riemannian metric, $h$, has compatible volume form with the invariant measure of the dynamics. As a result, our Galerkin approximation spaces for the Koopman eigenvalue problem are spanned by orthonormal functions with minimal expected roughness (with respect to the invariant measure). 
\item Through an eigenvalue-dependent normalization, the diffusion eigenfunctions form a natural basis for the $ H_1 $ Sobolev space associated with $h$. In this basis, the Dirichlet form associated with $ h $ is represented by an identity matrix, and therefore diffusion regularization can be carried out in a computationally efficient and well-conditioned manner. 
\item Through the use of kernels in operating in Takens delay-coordinate space \cite{GiannakisMajda12a,BerryEtAl13}, the basis functions from diffusion maps can be made to converge to Koopman eigenfunctions in pure point spectrum systems, leading to efficient and noise-robust Galerkin schemes.    
\item The variable-bandwidth family of kernels provides a flexible framework to perform a change of measure which is equivalent to a time change \citep[][]{KatokThouvenot06} in the dynamical system---we use such a change of measure as a regularization tool in certain classes of systems with continuous Koopman spectra. 
\end{enumerate}

In what follows, we first develop our approach in the case of systems with pure point spectra, where the Koopman eigenfunctions form a complete orthonormal basis of the $ L^2 $ space associated with the invariant measure. We formulate a Galerkin method for the Koopman eigenvalue problem in the appropriate  $ H_1 $ space on the phase space manifold---this is particularly important in systems with non-isolated eigenfrequencies (including pure point spectrum systems), as it eliminates highly oscillatory eigenfunctions from the spectrum through Tikhonov regularization. Using the Dirichlet energy functional available from diffusion maps, we select the least rough set of Koopman eigenfunctions with corresponding rationally independent eigenvalues, and take advantage of a group structure of the Koopman eigenfunctions to generate a complete basis of the $ L^2 $ space on the manifold recursively from the group generators. This basis is employed to decompose the vector field of the dynamics into a set of linearly-independent, nowhere-vanishing, mutually commuting vector fields (i.e., a set of dynamically independent components), and to represent the pushforward map carrying along these vector fields in data space. Associated with the reconstructed vector fields are spatiotemporal patterns that can be thought of as analogs of Koopman modes, but transforming naturally (as tensors) under changes of observation modality. The generating eigenfunctions are also used to construct factor maps \cite{MezicBanaszuk04,Mezic05} from the phase space manifold to the complex plane, under which the dynamical vector field is projectible, and the dynamics in the image spaces are simple harmonic oscillations with frequencies given by the corresponding eigenvalues. In addition, we show that through the use of delay-coordinate maps, the Laplace-Beltrami operator approximated by diffusion maps in the limit of infinitely many delays commutes with the dynamical vector field (which is equivalent to the generator of the Koopman group restricted to smooth functions); as a result, the two operators have common eigenfunctions. This result bridges these two important families of dimension reduction and mode decomposition techniques, and naturally leads to efficient Galerkin schemes for the Koopman eigenvalue problem with high robustness to i.i.d.~measurement noise. 

Besides dimension reduction and mode decomposition, we utilize the Koopman eigenvalues and eigenfunctions in nonparametric forecasting schemes for probability measures and observables. This approach is closely related to the nonparametric forecasting method for stochastic systems developed in \cite{BerryEtAl15}, with the difference that we advance probability densities using the simple harmonic oscillator structure of the Koopman eigenfunctions, as opposed to taking powers of a finite-dimensional Fokker-Planck operator approximation. In particular, Koopman eigenfunctions of pure point spectrum systems evolve as uncoupled oscillators with frequencies given by the corresponding eigenvalues, and by completeness, these eigenfunctions can be used to predict the time evolution of arbitrary observables in $ L^2 $ and probability measures with $ L^2 $ densities. 

Next, we consider systems whose Koopman operators do not have pure point spectra, including a class of weak mixing systems with no nonconstant eigenfunctions. There, a regularization technique inspired from the theory of time change in dynamical systems \citep[][]{KatokThouvenot06} is developed which attempts to construct a new dynamical system having the same orbits as the original one, but with improved spectral properties for eigendecomposition and forecasting. In particular, we perform a time-change transformation using the norm of the dynamical vector field in the ambient data space as the time-change function. This technique can be applied to arbitrary dynamical systems as it requires no a priori knowledge of the governing equations, and in special cases it recovers a pure point spectrum system with the same orbits as the mixing system generating the data. We generalize our vector field decomposition and nonparametric forecasting techniques to the time-changed framework, where the recovered vector field components become non-commuting, and the simple harmonic oscillators acquire couplings and time-dependent frequencies.  

This paper is organized as follows.  In Section~\ref{secErgodicity}, we lay out our notation and summarize basic results from ergodic theory. In Section~\ref{secPurePointSpectra}, we formulate our dimension reduction, mode decomposition, and forecasting techniques for systems with pure point spectra. We discuss the numerical implementation of these techniques, including the Galerkin method with diffusion regularization for the Koopman eigenvalue problem, in Section~\ref{secGalerkinImplementation}, where we also present numerical applications to variable-speed flows on the 2-torus. In Section~\ref{secTakens}, we extend this scheme to delay-coordinate mapped data, and establish the correspondence between the Laplace-Beltrami and Koopman operators arising in the limit of infinitely many delays along with methods for removing i.i.d.\ observational noise. In Section~\ref{secTimeChange}, we formulate our time-change approach for improving the spectral properties of the Koopman group, and present applications to mixing dynamical systems on the 3-torus and fixed-point ergodic flows on the 2-torus. We conclude in Section~\ref{secConclusions}. Appendices A--C contain auxiliary material and technical results, including  high-level listings of the algorithms developed in the main text, and an asymptotic analysis of denoising using diffusion maps in conjunction with delay-coordinate maps. Videos showing the temporal evolution of probability densities in the applications of Section~\ref{secGalerkinImplementation} and~\ref{secTimeChange} are included as supporting online material (SOM). Also included as SOM is a comparison of our Koopman eigenfunction results to Koopman eigenfunctions computed via EDMD. An application of the techniques developed here to atmospheric dynamics can be found in \cite{GiannakisEtAl15,SlawinskaGiannakis16}.

\section{\label{secErgodicity}Notation and basic results from ergodic theory}

Let $ M $ be a closed, connected, compact, orientable, smooth (of class $ C^\infty $),  $ m $-dimensional manifold with a Borel $ \sigma $-algebra $ \mathcal{ B } $. In what follows, $ M $ will be the phase space of a continuous-time dynamical system with a smooth, invertible evolution map $ \Phi_t : M \mapsto M $, $ t \in \mathbb{ R } $. The evolution map induces a map $ \Phi_{t*} $ on measures on $ \mathcal{ B } $ such that for a measure $ \mu $ and $ A \in \mathcal{ B } $, $ \Phi_{t*} \mu( A ) = \mu( \Phi_t^{-1} (A) ) = \mu( \Phi_{-t}( A )  ) $. We will take $ \mu $ to be a $ \Phi_t $-invariant probability measure  (i.e., $ \mu( M ) = 1 $ and $ \Phi_{t*} \mu = \mu $), characterized by a smooth, positive density relative to local coordinates. We consider that  the dynamical system is observed at a fixed sampling interval $ T  $ through a smooth, vector-valued observation function $ F : M \mapsto \mathbb{ R }^d $. That is, we have a dataset consisting of time-ordered samples $ \{ x_0, x_1, \ldots, x_{N-1} \} $ in the $ d $-dimensional data space, with 
\begin{equation}
  \label{eqDataset}
  x_i = F( a_i ), \quad a_i = \Phi_{t_i}( a_0 ), \quad t_i = T ( i - 1 ). 
\end{equation}

Without loss of generality, we assume that $ F $ is an embedding of $ M $ into $ \mathbb{ R }^d $, i.e., it is one-to-one and its derivative, $ F_* \rvert_a : T_a M \mapsto T_{F(a)} \mathbb{ R }^d $, has full rank at every $ a \in M $. If $ F $ is not an embedding, then with high probability an embedding can be constructed from $ F $ using delay-coordinate maps \citep[][]{PackardEtAl80,Takens81,SauerEtAl91,Robinson05,DeyleSugihara11}, and the methods described below can be applied in that space (e.g., \cite{GiannakisMajda12a,BerryEtAl13}). The embedding of $ M $ into data space induces a smooth Riemannian metric tensor $ g $ on $ M $ given by pulling back the canonical inner product of $ \mathbb{ R }^d $, i.e., $ g( u_1, u_2 ) = F_* \, u_1 \cdot F_* \, u_2 $ for any two tangent vectors $ u_1, u_2 \in T_a M $. We denote the Riemannian volume form  associated with $ g $ and the corresponding Riemannian measure by $ \dvol_g $ and $ \vol_g $, respectively. From our point of view, $ F $, $ g $, and $ \vol_g $ are extrinsic objects to the dynamical system and we would like the results of our dimension reduction and mode decomposition methods to not depend on them.

Next, we introduce the Koopman and Perron-Frobenius groups of operators  and their generator which will play a central role in our study. We refer the reader to one of the many books and reviews on ergodic theory (e.g., \cite{Walters81,BunimovichEtAl00,HasselblattKatok02,EisnerEtAl15}) for more detailed expositions on these and other related topics. Let $ L^2( M, \mu ) $ be the Hilbert space of square-integrable, complex valued functions on $ M $ with inner product $ \langle f_1, f_2 \rangle = \int_M f_1^* f_2 \, d\mu $ and the corresponding norm $ \lVert f \rVert = \langle f, f \rangle^{1/2} $. In general, $ L^2( M, \mu ) $ is different from the Hilbert space $ L^2( M, \vol_g ) $ associated with the Riemannian inner product $ \langle f_1, f_2 \rangle_g = \int_M f_1^* f_2 \dvol_g $, but because $ M $ is compact the two spaces are isomorphic. The Koopman operator $ U_t : L^2( M, \mu ) \mapsto L^2( M, \mu ) $ for time $ t \in \mathbb{ R } $ is defined via  $ U_t f = f \circ \Phi_t $; that is, for $ \mu $-a.e.\ $ a \in M$, the $ L^2 $ function  $ f_t = U_t f $ is equal to $ f $ evaluated at the shifted point $ a_t = \Phi_t( a )$ along the dynamics, $ f_t( a ) = f( a_t ) $. The set of Koopman operators $ \{ U_t \}_{t\in\mathbb{R}} $ forms an Abelian group under composition of operators, $ U_{t_1} U_{t_2} = U_{t_1+t_2} $, with $ U_0 = I $ acting as the identity element of the group. These operators are unitary on $ L^2( M, \mu ) $ with $ U^*_t = U^{-1}_t = U_{-t} $, and therefore they induce an isometry, $ \lVert U_t f \rVert = \lVert f \rVert $, for all $ t \in \mathbb{ R } $. Moreover, the map $ t \mapsto U_t $ is continuous in the strong operator topology. Besides $ U_t : L^2( M, \mu ) \mapsto L^2( M, \mu ) $ analogous isometric Koopman operators can be defined on other $ L^p $ spaces with $ p \in [ 1, \infty ] $, and all Koopman operators on $ L^p( M, \mu ) $ extend to the Koopman operator on $ L^1( M, \mu ) $. Unless otherwise stated, in what follows we will work in the $L^2 $ setting. 

By Stone's theorem, the generator of the Koopman group on $ L^2( M, \mu ) $, defined as $  \tilde v : f \mapsto \tilde v( f )= \lim_{t\to 0} ( U_t f - f ) / t  $ whenever that limit exists, is an unbounded, skew-adjoint operator, $ \tilde v^* = - \tilde v $, with a dense domain $ D(\tilde v) \subset L^2(M,\mu) $. This operator is the maximal skew-adjoint extension of the smooth vector field $ v : C^\infty( M ) \mapsto L^2( M, \mu ) $ with $  v(f) = \lim_{t\to0} ( f \circ \Phi_t - f ) / t $, giving the directional derivative of functions along the dynamical flow. For every observable $ f \in D( \tilde v ) $ the map $ t \mapsto U_t f $ is differentiable at all $ t \in \mathbb{ R } $, and $ U_t f $ is governed by the evolution equation
\begin{equation}
  \label{eqUEvolve}
  \frac{ d\ }{ dt } U_t f = \tilde v U_t f = U_t \tilde v f.
\end{equation} 
Moreover, by the spectral theorem for skew-adjoint operators, for any $  t \in \mathbb{ R} $, we have $ U_t = e^{t \tilde v } $.

An important property of the dynamical vector field $ v $ is the Leibniz rule,
\begin{equation}
  \label{eqLeibniz}
  v( f_1 f_2 ) = v( f_1 ) f_2 + f_1 v( f_2 ), \quad \forall f_1, f_2 \in C^\infty( M ).
\end{equation}
Moreover, because $ \Phi_t $ preserves $ \mu $, $ v $ has the Liouville property, $ \divr_\mu v = 0 $, i.e., it generates an incompressible flow on $M $ with respect to the invariant measure.  In data space, $ v $ corresponds to a vector field $ V = F_* v \in T \mathbb{ R }^d $, and because of the canonical isomorphism $ T \mathbb{ R }^d \simeq \mathbb{ R }^d $, we can intuitively think of $ V $ as a collection of Euclidean vectors (``arrows'') tangent to the data manifold $ F( M ) $ along the direction of the dynamical flow. Note that $ V $ can be approximated at $ O( T^p ) $ accuracy by taking finite differences of the time-ordered snapshots $ \{ x_i \} $ \cite{Giannakis15}; e.g., $ ( x_{i+1} - x_{i-1} ) / ( 2 T ) $ is an $ O( T^2 ) $ approximation  of $ V $ at $ x_i $.  In Section~\ref{secVectorField}, we will represent the pushforward map $ F_* $ in a smooth orthonormal basis of $ L^2( M, \mu ) $, allowing us to reconstruct arbitrary vector fields on $ M $, and in particular, a vector field decomposition of $ v $. Recently, it was shown that an analogous property to the Leibniz rule in~\eqref{eqLeibniz} also holds for the generator. In particular, it follows from \cite{TerEstLemanczyk17}, Theorem~1.1,  that $ \tilde v $ acts as a derivation on $ \mathcal{ A } := D( \tilde v ) \cap L^\infty( M, \mu ) $ (which is a dense, invariant subalgebra of $L^2(M,\mu) $); i.e.,\ $ \tilde v( f_1 f_2 ) = \tilde v( f_1 ) f_2 + f_1 \tilde v( f_2 ) $ holds for any $ f_1, f_2 \in \mathcal{A}$. A related property of Koopman operators is that $ U_t( f_1 f_2) = U_t( f_1 ) U_t(f_2) $ for any $ f_1, f_2 \in L^2(M,\mu ) $, where $ U_t $ on the left-hand side is understood as an operator on $ L^1(M,\mu) $.  Thus, $ U_t $ acts distributively on products of $ L^2 $ functions.

Let now $ \mathcal{ M } $ be the space of complex-valued measures on $ \mathcal{ B } $ with densities in $ L^2( M,\mu ) $. The Perron-Frobenius operator,  $ \tilde U_t : L^2( M, \mu ) \mapsto L^2(M,\mu ) $, characterizes the action of the flow on the densities of measures in $ \mathcal{ M } $; that is, given $ \nu \in \mathcal{ M } $ with $ d \nu / d\mu = \rho $, then $ \rho_t  = \tilde U_t \rho $ is equal to the density  $ d( \Phi_{t*}\nu ) / d\mu $. In the $L^2$ setting of interest here, the Koopman and Perron-Frobenius operators form an adjoint pair, $ \tilde U_t = U^*_t = e^{-t \tilde v}$, so we can characterize both through the generator $ \tilde v $.

Next, consider the eigenvalue equation for the generator,
\begin{equation}
  \label{eqEigV}
  \tilde v( z ) = \lambda z, \quad \lambda \in \mathbb{ C }, \quad z \in D( \tilde v ).
\end{equation}
Because $ \tilde v $ is skew-adjoint on $ L^2( M, \mu ) $, for any solution to~\eqref{eqEigV} we  have $ \lambda = \ii \omega $, where $ \omega $ is a real number that corresponds to an intrinsic (i.e., independent of observation modality) frequency of the dynamical system. Thus, the set of eigenvalues of $ \tilde v $ lies on the imaginary line, and if $ \lambda $ is the eigenvalue corresponding to the eigenfunction $ z $ then $ \lambda^* = - \lambda $ is also an eigenvalue corresponding to the eigenfunction $ z^* $. Because $ \tilde v $ annihilates constant functions, \eqref{eqEigV} always has the solution $ \lambda = 0 $, $ z = 1 $, where $ 1 $ denotes here a constant function equal to one  $ \mu $-a.e. It also follows from the skew-adjointness of $ \tilde v $ that eigenfunctions corresponding to distinct eigenvalues are orthogonal on $ L^2( M , \mu ) $. Another important property of $ ( \lambda, z ) $, which is a consequence of $ U_t = e^{t\tilde v} $, is that $ z $ is an eigenfunction of $ U_t $ at eigenvalue $ e^{\lambda t } = e^{\ii\omega t } $; that is, for $ \mu $-a.e.\ $ a \in M $,
  \begin{equation}
    \label{eqZSHO}
    z( \Phi_t( a ) ) = e^{i\omega t } z( a )
  \end{equation}
  In other words, along a typical  orbit of the dynamics passing through $ a $, the values $ z( \Phi_t( a ) ) $ of the observable $ z $ evolve like the state of simple harmonic oscillator of frequency $ \omega $ subject to the initial condition $ z(a) $.

In practical applications, the data is acquired at the finite sampling interval $ T $, and we can consider the discrete-time system $ ( M, \mathcal{ B }, \mu, \hat \Phi_n ) $, where $ \hat \Phi_n = \Phi_{nT} $ and $ n \in \mathbb{ Z } $. In the discrete case, the Koopman group $ \{ \hat U_n \}_{n \in \mathbb{Z}} $ is generated by $  U_T $ so that $ \hat U_n = U_T^n $. Hereafter, we will assume that this discrete-time system is ergodic. It then follows by Birkhoff's pointwise ergodic theorem that for $ \mu $-a.e.\ $ a \in M $ and for every $ f \in L^1( M, \mu ) $ the temporal means, $ \bar f_t( a ) = \int_0^t f( \Phi_{\tau} a ) \, d\tau / t $ and $ \bar f_N( a ) = \sum_{i=0}^{N-1} f( \hat \Phi_i a ) / N $ in the continuous- and discrete-time cases, respectively, both converge  to the ergodic average $ \int_M f \, d\mu $. Under these assumptions, inner products on $ L^2( M, \mu ) $ and other related Sobolev spaces can be approximated by Monte Carlo sums of sampled time series; i.e., given the time series $ \{ f_{1i} \}_{i=0}^{N-1} $ and $ \{ f_{2i} \}_{i=0}^{n-1} $ with $ f_{ji} = f_j( \hat \Phi_i a ) $ and $ f_j \in L^2( M, \mu ) $, then for $ \mu $-a.e.\ $ a \in M $ we have 
\begin{equation}
  \label{eqMonteCarlo}
  \lim_{N\to\infty} \frac{ 1 }{ N } \sum_{i=0}^{N-1} f^*_{1i} f_{2i} = \int_M f^*_1 f_2 \, d\mu = \langle f_1, f_2 \rangle.
\end{equation} 
This property makes the data-driven techniques in Sections~\ref{secGalerkinImplementation} and~\ref{secTimeChange} ahead and in \cite{BerryEtAl15} feasible. Note that if the dynamical system is not ergodic but preserves $ \mu $, the Monte Carlo sum in~\eqref{eqMonteCarlo} will converge to the $ L^2 $ inner product over the ergodic component associated with the point $ a $, and the techniques developed below can be applied for that component provided that the assumptions laid out in the beginning of Section~\ref{secErgodicity} are satisfied.

An important spectral implication of ergodicity (which is, in fact, equivalent to the usual measure-theoretic definition), is that the unit eigenvalue of $ U_t : L^1(M,\mu) \mapsto L^1(M,\mu) $ is simple, and the corresponding eigenspace is spanned by the function equal $ \mu $-a.e.\ to 1. As a result, given any eigenfunction $ z $ of $ U_t $ acting on $ L^2(M,\mu) $, we have $\lvert z \rvert^2 \in L^1( M,\mu ) $, and $ U_t( \lvert z^2 \rvert ) = U_t( z ) U_t(z^*) = \lvert z \rvert^2 $,  which implies that $ \lvert z \rvert^2 $ is an eigenfunction at eigenvalue 1 and therefore $ \mu $-a.e.\ constant. Thus, in the presence of ergodicity, every eigenfunction $ z $ of the Koopman operator on $ L^2 $ lies in $ L^\infty $, and therefore in the invariant subalgebra $ \mathcal{ A }$. In particular, $ z $ can be normalized to take values on the unit circle on the complex plane, and because  of~\eqref{eqZSHO},  for $ \mu $-a.e.\ $ a \in M $,
\begin{equation}
  \label{eqZFactor}
  z \circ \Phi_t(a)  = \Xi_t \circ z( a ), 
\end{equation}
where $ \Xi_t : S^1 \mapsto S^1 $ is the circle rotation with frequency $ \omega $. In other words, the flow  $ \Phi_t $ is metrically semiconjugate to the rotation on the unit circle with frequency $ \omega $, with $ z $ acting as a semiconjugacy map.  If, in addition, $ z $ is continuous, the semiconjugacy is also topological. The use of such factor maps in data-driven techniques was originally proposed in \cite{MezicBanaszuk04,Mezic05}.  

These results can also be stated in terms of the generator. In particular, ergodicity implies that $ \tilde v $ has a simple eigenvalue $ \lambda = 0 $ corresponding to the constant eigenfunction. Moreover, since in this case all Koopman eigenfunctions lie in $ \mathcal{ A } $, it follows from the Leibniz property  that if $ z_1 $ and $ z_2 $ are eigenfunctions corresponding to the eigenvalues $ \lambda_1 $ and $ \lambda_2 $, respectively, then $ z_1 z_2 $ is also an eigenfunction corresponding to the eigenvalue $ \lambda_1 + \lambda_2 $. Thus, the  eigenfunctions and eigenvalues of $ \tilde v $ have a group structure under multiplication and addition, respectively. Let now $ \mathcal{ D } $ be the closed subspace of $ L^2(M,\mu) $ spanned by Koopman eigenfunctions, and note that $ \mathcal{ D } $ is $ U_t $-invariant and necessarily infinite-dimensional if it contains nonconstant functions. As we will see below, the group structure of Koopman eigenfunctions is highly beneficial, for it allows one to generate an orthonormal basis of $ \mathcal{ D} $ from a finite number of eigenfunctions corresponding to rationally independent eigenvalues. In contrast, the spectra of the diffusion operators used traditionally for geometrical analysis of data do not have this group structure, though on compact manifolds these operators have pure point spectra with the associated complete orthonormal eigenfunctions (which, as stated in Section~\ref{secBackground}, is not necessarily the case for Koopman operators).

\begin{rk}[Continuous vs.\ discrete time]\label{remTime}
   While some of the constructions discussed below can also be be made using the discrete-time generator $ U_T $ instead of  $\tilde v $ (replacing the eigenvalue $ \lambda $ in~\eqref{eqEigV} by $ e^{\lambda T} $), the skew-adjointness of $ \tilde v $, which is a consequence of the continuity structure of the Koopman group $ \{ U_t \}_{t\in\mathbb{R}} $, provides additional properties such as a the Leibniz rule, the ability to reconstruct in data space via the pushforward map $ F_* $, and a coercivity property of an associated regularized generator (see~\eqref{eqCoercivity} ahead). Thus, in what follows we carry out our analysis in the continuous-time setting, and use finite differences for numerical approximation. Similarly, the time-change techniques of Section~\ref{secTimeChange} are also more naturally formulated in the continuous-time setting and implemented with finite-difference approximations.    
\end{rk}

It follows from the spectral theorem for skew-adjoint operators that in mixing systems the spectrum of $ \tilde v $ consists of an eigenvalue at zero and a nonempty continuous spectrum (but empty residual spectrum) \cite{KatokThouvenot06}. This observation, as well as the existence of many known ergodic dynamical systems with exotic spectral behavior (e.g., \cite{KornfeldSinai00,KatokThouvenot06}) raises concerns about the suitability of data-driven eigendecomposition techniques involving the generator in systems of high complexity. In Section~\ref{secTimeChange}, we will present a regularization scheme based on the theory of time change in dynamical systems that attempts to transform $ \tilde v $ to a generator of a dynamical system which is more amenable to eigendecomposition via numerical methods. For now, however, we restrict attention to systems with pure point spectra where the spectral properties of $ \tilde v $ are ``optimal'' for both dimension reduction and nonparametric forecasting. 

\begin{defn}[\label{defnPurePoint}Pure point spectrum] A dynamical system $ ( M, \mathcal{ B }, \mu, \Phi_t ) $ is said to have pure point spectrum if there exists an orthonormal basis of $ L^2( M, \mu ) $ consisting of eigenfunctions of its generator $ \tilde v $. We say that the spectrum is generated by $ l $ basic frequencies if there exist $ l $ rationally independent real numbers $ \{ \Omega_i \}_{i=1}^l $ such that the eigenvalues in~\eqref{eqEigV} can be expressed as $ \lambda_k = \ii \sum_{i=1}^l k_i \Omega_i $, with $ k = ( k_1, k_2, \ldots, k_l ) \in \mathbb{ Z }^l $. Denoting the eigenfunction corresponding to the eigenvalue $ \ii \Omega_i $ by $ \zeta_i $, where we take $ \lVert \zeta_i \rVert= 1 $ by convention, the eigenfunction corresponding to eigenvalue $ \lambda_k $ is given by $ z_k = \prod_{j=1}^l \zeta_i^{k_i} $.
\end{defn}

\begin{rk}[Non-isolated eigenvalues]
  \label{remEigDense}The set of frequencies $ \{ \sum_{i=1}^l k_i \Omega_i \mid k_1, \ldots, k_l \in \mathbb{ Z } \} $ is countable and therefore has zero Lebesgue measure on $ \mathbb{ R } $. However, for $ l \geq 2 $, and by rational independence of the $ \Omega_i $,  the set of frequencies is dense in $ \mathbb{ R } $; i.e., $ \tilde v $ has no isolated eigenvalues.  Moreover, the basic frequencies  $ \{ \Omega_i \}_{i=1}^m $ are non-unique as there exist linear combinations of the $ \Omega_i $ with nonzero integer coefficients that are also rationally independent. This fact raises the questions of how to select appropriate generators of the spectrum, and how to ensure the good conditioning of approximate generators obtained via finite-dimensional numerical algorithms. We will address these issues in Sections~\ref{secDimensionReduction} and~\ref{secGalerkin}, respectively.       
\end{rk}

An important property of systems with pure point spectra is the existence of a unitary Fourier operator $ \mathcal{ U } : L^2( M,\mu) \mapsto \ell^2 $, defined as $ \mathcal{ U } f = \hat f = ( \hat f_k )_k $, where $ \hat f_k = \langle z_k, f \rangle $. That is, $ \mathcal{ U } $ maps $ f $ to its corresponding sequence of expansion coefficients in the Koopman eigenfunction basis. Applying this operator pointwise to the vector-valued observation map $  F $ leads to the decomposition 
\begin{equation}
  \label{eqFZDecomp}\
  F = \mathcal{U}^{-1} \hat F, \quad \hat F = \mathcal{U} F = ( \hat F_k )_k, \quad F_k  = \langle z_k, F \rangle,
\end{equation}
where the Fourier coefficients $ \hat F_k $ are now spatial patterns in $ \mathbb{ R }^d $. In particular, these patterns are the Koopman modes introduced in \cite{MezicBanaszuk04,Mezic05}. In Section~\ref{secVectorField}, we will see that Koopman modes also arise naturally in the spectral representation of the pushforward map $ F_* $ for vector fields. 

A key property of the Fourier operator is that it transforms the Koopman group and its generator into multiplication operators. Specifically, we have $ \mathcal{ U } U_t \mathcal{ U }^{-1}  = T_{e^{\ii\omega t}} $ and and $ \mathcal{ U } \tilde v \mathcal{ U }^{-1}  = T_{\ii \omega} $, where $ T_{e^{\ii\omega t}} : \ell^2 \mapsto \ell^2 $ is the bounded, unitary multiplication operator by $ e^{\ii\omega t} = ( e^{\ii \omega_k t } )_k $, and $ T_{\ii \omega} : D( T_{ \ii \omega} ) \mapsto \ell^2 $ is the skew-adjoint, unbounded multiplication operator by $ \ii \omega = ( \ii \omega_k )_k $ with the (dense) domain $ D( T_{\ii \omega }) = \{ ( \hat f_k )_k \in \ell^2 \mid ( \omega_k \hat f_k )_k \in \ell^2 \} $. In the Fourier representation, the evolution equation~\eqref{eqUEvolve} becomes
\begin{equation}
  \label{eqUEvolveFourier}
  \frac{ d\ }{ dt } \hat f_t = T_{\ii \omega} \hat f_t, \quad \hat f_t = \mathcal{ U } U_t f = ( \hat f_k(t) )_k.
\end{equation}
That is, the Fourier coefficients evolve as uncoupled simple harmonic oscillators, $ \frac{ d \hat f_k( t ) }{ dt } = \ii \omega_k \hat f_k( t ) $, whose solutions $ \hat f_k( t ) = \hat e^{\ii \omega_k t } f_k( 0 )  $ express the fact that $ \hat f_t = T_{e^{\ii \omega t} } \hat f_0 $. Note that by virtue of the group structure of Koopman eigenfunctions, one needs to have access to only $ l $ generating eigenfunctions and the corresponding basic frequencies in order to evaluate the action of $ \mathcal{ U } $, $ T_{\ii \omega} $, and $ T_{e^{\ii\omega t}} $.  The group structure of the eigenfunctions also leads to a useful convolution relationship giving the Fourier coefficients of products of functions (which we do not quote here in the interest of brevity).

A classical result in ergodic theory \cite{HasselblattKatok02} states that ergodic dynamical systems with pure point spectra are metrically isomorphic (though not necessarily homeomorphic) to translations on compact Abelian groups equipped with the Haar measure. In particular, for pure point systems realized through diffeomorphisms of a smooth $m $-dimensional manifold $M$ (which is the case studied here) with  $ C^\infty $ Koopman eigenfunctions, the number $ l $ of basic frequencies is necessarily equal to $ m $, and $ M$ is diffeomorphic to the $ m $-torus. In that case, there exist canonical angle coordinates $ \theta = ( \theta^1, \ldots, \theta^m ) $ on $ M $ with $ \theta_i \in [ 0, 2 \pi ) $ such that $ \zeta_i( \theta ) = e^{\ii \theta^i} $. Pure-point  spectrum systems with arbitrarily many basic frequencies can be constructed from smooth, ergodic diffeomorphisms of any manifold supporting a periodic flow \cite{AnosovKatok70},  but these systems have discontinuous Koopman eigenfunctions. To our knowledge, no such diffeomorphism has been realized through a continuous-time flow.

Throughout this paper, the case with $ M \simeq \mathbb{ T }^m $  and $ l= m $ smooth Koopman eigenfunctions will be the canonical setting where we develop our dimension reduction and forecasting techniques for pure point spectrum systems, as well as for systems with more general spectral properties. Note that our smoothness assumption on Koopman eigenfunctions is no stronger than assuming that they are all continuous, for every eigenfunction $ z_k $ lies in the domain of $ \tilde v^n $ for every $ n \in \mathbb{ N } $, which implies that if it is continuous it has continuous derivatives at every order. Nevertheless, some of our methods, in particular, the nonparametric prediction scheme in Section~\ref{secForecasting} and the Galerkin method for Koopman eigenvalues and eigenfunctions in Section~\ref{secGalerkin}, are also applicable in systems with eigenfunctions of weaker, $ H_1 $, regularity.

\section{\label{secPurePointSpectra}Dimension reduction and forecasting in systems with pure point spectra}

\subsection{\label{secDimensionReduction}Intrinsic dimension reduction coordinates}

Our approach for systems with pure point spectra and smooth Koopman eigenfunctions is to construct a family of nonlinear projection maps $ \{ \pi_i \}_{i=1}^m $ of the phase space manifold $ M $ (or, equivalently, its diffeomorphic copy $F( M ) $ in data space) to the complex plane, using the generating eigenfunctions $ \{ \zeta_i \}_{i=1}^m $ from Definition~\ref{defnPurePoint} as dimension-reduction coordinates. Specifically, we set $ \pi_i : M \mapsto \mathbb{ C } $ with 
\begin{equation}
  \label{eqPiProj}
    \text{$\pi_i = \zeta_i$, $\mu$-a.e.}, \quad \text{and} \quad \pi_i(M) = S^1 .
\end{equation}
We also consider the composite map $ \pi : M \mapsto \mathbb{ C }^m $ with $ \pi( a ) = ( \pi_1( a ), \ldots, \pi_m( a ) ) $. Note that we can choose the image of $ M $ under $ \pi_i $ to lie in the unit circle by ergodicity. Moreover, since the eigenfrequencies $ \Omega_i $ are rationally independent, $ \pi(M ) = \mathbb{ T }^m $. It is straightforward to verify that   $ M \mapsto \pi_i(M) $ is a $ C^\infty $ manifold submersion of $ M $ onto the unit circle, and $ M \mapsto \pi(M) $ a diffeomorphism of the $m$-torus.

The projection maps in~\eqref{eqPiProj} follow the widely adopted paradigm of applied harmonic analysis and machine learning, which is to perform dimension reduction of data on nonlinear manifolds using eigenfunctions of linear operators on function spaces on these manifolds \cite{CoifmanLafon06,BelkinNiyogi03,VonLuxburgEtAl04,Singer06,BelkinNiyogi07,VonLuxburgEtAl08,BerryHarlim16,BerrySauer16,HeinEtAl05,JonesEtAl08,GiannakisMajda12a,BerryEtAl13,TalmonCoifman13,Portegies14,Giannakis15,BerrySauer16b,DsilvaEtAl16,YairEtAl16}. While these methods are typically based on eigenfunctions of diffusion operators or heat kernels, in this case we use eigenfunctions of a skew-adjoint operator intrinsic to the dynamical system generating the data, namely the generator of the Koopman group. We will formulate algorithms for approximating these eigenfunctions from data in Section~\ref{secGalerkin}. For now, we discuss the main features of dimension reduction with Koopman eigenfunctions, which follow from basic results of ergodic theory and maps of manifolds. In what follows, we use the notation $ \{ \theta^1, \ldots, \theta^m \} $ for general local coordinates defined in a neighborhood of $ a \in M $. Moreover, we denote the corresponding coordinate basis vectors of $ T_a M $ and their duals by $ \{ \frac{ \partial \; }{ \partial \theta^1 }, \ldots, \frac{ \partial \ }{ \partial \theta^m } \} $ and $ \{ d\theta^1, \ldots, d\theta^m \} $, respectively, where $ d\theta^i( \frac{ \partial \  }{ \partial \theta^j } ) = \delta_{ij} $. 

First,  the sets of eigenfrequencies and eigenfunctions of $ \tilde v $  do not depend on the observation map $ F $ and its associated Riemannian metric, but as stated in Remark~\ref{remEigDense}, these sets can be generated by non-unique choices of basic frequencies and their corresponding eigenfunctions. Nevertheless, given any two observation maps, it is possible to find generating eigenfunctions such that the corresponding projection maps from~\eqref{eqPiProj} are consistent.  This means that the $ \pi_i $ provide a unified space to parameterize data acquired from different sensors (different choices of $ F $), but generated by the same dynamical system. This property has several data analysis applications (which we do not study here), such as fusion and inference of data acquired from different sensors and from potentially distinct states in $ M $.

 In the setting of a single observation modality, our basic criterion used throughout the paper will be to choose the generating eigenfunctions with the least oscillatory behavior in the Riemannian metric $ g $ associated with the observation map. A natural measure to quantify this behavior for a function $ f $ is the Dirichlet energy, $ E_g(f) = \lVert \grad_g f \rVert^2 $, where $ \grad_g $ is the gradient operator associated with $ g $ and $ \lVert \cdot \rVert $ the $ L^2 $ norm for vector fields. We will provide more precise definitions for the gradient operator and related Hilbert space notions in Section~\ref{secGalerkinPrelim},  but for now it suffices to note that the set of Dirichlet energies $ \{ E_g( z_k ) \} $ of the eigenfunctions is discrete in $ \mathbb{ R } $, and therefore the set  of eigenfunctions can be stably ordered in order of increasing $ E_g(z_k) $. (Actually, in Section~\ref{secGalerkinImplementation} we will work with a conformally transformed metric, $h$, with stronger invariance properties under changes of observation modality than $ g $, but for now we work with $ g $ to illustrate ideas.) Intuitively, we expect that functions with small Dirichlet energy can be more accurately approximated from finite datasets than highly oscillatory functions, so our selection criterion for the generating eigenfunctions is to set $ \{ \zeta_i \}_{i=1}^m $ to the first $ m $ eigenfunctions in that ordering corresponding to rationally independent frequencies. We will discuss the numerical implementation of this approach in Section~\ref{secGalerkin}. In general, the generating eigenfunctions selected in this way will not always agree among different observation modalities, but our selection criterion exhibits some rigidity in the sense that there exist equivalence classes of observation modalities for which the selected generating eigenfunctions and generating frequencies are the same.    

Further useful properties of the Koopman eigenfunctions pertain to the dynamics in the image space from~\eqref{eqPiProj}. In particular, it follows directly from~\eqref{eqZFactor} that the projected dynamics under $ \pi_i $ are simple harmonic oscillations (rotations) with frequencies $ \Omega_i $. That is, the time evolution of the data in the image spaces $ \pi_i( M ) $ can be described by means of autonomous, integrable dynamical systems, avoiding closure issues---a non-trivial property which is not satisfied by general dimension reduction maps. Similarly, the dynamics under the composite map $ \pi $ is an ergodic rotation on the $m$-torus, i.e., it is quasiperiodic. If $ \zeta_i $ is $ C^\infty $, then that $ \pi_i $ is a factor map can be expressed as a projectibility property of $ v $:    

\begin{prop} The dynamical vector field $ v $ is projectible under $ \pi_i $ in the sense that for any $ a, b \in M $ such that $ \pi_i( a ) = \pi_i( b ) $, the equality $  \pi_{i*} v \rvert_{T_{\pi_i( a )} \mathbb{ C } } = \pi_{i*} v \rvert_{ T_{\pi_i( b )} \mathbb{ C }  }  $ holds, where $  \pi_{i*} : T  M \mapsto T \mathbb{ C } $ is the derivative map of $ \pi_i $, mapping tangent vectors on $ M $ to tangent vectors on $ \mathbb{ C } $. 
\end{prop}
\begin{proof}
  Let $ \{ \hat e_1, \hat e_2 \} = \{ 1, \ii \} $ be the canonical basis of $ \mathbb{ C } $. This basis can be canonically identified with a basis of  $ T_{\pi_i(a)} \mathbb{ C } $ so that $ v_* = \pi_{i*} v \rvert_{\pi_i( a ) } = \sum_{j=1}^2 v_*^j \hat e_j $, where $ v_*^j = \sum_{k=1}^m v^k \frac{ \partial \chi^j }{ \partial \theta^k } = v( \chi^j ) $, with $ \chi^1 = \Real \zeta_i $ and $ \chi^2 = \Imag \zeta_i $. Because $ v( \zeta_i ) = \ii \Omega_i \zeta_i  $, we have  $ v_*^\nu \rvert_{T_{\pi_i( a )} \mathbb{ C } } = v_*^\nu \rvert_{T_{\pi_i( b )} \mathbb{ C }  } $ whenever $ \pi_i( a ) = \pi_i( b ) $. In particular, note that $ v( \chi^1 ) = - \Omega_i \chi^2 $ and $ v( \chi^2 ) = \Omega_i \chi^1 $.
\end{proof}

\subsection{\label{secVectorField}Vector field decomposition}

In this Section, we describe a decomposition of the dynamical vector field $ v $ into a sum of vector fields which are mutually commuting, and describe ``simpler'' (but non-ergodic) dynamics than $ v $ in that they have non-trivial nullspaces. We also discuss how these vector fields can be realized in data space through a spectral representation of the pushforward map, leading to spatiotemporal patterns that can be thought of as generalizations of Koopman modes \citep[][]{Mezic05}. Throughout this section, we assume that the generating Koopman eigenfunctions $ \{ \zeta_i \}_{i=1}^m $ are all smooth, so that $ M $ is diffeomorphic to the $ m $-torus, with the map $ \pi $ from Section~\ref{secDimensionReduction} providing a $ C^\infty $ diffeomorphism between $ M $ and $ \pi( M) = \mathbb{ T}^m $. Due to this relation, we can choose  $ \theta = ( \theta^1,\ldots, \theta^m ) $ to be canonical angle coordinates for $ \mathbb{ T }^m $ such that $ \zeta_i( \theta ) = e^{\ii \theta^i} $.  With this choice, the  $ \frac{ \partial\ }{ \partial \theta^i } $  and $ d\theta^i $ become the corresponding globally defined basis vector fields and dual vector fields, respectively. The following Theorem summarizes our vector field decomposition.  
    
\begin{thm}
  \label{lemmaVDecomp} Assume that the generating Koopman eigenfunctions $ \zeta_1, \ldots, \zeta_m $ are smooth, and let $ v_i $ be the smooth vector fields on $ M $ defined as $ v_i = \Omega_i \frac{ \partial\ }{ \partial \theta^i } $, where $ \Omega_i $ and $\theta^i $ are the eigenfrequency and canonical angle coordinate associated with $ \zeta_i $, respectively.   

     (i) Every Koopman eigenfunction  $ z_k = \prod_{i=1}^m \zeta_i^{k_i} $, $ k = ( k_1, \ldots, k_m ) \in \mathbb{ Z }^m $, is also an eigenfunction of $ v_i $ at eigenvalue $ \omega^{(i)}_k = \ii k_i \Omega_i $. As a result, $ v_i $ has the nullspace $ \ker v_i = \bigotimes_{j \neq i } \mathcal{Z}_j $, where $ \{ \mathcal{ Z }_j \}_{j=1}^m $ are the orthogonal subrings of $ L^2( M, \mu ) $ generated by $ \zeta_j $.

     (ii) The $ v_i $ are nowhere-vanishing, linearly independent, and have vanishing commutator, $ [ v_i, v_j ] = v_i v_j - v_j v_i = 0 $.
     
     (iii) The flows $ \Phi_{i,t} : M \mapsto M $, $ t \in \mathbb{ R } $, generated by $ v_i $ preserve the invariant measure $ \mu $ of the full dynamics.

     (iv) The $ \Phi_{j,t} $ act on the $ v_i $ by translations, in the sense that the tangent vector $ u \in T_b M $ with $ u = \Phi_{j,t*}(  v_i \rvert_a ) $ and $ b = \Phi_{j,t}( a ) $ acts on $ f \in  C^\infty( M)  $ according to $ u( f ) = v_i \rvert_b( f )$ for every $ a \in M $, where $ \Phi_{j,t*} : TM \mapsto TM $ is the derivative map associated with $ \Phi_{j,t} $.
  
      (v) The decomposition $ v = \sum_{i=1}^m v_i $ holds.
\end{thm} 
\begin{proof} 
  (i) By definition of the $ \theta^i $ coordinates, we have $ d \zeta_i = \ii \zeta_i \, d\theta^i $. Therefore,  $ v_i( \zeta_j ) = d\zeta_j( v_i )  = \ii \Omega_j \zeta_j \, d \theta^j (\frac{ \partial\ }{ \partial \theta^i }) = \ii \Omega_j \zeta_j \delta_{ji} $. This relation in conjunction with the Leibniz rule in~\eqref{eqLeibniz} proves the claim. 

(ii) The claim follows immediately from the fact that the $ \frac{ \partial\ }{ \partial \theta^i } $ are nowhere-vanishing, linearly independent, and mutually commuting vector fields.

(iii)  A necessary and sufficient condition that the $ \Phi_{i,t} $  are measure preserving is that  $ \divr_\mu v_i $ vanishes everywhere on $ M $. This follows from the definition of these vector fields and the fact that $ \divr_\mu \frac{ \partial\ }{ \partial \theta^i } = 0 $.

(iv) Since $ \{z_k \}  $ is a smooth, orthonormal basis of $ L^2(M,\mu) $, it  suffices to show that $ u( z_k )  = v_i\rvert_b( z_k) $. Indeed, according to~\eqref{eqZSHO}, 
  \begin{displaymath}
   u( z_k ) = ( \Phi_{j,t*}(  v_i \rvert_a ) )( z_k ) = v_i \rvert_a( z_k \circ \Phi_{j,t})  = e^{\ii \Omega_j t k_j } v_i \rvert_a  ( z_k ) = e^{\ii \Omega_j t k_j } \ii \Omega_i k_i z_k( a ) = \ii \Omega_i k_i z_k( b ) = v_i \rvert_b( z_k ).
  \end{displaymath}

(v) Since $ \{ \frac{ \partial\ }{ \partial \theta^i  } \rvert_a \}_{i=1}^m $ is a basis of $ T_a M $ at every $ a \in M $, there exist smooth functions $ C_1, \ldots, C_m $ such that $ v = \sum_{i=1}^m C_i v_i $. These functions are given by $ C_i = d\theta^i( v ) = d\zeta_i( v ) / ( i \zeta_i ) = v( \zeta_i ) / ( i \zeta_i ) = \Omega_i $, leading to the desired result.
\end{proof}

The vanishing commutator of the $ v_i $ in (ii) is an intrinsic (observation map independent) dynamical independence property; that is, in (v) the dynamical vector field is decomposed into independent components.  This decomposition   has connections with the nonlinear independent component analysis technique of Singer and Coifman \cite{SingerCoifman08}, which recovers independent components of stochastic differential equations on manifolds using kernel methods. Globally on $ M $, the $ v_i $ generate measure-preserving, but non-ergodic, transformations $ \Phi_{i,t} $, giving the full evolution map through the composition $ \Phi_t = \Phi_{1,t} \circ \cdots \circ \Phi_{m,t} $, where the order of the components does not matter. As is the case with $ \Phi_t $, the projection maps $ \pi_i $ in~\eqref{eqPiProj} are factor maps mapping the dynamics on $ M $ associated with $ \Phi_{i,t} $ to the corresponding circle rotation $ \Xi_{i,t} $ with frequency $ \Omega_i $ as in~\eqref{eqZFactor}; i.e., $ \pi_i \circ \Phi_{i,t} = \Xi_{i,t} \circ \pi_i $. Moreover, associated with $ \Phi_{i,t} $ are are $C^0 $ unitary Koopman groups $ U_{i,t} : L^2( M, \mu ) \mapsto L^2( M, \mu ) $ generated by skew-adjoint operators $ \tilde v_i : D( \tilde v_i ) \mapsto L^2(M, \mu) $ that extend $ v_i $. Note that the domains $ D( \tilde v_1 ), \ldots, D( \tilde v_m ) $ and $ D( \tilde v ) $ are in general different. As with $ \tilde v $, the generators $ \tilde v_i $ are transformed into skew-adjoint multiplication operators $ T_{\ii \omega^{(i)} } =  \mathcal{ U } \tilde v_i \mathcal{ U }^{-1}  $ by the Fourier operator $ \mathcal{ U } $ introduced in Section~\ref{secErgodicity}, where $ T_{\ii \omega^{(i)}} $ denotes multiplication by $ \ii \omega^{(i)} = ( \ii \omega_k^{(i)} )_k $.      

The vector fields $ v_i $ are intrinsically defined as differential operators on $ M $,  and to reconstruct them in data space $ \mathbb{ R }^d $ we apply the pushforward map for tangent vectors introduced in Section~\ref{secErgodicity}. 
\begin{prop} \label{propPushforward} Let $ u $ be a smooth vector field on $ M $, and $ F $ a smooth embedding of $ M $ into $ \mathbb{ R }^d $. Then, the image $ u_* = F_*( u ) $ under the pushforward map $ F_* : T M \mapsto T \mathbb{ R }^d $ is given by $ u_* = u( F ) $, where $ u $ acts on $ F $ componentwise in a basis of $ \mathbb{ R }^d $.   
\end{prop}
 \begin{proof}
   Let $ \{ e_1, \dots, e_d \} $ be a basis of $ \mathbb{ R }^d $. This basis induces a basis of the tangent space $ T_x \mathbb{ R }^d $ at every $ x \in \mathbb{ R }^d $ through the canonical isomorphism $ T_x \mathbb{ R }^d \simeq \mathbb{ R }^d $, and we have $ u_* = \sum_{j=1}^d u_*^j e_j $. Denoting the components of $ u $ in local coordinates $ \{ \theta^1, \ldots, \theta^m \} $ on $ M $ by $ \{ u^1, \ldots, u^m \} $, for  $ f \in C^\infty(M) $ we have $ u( f ) = \sum_{i=1}^m u^i \frac{ \partial f }{ \partial \theta^i } $. In particular, expanding the observation map as $ F = \sum_{j=1}^d F^j e_j $, and using the transformation law for tangent vectors, we obtain $ u^j_* = \sum_{i=1}^m \frac{ \partial F^j }{ \partial \theta^i } u^i = u( F^j ) $ and therefore $ u_* = u( F ) $.
\end{proof}
Using Proposition~\ref{propPushforward} and the fact that $ v_i( F ) = \tilde v_i( F ) = \mathcal{ U }^{-1} T_{\ii \omega^{(i)} }  \hat F $,  where $ \hat F $ is the Fourier representation of the observation map from~\eqref{eqFZDecomp}, we can compute the reconstructed vector fields $ V_i = F_* v_i $ in $ \mathbb{R}^d $ through the expression 
\begin{equation}
  \label{eqVPushZeta}
  V_i =  \tilde v_i( F ) = \sum_k \ii k_i \Omega_i \hat F_k z_k,
\end{equation}
which holds $ \mu $-a.e. Moreover, it follows from Theorem~\ref{lemmaVDecomp}(iv) that $ V = F_* v = \sum_{i=1}^m V_i $. Note that $ V $ can also be approximated from finite differences of time-ordered data as described in Section~\ref{secErgodicity} and \cite{Giannakis15}, but the components $ V_i $ in general cannot. In applications, the $ V_i $ can be visualized as spatiotemporal patterns (movies), or if $ d $ is sufficiently small, as arrow plots on the data manifold $ F( M ) $. In particular, because $ F_* $ is the pushforward map for vector fields, the vector $ V_i \rvert_a = \sum_k \ii k_i \Omega_i \hat F_k z_k( a ) $ is guaranteed (modulo numerical errors) to be tangent to $ F( M ) $ at the point $ F( a ) $;  see Figs.~\ref{figIrrationalV},  \ref{figTorusV}, \ref{figMixing3TorusV}, and~\ref{figFixedPointTorusV} for examples. The following is a direct consequence of Proposition~\ref{propPushforward} and the definition of the $ V_i $ through~\eqref{eqVPushZeta}. 

\begin{cor} \label{corTransf}The reconstructed vector fields $ V_i $ transform naturally as type $ ( 1, 0 ) $ tensors. That is, given a map $ G : \mathbb{ R }^d \mapsto \mathbb{ R }^{\tilde d } $ such that $ \tilde F = G \circ F $ is an embedding of $ M $ into $ \mathbb{ R }^{\tilde d} $, the reconstructed vector fields $ \tilde V_i = \tilde F_* v_i $ in $ \mathbb{ R }^{\tilde d} $ are related to the $ V_i $ via $ \tilde V_i = G_* V_i $. 
\end{cor} 

Note that the individual Koopman modes $ \hat F_k $ do not obey an analogous transformation law. That is, if $ \hat{\tilde F}_k = \langle z_k, \tilde F \rangle = \langle z_k, G \circ F \rangle $, then in general $ \hat{\tilde F}_k $ is not equal to $ G( \hat F_k )  $.
 
\subsection{\label{secForecasting}Nonparametric forecasting}

In systems with pure point spectra, the facts that Koopman eigenfunctions form a complete basis of $ L^2( M, \mu ) $ in which Koopman operators are represented by multiplication operators with an available closed-form expression for the multiplication function, $ e^{\ii \omega t } $, makes these eigenfunctions well suited for data-driven nonparametric forecasting. In this section, we formulate a technique for forecasting probability densities and expectation values of observables with initial data specified as a probability measure. Our approach follows closely the nonparametric framework developed in \cite{BerryEtAl15}, with the difference that here we use the additional structure in the temporal evolution of the eigenfunctions. 

Consider the initial data given as a probability measure $ \mu_0 $ on $ ( M, \mathcal{ B } ) $ with a smooth probability density $ \rho_0 $ relative to the invariant measure $ \mu $. The measure $ \mu_0 $ evolves at time $ t $ according to $ \mu_t = \Phi_{t*} \mu_0 $, and the corresponding density is given through the Perron-Frobenius operator by $ \rho_t = U^*_t \rho_0 $. In the Fourier representation introduced in Section~\ref{secErgodicity}, we have $ \rho_t = \mathcal{ U }^{-1} \hat \rho_t $ for all $ t \in \mathbb{ R }$, with $ \rho_t \in \ell^2 $ given by $ \hat \rho_t = T_{e^{-\ii \omega t} } \hat \rho_0 $. Explicitly, the forecast density $ \rho_t $ expanded in the Koopman eigenfunction basis becomes
\begin{equation}
  \label{eqL2Rho} \rho_t = \sum_k \hat \rho_k( t ) z_k, \quad \text{with} \quad \hat \rho_k( t ) = e^{-\ii \omega_k t } \hat \rho_k( 0 ), \quad \hat \rho_t = ( \hat \rho_k( t ) )_k.
\end{equation}

Next, to compute the expectation value $ \bar f_t = \mathbb{E}_{\mu_t} f $ of an observable $ f \in L^2( M, \mu ) $ with respect to $ \mu_t $, we make use of the fact that 
\begin{equation}
  \label{eqMeanForecast0}
  \mathbb{E}_{\mu_t} f = \int_M f \rho_t \, d\mu = \langle \rho_t, f \rangle = \langle \hat \rho_t, \hat f \rangle_{\ell^2}, \quad  \hat f = \mathcal{ U } f = ( \hat f_k )_k,
\end{equation}
where the last equality in the expression for $ \mathbb{E}_{\mu_t} f $ is a consequence of the unitarity of $ \mathcal{ U } $. We thus obtain
\begin{equation}
  \label{eqMeanForecast}
  \bar f_t =  \sum_k e^{\ii\omega_k t } \hat \rho^*_k( 0 ) \hat f_k.
\end{equation}
Assuming, further, that $ f \in L^4( M,\mu) $, we can evaluate the mean square forecast $ \mathbb{ E }_{\mu_t} f^2 $ using a similar approach (or the convolution identity mentioned in Section~\ref{secErgodicity}), leading to the variance forecast $ \sigma^2_t = \mathbb{ E }_{\mu_t} f^2 - \bar f_t^2 $ which is useful for uncertainty quantification. 

\begin{rk} 
  Forecasting with  Koopman eigenfunctions can also be performed with initial data given as a single observation $ y $ in the ambient data space $ \mathbb{ R }^d $. In this case, we first compute values $ \{ \hat \zeta_i \}_{i=1}^m  $ for the $ m $ generating eigenfunctions at the point $ a \in M $ with $ F( a ) = y $ using out-of-sample extension techniques for functions (e.g., \cite[][]{CoifmanLafon06b,RabinCoifman12}), and then evolve the initial values $ \{ \hat \zeta_i \} $  via~\eqref{eqZSHO}. We then determine the values of other eigenfunctions using their group structure, and reconstruct the observable through its expansion coefficients and the eigenfunction values at the desired lead time. Note that if $ F( M ) $ is the data manifold from an imperfect model with model error, then $ y $ may not lie on $ F( M ) $, but extended function values  can also be computed in this case. This approach is closely related to a kernel analog forecasting framework developed in \cite{ZhaoGiannakis16}. In numerical experiments not reported here, we have observed comparable skill with this method and the results of Section~\ref{secTorus}.       
\end{rk}

\subsection{\label{secIrrational}Irrational flow on the 2-torus}

We demonstrate the techniques presented in Sections~\ref{secDimensionReduction}--\ref{secForecasting} in an analytically solvable example involving an irrational flow on the 2-torus. Denoting the azimuthal and polar angles on the 2-torus by $ ( \theta^1, \theta^2 ) $, respectively, we consider the  dynamical vector field on $ M = \mathbb{ T }^2 $ given by 
\begin{equation}
  \label{eqTorusIrrational}
  v = \sum_{i=1}^2 v^{i} \frac{ \partial \ }{ \partial \theta^{i} }, \quad \text{with} \quad v^1 = 1, \quad   v^2  = \alpha, 
\end{equation}
where $ \alpha $ is a positive angular frequency parameter which is set to an irrational number to produce an ergodic flow. This dynamical system is observed through the observation map $ F : M \mapsto \mathbb{ R }^3 $ corresponding to the standard embedding of the 2-torus into three-dimensional Euclidean space, i.e., for the point $ a \in M $ with coordinates $ ( \theta^1, \theta^2 ) $, we have $ F( a ) = ( F^1( a ), F^2( a ), F^3( a ) ) = ( x^1, x^2, x^3 ) $, where 
\begin{equation}
  \label{eqTorusEmbedding}
  x^1 = ( 1 + R \cos \theta^2 ) \cos \theta^1, \quad x^2 = ( 1 + R \cos \theta^2 ) \sin\theta^1, \quad x^3 = \sin \theta^2, \quad R \in ( 0,  1 ).
\end{equation}
The volume form $ d\mu = d \theta^1 \wedge d \theta^2 / (2 \pi )^2 $ associated with the invariant measure of this system has uniform density relative to the Haar measure.  Moreover, the eigenvalue problem for $ \tilde v $ in~\eqref{eqEigV} has solutions $ \lambda_k = \ii( k_1 + k_2 \alpha ) $ and $ z_k( \theta^1, \theta^2 ) = e^{\ii( k_1 \theta^1 + k_2 \theta^2 )} $, with $ k = ( k_1, k_2 ) $, $ k_i \in \mathbb{ Z } $. 

To select generating eigenfunctions $ \{ \zeta_1, \zeta_2 \} $ with low roughness on the torus and their corresponding basic frequencies $ \{ \Omega_1, \Omega_2 \} $, we compute the Dirichlet energy of the eigenfunctions in the induced Riemannian geometry from the observation map as described in Section~\ref{secDimensionReduction}. For the embedding in~\eqref{eqTorusEmbedding}, the induced Riemannian metric has components $ g_{11} = ( 1 + R \cos\theta^2 )^2 $, $ g_{22} = R^2 $, and $ g_{12} = g_{21} = 0 $ in the $ \{ \theta^{i} \} $ coordinates, leading to the Dirichlet energy values $ E_g(z_k) = C_1 k_1^2 + C_2 k_2^2 $, where $ C_1 = 1 / ( 1 - R^2 )^{3/2} $ and $ C_2 = \alpha^2/ R^2 $. Thus, in this geometry, the least-rough eigenfunctions corresponding to rationally independent frequencies are those with $ k= ( 0, 1 ) $ and $ k = ( 0, 2 ) $, i.e., we have $ \Omega_1 = 1 $, $ \Omega_2 = \alpha $, $ \zeta_1( \theta^1, \theta^2 ) = e^{\ii\theta^1} $, and $ \zeta_2( \theta^1, \theta^2 ) = e^{\ii\theta^2} $. Clearly, the image of the torus under each of the projection maps from~\eqref{eqPiProj} is the unit circle, $ \pi_i( a ) = e^{\ii \Omega_i \theta^i } $, and the system evolves in these coordinates as a simple harmonic oscillator in accordance with~\eqref{eqZSHO}. 

\begin{rk}[Highly oscillatory eigenfunctions and slow observables] \label{remOscillatory} In this torus rotation, for every eigenfunction $ z_k $ corresponding to the eigenvalue $ \lambda_k $ there exist eigenfunctions with eigenvalues arbitrarily close to $ \lambda_k $ and with arbitrarily large Dirichlet energy. This is a consequence of the density of the spectrum of $\tilde v $ on the real line (see Remark~\ref{remEigDense}). In particular, the eigenvalue and Dirichlet energy of the constant eigenfunction are both zero, but for any $ \epsilon, \bar E > 0 $ one can find integers $ i $ and $ j $ such that for $ k = ( i, j ) $, $ \lvert \lambda_k \rvert < \epsilon $ and $ E_g({z_k}) > \bar E $. In other words, there exist observables $ z_k $ with arbitrarily small frequency $ \lvert \lambda_k \lvert $ but arbitrarily large roughness $ E_g( z_k ) $. This behavior is generic in systems possessing two or more rationally independent Koopman eigenvalues, and can adversely affect the conditioning of numerical schemes for Koopman eigenvalues and eigenfunctions. More generally, the identification of slow observables is an important task in reduced dynamical modeling (e.g., \cite{KevrekidisEtAl04,FroylandEtAl14,DsilvaEtAl16}), and the simple example discussed here suggests that care may be needed to ensure that the identified slow observables are also smooth. In Section~\ref{secGalerkin}, we will suppress the pathological Koopman eigenfunctions with large Dirichlet energy by adding a small amount of diffusion to $ v $.   
\end{rk}
 
Next, we reconstruct the vector fields $ v_i $ from Theorem~\ref{lemmaVDecomp} associated with the generating eigenfunctions $ \{ \zeta_1, \zeta_2\} $ using the spectral expansion of $ F  $ in the $ \{ z_k \} $ basis in accordance with~\eqref{eqVPushZeta}. Setting $ k = ( i, j ) \in \mathbb{ Z }^2 $, the expansion coefficients are
\begin{align*}
  \hat F^1_{ij} &= \langle z_{ij}, F^1 \rangle = \frac{ 1 }{ 2 } ( \delta_{i1} \delta_{j0} - \delta_{i,-1} \delta_{j0} ) + \frac{ R }{ 4 } ( \delta_{i1} \delta_{j1} + \delta_{i,-1} \delta_{j1} + \delta_{i1} \delta_{j,-1} + \delta_{i,-1} \delta_{j1} ), \\
  \hat F^2_{ij} &= \langle z_{ij}, F^2 \rangle = \frac{ 1 }{ 2 \ii }( \delta_{i0} \delta_{j1} - \delta_{i0 } \delta_{j,-1} )  + \frac{ R }{ 4 \ii }( \delta_{i1} \delta_{j1} - \delta_{i,-1} \delta_{j1} + \delta_{i1} \delta_{j,-1} - \delta_{i,-1} \delta_{j,-1} ), \\
  \hat F^3_{ij} &= \langle z_{ij}, F^3 \rangle = \frac{ R }{ 2\ii } ( \delta_{i0} \delta_{j1} - \delta_{i0} \delta_{j,-1} ),
\end{align*}
giving $ V_i = ( V^1_i, V^2_i, V^3_i ) $ with
\begin{displaymath}
  V_1^1  = - \sin \theta^1 - \frac{ R }{ 2 } ( \sin(\theta^1+\theta^2 ) + \sin( \theta^1 - \theta^2 ) ), \quad
  V_2^1 = \cos \theta^1 + \frac{ R }{ 2 } ( \cos( \theta^1 + \theta^2 ) + \cos( \theta^1 - \theta^2 ) ), \quad
  V_3^1 = 0,
\end{displaymath}
and
\begin{displaymath}
  V_2^1 = - \frac{ R }{ 2 } ( \sin( \theta^1 + \theta^2 ) + \sin( \theta^1 - \theta^2 ) ), \quad V^2_2 = \frac{ R }{ 2 }( \cos( \theta^1 + \theta^2 ) - \cos( \theta^1 - \theta^2 ) ), \quad V^3_2 = R \cos \theta^2.
\end{displaymath}
This decomposition is depicted in Fig.~\ref{figIrrationalV}.

\begin{figure}
  \centering\includegraphics{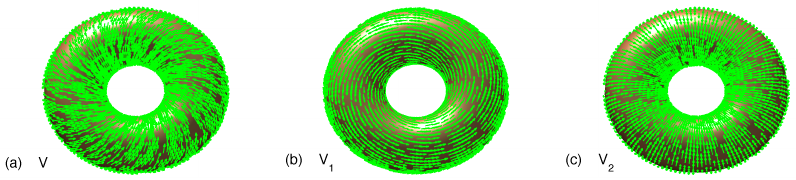}
  \caption{\label{figIrrationalV}Decomposition of the vector field of an irrational flow on the torus into mutually commuting components. (a) Full vector field $v$; (b,c) the components $v_i = \Omega_i \frac{ \partial\ }{ \partial \theta^i} $ from Theorem~\ref{lemmaVDecomp}, reconstructed in $ \mathbb{ R }^3 $ using the pushforward map from~\eqref{eqVPushZeta}.}
\end{figure}
  
Consider now statistical forecasting of the irrational-flow system using the nonparametric approach of Section~\ref{secForecasting}. In this example, we set the initial probability measure $ \mu_0 $ to a von Mises (circular Gaussian) distribution on the torus with the density function
\begin{equation}
  \label{eqVonMises}
  \rho_0( \theta^1, \theta^2 ) = e^{\kappa(\cos(\theta^1-\bar \theta^1) + \cos( \theta^2 - \bar \theta^2 ) ) } /  ( I_0( \kappa ) )^2
\end{equation}
relative to the invariant measure. In~\eqref{eqVonMises},  $ I_n $ is the modified Bessel function of order $ n $, and we use the values $ ( \bar \theta^1, \bar \theta^2 ) = ( 0, 0 ) $ and $ \kappa = 30 $ for the location and concentration parameters, respectively. We take the component $ F^1 $ of the observation map as the forecast observable $ f $, and compute the time-dependent expectation value and standard deviation of $ f $ using~\eqref{eqMeanForecast}. The latter equation can be evaluated analytically using properties of Bessel functions and the expansion of $ F^1 $ in the eigenfunction basis. In particular, using the result $ \int_0^{2\pi} e^{\ii n \theta + \kappa \cos \theta } \, d \theta = 2 \pi I_{\lvert n \rvert }( \kappa ) / I_0( \kappa ) $, we find
\begin{align*}
  \bar f_t &= \frac{ I_1( \kappa ) }{ I_0( \kappa ) } \cos t + \frac{ R I_1^2( \kappa ) }{ I_0^2( \kappa ) } ( \cos( ( 1 + \alpha ) t ) + \cos( ( 1 - \alpha ) t ) ),  \\
  \overline{ f^2 }_t &= \left( 1 + \frac{ I_2( \kappa ) }{ I_0( \kappa ) } \cos( 2 t ) \right) \left[ \frac{ 1 }{ 2 } + \frac{ R I_1( \kappa ) }{ I_0( \kappa ) } \cos( \alpha t ) + \frac{ R^2 }{ 4 } \left( 1 + \frac{ I_2( \kappa ) }{ I_0( \kappa ) } \cos( 2 \alpha t ) \right) \right].
\end{align*}
 The time evolution of $ \bar f_t $ and the standard deviation $ \sigma_t = ( \overline{f^2}_t - \bar f_t^2 )^{1/2} $ are shown in Fig~\ref{figIrrationalForecast}.

\begin{figure}
  \centering\includegraphics[scale=.95]{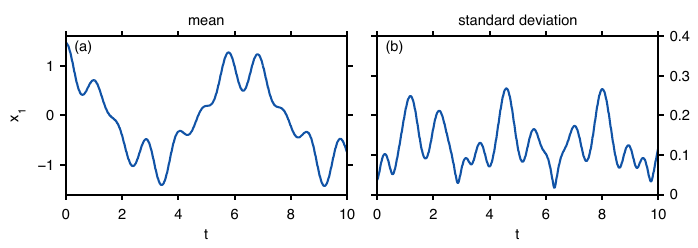}
  \caption{\label{figIrrationalForecast}Statistical forecast of the component $ x^1 $ of the torus embedding in~$ \mathbb{ R }^3 $ for the irrational flow  with frequencies $ (1, 30^{1/2} ) $. The initial probability measure has the circular Gaussian density from~\eqref{eqVonMises} relative to the Haar measure with location and concentration parameters $ ( 0, 0 ) $ and 30, respectively. (a) Mean forecast; (b) standard deviation.} 
\end{figure}

\section{\label{secGalerkinImplementation}Galerkin approximation in a data-driven orthonormal basis}

In this section, we present a Galerkin method with regularization for the eigenvalue problem of the generator in an orthonormal basis acquired through the diffusion maps algorithm \citep[][]{CoifmanLafon06}. We also discuss the spectral properties of the regularized generator, including its asymptotic behavior in the weak-diffusion limit---this discussion will also motivate the time-change techniques of Section~\ref{secTimeChange}. We then describe the implementation of the techniques of Section~\ref{secPurePointSpectra} in the eigenfunction basis of the regularized generator, and present numerical applications to a variable-speed ergodic flow on the 2-torus.
    
\subsection{\label{secGalerkinPrelim}Choice of Galerkin approximation space}

Let $ g $ be the Riemannian metric on $ M $ inherited from the observation map, and $ \sigma $ the $ C^\infty $ density of the invariant measure of the dynamics relative to the Riemannian measure of $ g $; i.e., $ \sigma = d \mu/ \dvol_g $, where $ \sigma $ is bounded away from zero by compactness of $ M $. For data generated by ergodic dynamical systems, $ \sigma $ is also  the sampling density relative to the Riemannian measure, though the samples collected from a single time series (as is the case in~\eqref{eqDataset}) are not independent. In general, $ \sigma $ will be a nonconstant function, and in what follows we work with the conformally transformed metric $ h = g \sigma^{2/m} $. In particular, we will use Laplace-Beltrami eigenfunctions associated with $ h$ as a basis of our approximation space for the Koopman generator.    Note that $ h $ contracts (expands) local distances with respect to the original metric $ g $ in regions of small (high) sampling density $ \sigma $. As we will see below, due to this property the Laplace-Beltrami eigenfunctions associated with $ h $ acquire increased ``resolution'' in high-$ \sigma $ regions, i.e., in regions where high resolution can be robustly attained using finite datasets.  In Section~\ref{secDataDrivenBasis}, we will discuss how to approximate the Laplace-Beltrami eigenfunctions associated with $ h$  via diffusion maps.
\begin{lemma}
  Let $ F : M \mapsto \mathbb{ R }^d $ and $ \tilde F : M \mapsto \mathbb{ R }^{\tilde d} $ be smooth embeddings of $ M $ with the corresponding induced Riemannian metrics $ g $ and $ \tilde g $, respectively. Assume that $ g $ and $ \tilde g $ are conformally equivalent, i.e., that there exists a positive function $ r \in C^\infty( M ) $ such that $ \tilde g = r g $ and $ 1/ r \in C^\infty(M) $. Then,  $ g \sigma^{2/m} = \tilde g \tilde \sigma^{2/m} $, where $ \sigma = d\mu/ \dvol_g $ and $ \tilde \sigma = d\mu/\dvol_{\tilde g} $ are the sampling densities associated with $ g $ and $ \tilde g $, respectively. 
\end{lemma}
\begin{proof}
  The claim follows immediately from the fact that $ \tilde \sigma = d\mu / \dvol_{\tilde g} = r^{-m/2} d\mu/ \dvol_g = r^{-m/2} \sigma. $ 
\end{proof}
\begin{cor}
  The metric $ h $ is unique for each equivalence class of observation maps associated with conformally equivalent induced metrics. Moreover, $ h $ has uniform volume form relative to the invariant measure of the dynamics since $ d\mu/ \dvol_h = (1/\sigma) \,  d\mu / \dvol_g =1 $. 
\end{cor}

Consider now the vector space of smooth vector fields on  $ M $,  equipped with the Hodge inner product $\langle u_1, u_2 \rangle = \int_M h( u_1, u_2 ) \dvol_h = \int_M h( u_1, u_2 ) \, d \mu $ associated with $ h $ and the norm $ \lVert u \rVert = \langle u, u \rangle^{1/2} $. We denote the gradient of a function $ f  \in C^\infty( M ) $ with respect to $ h $ by $ \grad_h f = h^{-1}( df, \cdot ) $, where $ h^{-1} $ is the ``inverse metric''. Note that $ \grad_h f = \sigma^{-2/m} \grad_g f $, and because $ d\mu/ \dvol_h  = 1 $,   $ \langle u, \grad_h f \rangle = - \langle \divr_\mu u, f \rangle $ for any smooth function $ f $ and vector field $ u$. In local coordinates, we have  $ \dvol_h = \sqrt{ \det h } \, d\theta^1 \wedge \cdots \wedge d\theta^m $, $ \grad_h = \sum_{i=1}^m ( \grad_h)^i \frac{ \partial\ }{ \partial \theta^i } $, and
\begin{equation}
  \label{eqGradDiv}
  ( \grad_h f )^i = \sum_{j=1}^m h^{-1, ij} \frac{ \partial f }{ \partial \theta^j }, \quad  \divr_\mu u  = \frac{ 1 }{ \sqrt{ \det h } } \sum_{i=1}^m \frac{ \partial \  }{ \partial \theta^i }\left(  \sqrt{ \det h }  u^i \right).
\end{equation}
We also introduce the order-1 Sobolev space $ H_1( M, h, \mu ) $ associated with $ h $ and $ \mu $  (henceforth abbreviated as $ H_1( M, h ) $ since $ \mu = \vol_h $), which is equipped with the inner product $ \langle f_1, f_2 \rangle_{H_1} = \langle f_1, f_2 \rangle + \langle \grad_h f_1,  \grad_h f_2 \rangle $ and the induced norm $ \lVert f \rVert_{H_1} = \langle f, f \rangle_{H_1}^{1/2} $. 

The Dirichlet energy of functions in $ H_1( M, h ) $ is given by the functional 
\begin{equation}
  \label{eqDirichlet}
  E_h(f) = \langle \grad_h f, \grad_h f \rangle = \int_M \lVert \grad_h f \rVert_h^2 \, d\mu = \int_M  \lVert \grad_g f \rVert_g^2 \sigma^{-2/m} \, d\mu,
\end{equation}
where $ \lVert \grad_h f \rVert_h^2 = h( \grad_h f, \grad_h f ) $ and $ \lVert \grad_g f \rVert_g^2 = g( \grad_g f, \grad_g f ) $. This functional provides a metric-dependent measure of roughness of functions which can be used to select generators for the spectrum of the Koopman group as described in Section~\ref{secDimensionReduction}. Note that due to the presence of the $ \sigma^{-2/m} $ term in the last integral in~\eqref{eqDirichlet}, functions with large gradient with respect to the ambient space metric $ g $ in regions of small sampling density will generally acquire large Dirichlet energy with respect to $ h $.       
  
Next, consider the Laplace-Beltrami operator $ \upDelta_h = - \divr_\mu \grad_h $ associated with the Riemannian metric $ h $, and a corresponding orthonormal basis $ \{ \phi_i \}_{i=0}^\infty $ of $L^2(M,\mu) $ consisting of  eigenfunctions
\begin{equation}
  \label{eqLapl}
  \upDelta_h \phi_i = \eta_i \phi_i, \quad \phi_i \in C^\infty( M ), \quad 0 = \eta_0 < \eta_1 \leq \eta_2 \leq \cdots \nearrow \infty, 
\end{equation}
corresponding to the eigenvalues $ \{ \eta_i \}_{i=0}^\infty $. Note that while the inner product of $ L^2( M, \mu ) $ is metric-tensor-independent, the gradient operator, and hence $ \upDelta_h $, $ \phi_i $, and $ \eta_i $, all depend on $ h $ (which depends in turn on the observation map $ F $). It is a standard result that the eigenvalues are extrema of the Rayleigh quotient $ R_h( f ) = E_h( f ) / \lVert f \rVert^2 $, and the corresponding eigenfunctions are the extremizers; i.e., $ \eta_i = R_h( \phi_i ) $,  and $R_h( \phi_i ) = E_h( \phi_i ) $ for normalized eigenfunctions. Since the integrals in the evaluation of $ R_h $  are with respect to the invariant measure of the dynamics, we interpret $ R_h $ as a measure of ``expected roughness'' with respect to the Riemannian metric $ h $. In particular,  due to the $ \sigma $-dependent term in~\eqref{eqDirichlet}, the leading extrema of $ R_h $ will generally correspond to eigenfunctions with weaker oscillatory behavior (as measured with respect to the ambient space metric) in regions of small sampling density and stronger oscillatory behavior (thus, higher resolution) in regions of large sampling density. Thus, we can interpret the finite collection $ \{ \phi_0, \phi_1, \ldots, \phi_{l-1} \} $ as the $ l $-element orthonormal set on $ L^2(M,\mu) $ with the least expected roughness for the equivalence class of observation maps associated with $ h$, in the sense that $ \sum_{i=0}^{l-1} R_h( \phi_i ) \leq \sum_{i=0}^{l-1} R_h( f_i ) $, where $ \{ f_0, f_1, \ldots, f_{l-1} \} $ is any $ l $-element orthonormal set on $ L^2( M, \mu ) $. 

Intuitively, functions with small expected roughness can be robustly approximated from finite datasets in $ M $, and in the case of i.i.d.\ samples this intuition can be rigorously verified through pointwise and spectral convergence results established for graph Laplacians \cite{VonLuxburgEtAl04,CoifmanLafon06,Singer06,BelkinNiyogi07,VonLuxburgEtAl08,BerryHarlim16,BerrySauer16}. In particular, it has recently been shown \cite{BerrySauer16} that a spectrum of the Laplacian-Beltrami operator for a Riemannian metric analogous to $ h $ can be consistently estimated by normalized graph Laplacians for a variable-bandwidth kernel of the same class as~\eqref{eqKVB}. In \cite{BerrySauer16}, they also show that for the class of conformally invariant metrics $ h $ the variance of the approximated eigenvalues has the leading-order behavior $ \var \eta_i = \frac{ \epsilon^{-(m/2+1)} }{N( N- 1)} C \eta_i \langle \phi_i^2, \phi_i^2 \rangle / \langle \phi_i, \phi_i \rangle^2 $, where  $ C $ is a constant independent of $ \eta_i $, $ \phi_i $, and, importantly, the sampling density $ \sigma $. In other words, the effect of the conformal change of metric $ g \mapsto h $ is to ``undo'' the effect of sampling density fluctuations and increase the robustness of the approximated spectrum of $ \upDelta_h $.  While we are not aware of analogous error estimates in the case of correlated samples generated by ergodic dynamical systems, the favorable properties of $ h $ for robust data analysis should hold in that case too.

Notice now that the Laplace-Beltrami eigenfunctions in~\eqref{eqLapl} are orthogonal on $ H_1( M, h ) $ with $ \langle \phi_i, \phi_j \rangle_{H_1} = ( 1 + \eta_i ) \delta_{ij} $,  but because $ \lVert \phi_i \rVert_{H_1} = ( 1 + \eta_i )^{1/2} $ exhibits unbounded growth as $ i \to \infty $, functions of the form $ f = \sum_{i=0}^\infty c_i \phi_i $ with $ ( c_0, c_1, \ldots ) \in \ell^2 $ are not necessarily in $ H_1(  M, h ) $. On the other hand, the rescaled eigenfunctions 
\begin{equation}
  \label{eqBasisH1}
  \varphi_0 = \phi_0, \quad  \varphi_{i>0} = \eta_i^{-1/2} \phi_i, 
\end{equation}
with $ \langle \varphi_i, \varphi_j \rangle_{H_1} = ( 1 + \eta_i^{-1} ) \delta_{ij} $ for $ i>0 $ and $ ( \varphi_0, \varphi_i ) = \delta_{0i} $, are orthogonal  (but not orthonormal)  on $ H_1( M, h) $ and have bounded $ H_1 $ norm,  $ \lVert \varphi_i \rVert_{H_1} = ( 1 + \eta_i^{-1} )^{1/2} $. Therefore, $ \{ \varphi_i \}_{i=0}^\infty $ is an orthogonal basis of $ H_1(  M, h ) $ with the property that every sequence  $ ( c_0, c_1, \ldots ) \in \ell^2 $ corresponds to a function  $ f = \sum_{i=0}^\infty c_i \varphi_i \in H_1( M, h ) $. Moreover, the Dirichlet energies $E_h(\varphi_i)$ of the basis elements are all equal to one for $ i > 0 $. Due to this property and the fact  that the approximated eigenvalues and eigenfunctions of $ \upDelta_h $ are robust against variations in the sampling density, $ \{ \varphi_i \} $ will be our basis of choice for a well-conditioned Galerkin method for the eigenvalue problem of the Koopman generator. We also note that because $ E_h( \varphi_0 ) = 0 $ and $ E_h( \varphi_i ) = 1 $ for $ i > 1 $, the Dirichlet energy from~\eqref{eqDirichlet} can be conveniently computed from the $ \ell^2 $ norm of the expansion coefficients with $ i \geq 1 $, i.e., 
\begin{equation}
  \label{eqDirichletH1}
  E_h(f) = \sum_{i=1}^\infty \lvert c_i \rvert^2.
\end{equation}
 
\begin{rk}[Weighted Laplacian] An alternative elliptic operator to $ \upDelta_h $, whose eigenfunctions also provide an orthonormal basis of $ L^2( M, \mu ) $ is the weighted Laplacian $ \upDelta_{g,\mu} = - \divr_\mu \grad_g $ associated with the invariant measure of the dynamics and the ambient-space metric $ g $. This operator is the generator of a gradient flow on  $ ( M, g ) $ with potential $ - \log \sigma $, and its eigenvalues and eigenfunctions can be obtained by extremizing the Rayleigh quotient $ R_{g,\mu}( f ) = E_{g,\mu}(f ) /\lVert f \rVert^2 $, where $ E_{g,\mu}( f ) = \int_M \lVert \grad_g f \rVert^2_g \, d\mu $. Note that unlike $ E_h $ from~\eqref{eqDirichlet}, the Dirichlet energy $ E_{g,\mu} $ does not feature a $ \sigma $-dependent term in the integral with respect to the invariant measure. Numerically, eigenfunctions of $ \upDelta_{g,\mu} $ can be computed using variable-bandwidth kernels \cite{BerryHarlim16}, or the canonical formulation of diffusion maps with radial Gaussian kernels and the ``$\alpha=1/2$'' normalization \cite{CoifmanLafon06}.  In \cite{BerryEtAl15}, eigenfunctions of $ \upDelta_{g,\mu} $ approximated via a variable bandwidth kernel were used to approximate the evolution semigroup operators (the stochastic analogs of the Koopman and Perron-Frobenius operators) of stochastic dynamical systems on manifolds. The Galerkin scheme for the Koopman generator developed here can be implemented using eigenfunctions of either $ \upDelta_h $ or $ \upDelta_{g,\mu} $ (computed via either standard diffusion maps, or variable-bandwidth kernels), though in practice we find that $ \upDelta_h $ behaves more stably in applications with large variations of the sampling density $ \sigma $ (including the applications discussed in this paper).
\end{rk}
 
\subsection{\label{secDataDrivenBasis}Data-driven orthonormal basis}

To approximate the basis in~\eqref{eqBasisH1} from data, we start from the variable-bandwidth kernel $ K_\epsilon : \mathbb{ R }^d \times \mathbb{ R }^d \mapsto \mathbb{ R } $ given by \cite{BerryHarlim16}
\begin{equation}
  \label{eqKVB}K_\epsilon( x, y ) = \exp\left( - \frac{ \lVert x - y \rVert^2 }{ \epsilon \hat\sigma^{-1/m}_\epsilon( x ) \hat\sigma_\epsilon^{-1/m }( y )  } \right),
\end{equation}
where $ \epsilon $ is a positive bandwidth parameter, $ \hat\sigma_\epsilon $ is a function approximating  $ \sigma = d\mu/\dvol_g $ at $ O( \epsilon ) $ accuracy, and $ m $ is the dimension of $ M $. The function $ \hat\sigma_\epsilon $ can be computed using any suitable density-estimation technique, and in what follows we employ the kernel method described in \cite{BerryHarlim16,BerryEtAl15}. This method uses an automatic bandwidth-selection procedure based on the method originally developed in \cite{CoifmanEtAl08}, and also provides an estimate of $ m $. Alternatively, $ m $ can be estimated using one of the dimension estimation techniques available in the literature (e.g., \cite{HeinAudibert05,LittleEtAl11}). In the numerical experiments that follow we use a priori known values for $ m $, though the estimates from the bandwidth selection algorithm are in good agreement with the true values. 

We assume that data $ x_i = F( a_i ) $ is collected from an orbit $ a_0, a_1, \ldots, a_{N-1} $ on $ M$ of a discrete-time dynamical system with flow map $ \hat \Psi_n  : M \mapsto M $, $ n \in \mathbb{ Z }$, possessing an ergodic invariant measure $ \nu $ with a smooth density $ q = d\nu/ d\mu $,  bounded away from zero. Throughout this section, the sampling system $ ( M, \mathcal{B}, \hat \Psi_n, \nu ) $ will be the same as the discrete-time system  $ ( M, \mathcal{B}, \hat \Phi_n, \mu ) $ under study; i.e.,  $ q = 1 $ and the dataset  $ \{ x_i \}_{i=0}^{N-1} $ is as in~\eqref{eqDataset}. However, we carry out our analysis in the more general setting with nonuniform $ q $ in anticipation  of the time-change techniques of Section~\ref{secTimeChange}. Using the same method to tune the kernel bandwidth parameter as in the  density-estimation step, we perform the normalizations originally introduced in diffusion maps \cite{CoifmanLafon06} and further developed in \cite{BerrySauer16b} to construct from the kernel in~\eqref{eqKVB}  a compact, ergodic, Markov operator $ \mathcal{ P }_\epsilon : L^2( M, \mu ) \mapsto L^2( M, \mu ) $ approximating the heat operator on the Riemannian manifold $ ( M, h ) $.  Specifically, we compute the action of $ \mathcal{ P }_\epsilon $ on a function $ f \in L^2( M, \mu ) $ through the operations, 
\begin{equation}
  \label{eqG}
  \mathcal{ G }_\epsilon f = \frac{ 1 }{ \epsilon^{m/2 } } \int_M K_\epsilon( F( \cdot ), F( a ) ) f( a )  \, d \mu( a ), \quad
  q_\epsilon = \mathcal{ G }_\epsilon q, \quad \mathcal{ H }_\epsilon f = \mathcal{ G }_\epsilon( f  q / q_\epsilon ),  \quad \mathcal{ P }_\epsilon f = \frac{ \mathcal{ H }_\epsilon f }{ d_\epsilon }, \quad  d_\epsilon = \mathcal{ H }_\epsilon 1.
\end{equation}
The operator $ \mathcal{ P }_\epsilon $ preserves constant functions (i.e., it is an averaging operator with $ \mathcal{ P }_\epsilon 1 = 1 $), and we have $ \mathcal{ P }_\epsilon f = \int_M p_\epsilon( \cdot, a ) f( a ) \, d \mu( a ) $ for the Markov kernel 
\begin{displaymath}
  p_\epsilon( a, b ) = \frac{ K_\epsilon( F( a ), F( b ) ) q( b ) }{ d_\epsilon(a) q_\epsilon( b ) }, \quad \int_M p_\epsilon( a, b ) \, d\mu( b) = 1, \quad a, b \in M.
\end{displaymath}
The ergodicity of $ \mathcal{ P}_\epsilon $ follows from the fact that $ p_\epsilon $ is bounded away from zero; the latter is due to $M $ being compact and  $ q $ being bounded away from zero. 

The role of the normalizations in~\eqref{eqG} is to remove biases due to curvature in $ h $ and the potentially nonuniform sampling density $ q $. In particular, taking Taylor expansions of $ \mathcal{ P }_\epsilon $ about $ \epsilon = 0 $ \citep[][]{BerryHarlim16}, one can show that uniformly on $ M $, and independently of $ q $, 
\begin{equation}
  \label{eqPDM}
  \mathcal{ P }_\epsilon f( a ) = f( a ) - \epsilon c \upDelta_h f( a ) + O( \epsilon^2 ), \quad \forall f \in C^\infty(M), 
\end{equation}
for an $ f $-independent constant $ c $, so that $( I - \mathcal{ P }_\epsilon ) f(a)  / ( c \epsilon ) $ converges uniformly to $ \Delta_h f( a ) $. Equation~\eqref{eqPDM} together with the results in \citep[][]{CoifmanLafon06,VonLuxburgEtAl08} imply that for any $ \tau > 0 $, $ \mathcal{ P}_\epsilon^{\tau/\epsilon} $ converges in operator norm (and hence in spectrum) to the heat operator $ e^{-\tau \upDelta_h} $ associated with $ h $. As a result, we can approximate eigenvalues and eigenfunctions of $ \upDelta_h $ by eigenvalues and eigenfunctions of $ \mathcal{ P }_\epsilon $. In particular, one can verify that the eigenvalues $ \kappa_i $ of $ \mathcal{ P }_\epsilon $ are real, and (because $ \mathcal{ P }_\epsilon $  is compact, Markov, and ergodic) have finite multiplicities and admit the ordering $ 1 = \kappa_0 > \kappa_1 \geq \kappa_2 \geq \cdots $, accumulating only at 0. Moreover,  $ - \log \kappa_i / \epsilon $ converges as $ \epsilon \to 0 $ to the Laplace-Beltrami eigenvalue $ \eta_i $.   

Turning to the discrete setting, we represent functions on the dataset by $ N $-dimensional vectors $ \vec f = ( f_0, \ldots, f_{N-1} ) $ with components $ f_i = f( a_i ) $, and the integral operator $ \mathcal{ P }_\epsilon $ by an $ N \times N $ Markov matrix $ P $ such that $ \sum_{j=0}^{N-1} P_{ij} f_j $ converges to  $ \mathcal{ P }_\epsilon f( a_i ) $ for every $ f \in L^2( M,\mu) $ and $ \mu $-a.e.\ starting state $ a_0 $ in the training data. In particular, by the pointwise ergodic theorem, time averages of the form $ \sum_{i=0}^{N-1} f_i / N $ converge $ \mu $-a.s.\ to the integrals $ \int_M f q \, d\mu $. Therefore, up to  proportionality constants, we approximate $ q_\epsilon( a_i ) $ and $ d_\epsilon ( a_i ) $ by $ \hat q_i = \sum_{j=0}^{N-1} K_{ij} $ and $ \hat d_i = \sum_{j=0}^{N-1} H_{ij}  $, respectively, where $ K_{ij} = K_\epsilon( F( a_i ), F( a_j ) ) $, and $ H_{ij} = K_{ij} / \hat q_j $. The matrix elements of $ P $ are then given by $ P_{ij} = H_{ij} / \hat d_i $. With these ingredients, we approximate the eigenvalues and eigenfunctions of $ \mathcal{ P }_\epsilon $ using the corresponding eigenvalues and eigenvectors of $ P $, 
\begin{equation}
  \label{eqPhiDiscrete}
  P \vec \phi_i = \hat \kappa_i  \vec \phi_i, \quad i \in \{ 0, 1, \ldots, N -1 \}, \quad 1 = \hat \kappa_0 > \hat \kappa_1 \geq \hat \kappa_2 \geq \cdots \geq \hat \kappa_{N-1}, \quad \vec \phi_i = (  \phi_{0i}, \ldots,  \phi_{N-1,i} ) \in \mathbb{ R }^N.
\end{equation}
As $ N \to \infty $, and for a suitable scaling $ \epsilon( N ) \to 0 $,  we have $ \phi_{ji} \to \phi_j( a_i ) $, $ \hat \kappa_i \to \kappa_i $, and $ - \log \hat \kappa_i / \epsilon \to \eta_ i $, $ \mu $-a.s.\ and up to proportionality constants. By convention, we will work with the normalized eigenvalues $ \hat \eta_i = \log \hat \kappa_i / \log \hat \kappa_1 $. Geometrically, this scaling is equivalent to a uniform scaling of the Riemannian metric $ h $, which has no  influence on the techniques presented in Section~\ref{secGalerkin} ahead as it can be absorbed by a rescaling of the diffusion regularization parameter $ \varepsilon $.  

The eigenvectors from~\eqref{eqPhiDiscrete} form an orthonormal basis for functions sampled on the dataset with respect to the weighted inner product (cf.\ \eqref{eqMonteCarlo})
\begin{equation}
  \label{eqInnerProdPi}
  \langle \vec f_1, \vec f_2 \rangle_{w } = \frac{ 1 }{ N } \sum_{i=0}^{N-1}   f^*_1(a_i) f_2( a_i ) w_i, \quad  \vec f_j = ( f_j(a_0), \ldots, f_j( a_N ) ), \quad f_j \in L^2(M,\mu),
\end{equation}
where $ w_i $ are the elements of the stationary density $ \vec w = ( w_0, \ldots, w_{N-1} ) $ of $ P $, satisfying $ \vec w P = \vec w $ and $  \sum_{i=0}^{N-1 } w_i = 1 $. For $ P $ constructed as described above one can verify that $ w_i = ( \hat d_i / \hat q_i ) /  ( \sum_{j=0}^{N-1} \hat d_j /  \hat q_j   ) $. Using small-$\epsilon $ Taylor expansions as in~\eqref{eqPDM} and the pointwise ergodic theorem, it can also be shown that $ \lim_{\epsilon\to 0} \lim_{N\to \infty} ( N w_i ) = 1 / q(a_i ) $ for $ \nu $-a.e.\ starting state $ a_0 $ in the training data.  Thus, the inner product in~\eqref{eqInnerProdPi} is asymptotically equivalent to the $L^2 $ inner product with respect to $ \mu $; explicitly,  $ \lim_{\epsilon\to0} \lim_{N\to \infty} \langle \vec f_1, \vec f_2 \rangle_{w }  = \int f_1^* f_2 \, d\mu $, $ \mu $-a.s. 

The analog of the Dirichlet energy in~\eqref{eqDirichlet} for the discrete eigenfunctions is $  E(\vec \phi_i ) = \hat \eta_i $ (whenever $ \hat \kappa_i > 0$), and for the function $ \vec f $ above we have $  E(\vec f) =\sum_{i=1}^{N-1} \hat \eta_i \lvert c_i \rvert^2  $. This quantity converges  up to a proportionality constant to $ E_h( f) $. Hereafter, whenever there is no risk of confusion with the continuous case, we will omit hats and overarrows in our notation for quantities computed for the discrete dataset, such as $ \vec\phi_i  $,  $ \hat \eta_i $, and $ \hat \kappa_i $. The construction of our data-driven basis of  is summarized in Algorithm~\ref{algBasis} in~\ref{appAlgorithms}.   
  
\subsection{\label{secGalerkin}Spectral Galerkin method}

We solve the eigenvalue problem for the Koopman generator in weak form in the basis of $ H_1( M, \mu ) $ in~\eqref{eqBasisH1}. Because the set of eigenvalues may be dense on the imaginary line (see Remark~\ref{remEigDense}), we first regularize the problem by adding a small amount of diffusion to the dynamical vector field $ v $ to form the operator $L_{\varepsilon}: C^\infty(M) \mapsto C^\infty(M) $, where
\begin{equation}
  \label{eqL}L_{\varepsilon} = v - \varepsilon \upDelta_h, \quad \varepsilon > 0.
\end{equation}
We  solve the eigenvalue problem for this operator using spectral Galerkin methods for elliptic eigenvalue problems. We begin from the strong form of the problem, 
\begin{equation}
  \label{eqLStrong}
  L_{\varepsilon} u = \gamma u, \quad \gamma \in \mathbb{ C }, \quad u \in C^\infty( M ),
\end{equation}
where $ \Real \gamma $ and $ \Imag \gamma $ measure the growth rate and oscillatory frequency associated with eigenfunction  $u $. Note that $ L_\varepsilon $ is dissipative, $ \langle u, L_{\varepsilon} u \rangle \leq 0 $, so $ \Real \gamma $ is necessarily non-negative. Intuitively, for small $ \varepsilon $, $ \Imag \gamma $ and $ u $ should approximate a Koopman eigenfrequency $ \omega $ and an eigenfunction $ z $ from~\eqref{eqEigV}, respectively, and $\Real \gamma $ should be large and negative if the Dirichlet energy $ E_h( z ) $ is large. In other words, the role of the diffusion term $ \varepsilon \upDelta_h $ is to suppress highly oscillatory Koopman eigenfunctions, which, as stated in Remark~\ref{remOscillatory}, can have arbitrarily small frequencies. With these pathological eigenfunctions eliminated, we can identify approximate  generating Koopman eigenfunctions  $ \{ \zeta_i \}$ of minimal roughness and the corresponding frequencies $ \{ \Omega_i \} $ by ordering the eigenfunctions $ u $ in order of increasing Dirichlet energy. Note that in general $ L_\varepsilon $ is non-normal, and as a result its eigenfunctions will generally be non-orthogonal.  We will return to a discussion of the non-normality of $ L_\varepsilon $ and the effects of diffusion in its spectral properties in Section~\ref{secSpectral}.

To pass to a weak form of the problem, we multiply both sides of~\eqref{eqLStrong} by a test function $ \psi \in C^\infty( M ) $, integrate by parts with respect to the invariant measure, and require that the resulting integral equation is satisfied for all elements of appropriate trial and test spaces, which we both take  to be $ H_1( M, \mu )$.

\begin{defn}[Eigenvalue problem for $ L_{\varepsilon} $, weak form]
  \label{defEig}
    Find $ \gamma \in \mathbb{ C } $ and $ u \in H_1( M, \mu ) $ such that for any $ \psi \in H_1( M, \mu ) $,
  \begin{displaymath}
    \mathcal{ A }( \psi, u ) = \gamma \mathcal{ B }( \psi, u ),
  \end{displaymath}
  where $ \mathcal{ A } $ and $ \mathcal{ B } $ are sesquilinear forms on $ H_1( M, \mu ) \times H_1( M, \mu ) $ given by 
  \begin{displaymath}
    \mathcal{A}( \psi, u ) = \mathcal{ V }( \psi, u ) - {\varepsilon} \mathcal{ D }( \psi, u ), \quad \mathcal{ V }( \psi, u ) = \langle \psi, v( u ) \rangle, \quad  \mathcal{ D }( \psi, u ) = \langle \grad_h \psi, \grad_h u \rangle, \quad \mathcal{ B }( \psi, u ) = \langle \psi, u \rangle.
  \end{displaymath}
\end{defn}

One can verify that $ \mathcal{ A } $ has the boundedness and coercivity properties 
\begin{equation}
  \label{eqCoercivity}
  \lvert \mathcal{A}( \psi, z ) \rvert \leq C_1 \lVert \psi \rVert_{H_1} \lVert z \rVert_{H_1}, \quad  \mathcal{A}( z, z ) \leq -\varepsilon C_2 \lVert z \rVert^2_{H_1},
\end{equation}
respectively, where $ C_1, C_2 $ are positive constants, and $ \psi, z $ arbitrary functions in $ H_1(M,h) $ orthogonal to constant functions. Together, these conditions ensure that the variational eigenvalue problem in Definition~\ref{defEig} is well posed \citep{BabuskaOsborn91}. Note that the skew-symmetry of $ v $ (and therefore $ \mathcal{V  }$) is important for establishing the coercivity of $ \mathcal{A} $, and thus for us to be able to take advantage of the theory and spectral approximation techniques for variational eigenvalue problems. As stated in Remark~\ref{remTime}, meeting such well-posedness conditions would have been more challenging had we worked with the unitary Koopman operator $ U_t $ instead of $ v $. Another useful property due to the skew-symmetry of $ v $ is that the Dirichlet energies of the solutions can be determined from the real part of the eigenvalues, i.e.,
\begin{equation}
  \label{eqDirichletGamma}
  E_h( z ) = - \mathcal{ D }( z, z ) / \varepsilon = - \Real( \gamma ) / \varepsilon.
\end{equation}

In the Galerkin approximation of the problem, we formally restrict the trial and test spaces to the $ n $-dimensional subspaces $ H_{1,n} = \spn \{ \varphi_0, \ldots, \varphi_{n-1} \} \subset H_1( M, \mu ) $ spanned by the basis functions in~\eqref{eqBasisH1}; i.e., we have  $ \psi = \sum_{i=0}^{n-1} d_i \varphi_i $ and $ u = \sum_{i=0}^{n-1} c_i \varphi_i $, where  $ c = ( c_0, \ldots, c_{n-1} ) $ and $ d = ( d_0, \ldots, d_{n-1} ) $ are complex-valued expansion coefficients. However, because instead of the true Laplace-Beltrami eigenfunctions we only have access to the approximate eigenfunctions from Algorithm~\ref{algBasis}, and furthermore we do not have access to the exact vector field $ v $, we additionally make the following approximations for the evaluation of the sesquilinear forms in the continuous problem.
\begin{subequations}
  \label{eqMatEig}
  \begin{enumerate}
  \item We approximate the basis function values on the dataset by the vectors $ \vec \varphi_i = ( \varphi_{0i}, \ldots, \varphi_{N-1,i} ) \in \mathbb{ R }^N $ with $ \varphi_{ji } =  \phi_{ji} / \hat \eta_i \approx  \varphi_i( a_j ) $, where $ \{ \vec \phi_i \} $ and $ \{ \hat \eta_i \} $ are the eigenvectors and eigenvalues from Algorithm~\ref{algBasis}.
  \item In the case of $ \mathcal{ D } $, we put 
    \begin{equation}
      \mathcal{ D }( \psi, u ) = \sum_{i,j=0}^{n-1} \int_M d_i^* c_j \grad_h \varphi_i  \cdot \grad_h \varphi_j \, d \mu = d^\dag D c, 
    \end{equation}
    where $ D $ is the $ n \times n $ identity matrix with $ D_{ij} = \langle \grad_h \varphi_i, \grad_h \varphi_j \rangle = \delta_{ij} $.
  \item In the case of $ \mathcal{ B } $, we put
    \begin{equation}
      \mathcal{ B }( \psi, u ) = \sum_{i,j=0}^{n-1} \int_M d_i^* c_j \varphi_i \varphi_j \, d \mu \approx \sum_{i,j=0}^{n-1} d_i^* c_j \langle \vec \varphi_i, \vec \varphi_j \rangle_{w}  = d^\dag B c, 
    \end{equation}
    where $ B $ is the $ n \times n $ diagonal matrix with the diagonal entries $ B_{00} = 1$ and $ B_{ii} =\eta_i^{-1} $ for $ i > 1$.
  \item In the case of $ \mathcal{ V } $, we proceed similarly as with $ \mathcal{ D } $ and $ \mathcal{ B } $, but we also approximate the action of the vector field $ v( \varphi_{j} ) $ on the basis elements using finite differences in time as stated in Remark~\ref{remTime}. Hereafter, we will use a second-order central scheme for the sampling interval $ T $, viz., $ v( \varphi_{j}( a_k ) )\approx (  \varphi_{j,k+1} - \varphi_{j,k-1} ) / ( 2 T ) $. We therefore set
    \begin{equation}
      \mathcal{ V }( \psi, u ) = \sum_{i,j=0}^{n-1} \int_M d_i^* c_j \varphi_i v( \varphi_j ) \, d \mu \approx d^\dag V c, 
    \end{equation}
    where $ V $ is the $ n \times n $ matrix with elements $  V_{ij} = \sum_{k=1}^{N-2}  \varphi_{ik}  w_k (  \varphi_{j,k+1} - \varphi_{j,k-1} ) / ( 2 T ) $.
  \end{enumerate}
\end{subequations}
With these approximations, we define:
\begin{defn}[Eigenvalue problem for $ L_{\varepsilon} $, discrete approximation]  \label{defEigDisc} Find $ \gamma \in \mathbb{ C } $ and $ c \in \mathbb{ C }^n $ such that for any $ d \in \mathbb{ C }^n $,
\begin{displaymath}
  \hat{\mathcal{ A }}( d, c ) = \gamma \hat{\mathcal{ B }}( d, c ),
\end{displaymath}
where $ \hat{\mathcal{ A } }$ and $ \hat{\mathcal{ B } }$ are sesquilinear forms on $ \mathbb{ C }^n \times \mathbb{ C }^n $ given by
\begin{displaymath}
  \hat{\mathcal{A}}(d,c) = d^\dag A c, \quad A = V - \varepsilon D,  \quad \hat{\mathcal{B} }( d, c ) = d^\dag B c.
\end{displaymath}
\end{defn}
The solution to the discrete problem is given by the matrix generalized eigenvalue problem
\begin{equation}
  \label{eqGEV}
  A c = \gamma B c.
\end{equation}

\begin{rk}
  The $ \{ \varphi_i \} $ basis from~\eqref{eqBasisH1} and its discrete counterpart $ \{ \vec\varphi_i \} $ are adapted to the $ H_1 $ regularity of the weak eigenvalue problem for $ L_{\varepsilon} $ in the sense that the highest-order sesquilinear form $ \mathcal{ D } $ appearing in the weak formulation of the problem is represented by the identity matrix, $ D = I $. This property ensures that the scheme remains well-conditioned at large spectral orders of approximation $ n $. In contrast, the condition number of the $ D $ matrix would exhibit unbounded growth with $ n $ if we were to work in the unscaled eigenfunction basis. This approach of tailoring the approximation basis to the Sobolev regularity of the continuous problem is sometimes used in spectral Galerkin methods with polynomial basis functions (e.g., \cite{MelenkEtAl00,GiannakisEtAl09}). 
\end{rk}

The solution of the generalized eigenvalue problem in~\eqref{eqGEV} yields $ n' \leq n $  (depending on the numerical algorithm used) eigenvalue-eigenvector pairs $ \{ ( \gamma_k, c_k ) \}_{k=0}^{n'-1} $ with $ c_k = ( c_{0k}, \ldots, c_{n-1,k} ) \in \mathbb{ C }^n $ and the corresponding discretely sampled eigenfunctions $ u_k = \sum_{i=0}^{n-1}  \varphi_i c_{ik} $ with $ u_k = ( u_{0k}, \ldots, u_{N-1,k} ) \in \mathbb{ C }^N$. Throughout, we work with the normalization $ \lVert u_k \rVert_{w} = 1 $ from~\eqref{eqInnerProdPi}, which approximates the normalization $ \lVert u_k \rVert = 1 $ on $ L^2( M, \mu ) $. The errors of approximating the exact eigenvalues and eigenfunctions of the generator in~\eqref{eqEigV} via the eigenvalue problem in Definition~\ref{defEigDisc} can be summarized as (1) sampling errors (i.e., errors that vanish as $ N \to \infty $), (2) approximation errors in the Laplace-Beltrami eigenvalues and eigenfunctions computed through the operator $ \mathcal{ P }_\epsilon $ in diffusion maps (i.e., errors that vanish as $ \epsilon \to 0 $), (3) finite-difference errors in approximating the action of $ v $ on functions (i.e., errors that vanish as $ T \to 0 $), (4) Galerkin approximation errors (i.e., errors that vanish as $ n \to \infty $), and (5) diffusion regularization errors from the use of the regularized generator  $ L_{\varepsilon} $ with $ \varepsilon > 0 $. Among these, errors of type (1)--(4) vanish unconditionally in the respective limits stated above. However, as with many other-data driven techniques for Koopman and Perron-Frobenius operators (see Section~\ref{secBackground}), our method's behavior with respect to diffusion regularization is more complicated to analyze, particularly if the system has Koopman eigenfunctions not lying in $H_1 $ and/or continuous spectrum. We will return to these points in Section~\ref{secSpectral}.    

To identify a generating set $ \{ \zeta_i\}_{i=1}^m $ for the eigenfunctions and the corresponding basic frequencies $\{ \Omega_i \}_{i=1}^m $ using the approach of Section~\ref{secDimensionReduction}, we first order the eigenfunctions $ u_k $ in order of increasing Dirichlet energy, which we compute from the discrete analog of~\eqref{eqDirichletH1}, $ E( u_k)= \sum_{i=1}^n \lvert c_{ik} \rvert^2 = \lVert c_k \rVert^2 $. Due to~\eqref{eqDirichletGamma}, we can expect that $ E( u_k ) \approx - \Real \gamma_k / \varepsilon $ (where the equality is approximate due to errors of type (1)--(4) discussed above), so it generally suffices to compute a subset of the eigenvalues with the largest real parts; this is particularly convenient when solving~\eqref{eqGEV} with iterative solvers.  Then, we select $ \{ \zeta_i \}_{i=1}^m $ and $ \{ \Omega_i \}_{i=1}^m $ from the first $ m $ nonconstant eigenfunctions in this set with ``numerically rationally independent'' eigenvalues up to some precision. Operationally, we declare $ \Omega_i $ and $ \Omega_j $ to be rationally independent at precision $ ( \delta, \bar q ) $ if there exist no integers $ q_i $ and $ q_j $ with absolute value smaller than $ \bar q $ such that $ \lvert q_j \Omega_i - q_i \Omega_j \rvert \leq \delta $. In practice, it is usually easy to identify  rationally independent frequencies from the first few numerical eigenvalues manually. The numerical procedure to compute the generating frequencies and eigenfunctions is summarized in Algorithm~\ref{algGenerators} in~\ref{appAlgorithms}. 

Using the identified generating sets, we form the product bases $ \{ z_k \} $ and associated frequencies $ \{ \omega_k \} $ as described in Definition~\ref{defnPurePoint}. Note that due to the errors described above,  the numerical generators $ \zeta_i = (\zeta_{0i}, \ldots, \zeta_{N-1,i} ) \in \mathbb{ C }^N  $ will not lie exactly on the unit circle. We therefore rescale the generators pointwise to ensure that $  \lvert \zeta_{ij} \rvert = 1 $ before taking powers. Specifically, fixing a positive spectral order parameter $ l $, we compute the approximate eigenfunctions $ z_k = \prod_{k=1}^m \zeta_i^{k_i} $ (with products and powers taken elementwise on vectors in $ \mathbb{ C }^N $)  and the corresponding eigenfrequencies $ \omega_k = \sum_{i=1}^m k_i \Omega_i $ for all integers $ k_i $ in the range $ [ -l, l ] $. This leads to a tensor product dictionary of data-driven observables, consisting of $ ( 2 l + 1 )^m $ approximate Koopman eigenfunctions. Experimentally, we find that computing the  $ z_k $  recursively from the generators produces more accurate results than the raw solutions of~\eqref{eqGEV}, presumably because the accuracy of the data-driven basis $ \{ \phi_i \} $ degrades more rapidly at large $ i $ than the loss of accuracy resulting from the unit-circle normalization of the  $  \zeta_i  $ and subsequent products and powers  to form higher-order Koopman eigenfunctions. 

With the availability of the dictionary $ \{ z_k \} $ and the corresponding eigenfrequencies $ \{ \omega_k \} $, algorithms for dimension reduction, vector field decomposition, and forecasting of densities in systems with pure point spectra can be constructed following closely the continuous formulation in Sections~\ref{secDimensionReduction} and~\ref{secForecasting}, replacing inner products on $ L^2( M, \mu ) $ with the weighted inner products from~\eqref{eqInnerProdPi}, and representing the action of the dynamical vector field through finite differences of the diffusion maps basis as described above. An additional issue that needs to be taken into account is the non-orthogonality of the eigenfunctions of $L_\varepsilon$ discussed in Section~\ref{secSpectral}. In particular, the lack of orthogonality of the eigenfunctions leads to a modification of the inverse transforms to reconstruct observables and vector fields. For an observable $ f $ with expansion coefficients $ \hat f_k = \langle f, z_k \rangle_{w} $, where the $ z_k $ are computed recursively from the generating eigenfunctions as described in Section~\ref{secGalerkin}, we now  have $ f = \sum_k \tilde f_k z_k $ (as opposed to $ f = \sum_k \hat f_k z_k$ in the case of exact Koopman eigenfunctions), where $ \tilde f_k = \sum_i G^{-1} _{ki} \hat f_i $, and $ G^{-1} $ is the inverse of the Gramm matrix $ G_{ij} = \langle z_i, z_j \rangle_{w} $. This matrix will generally fail to be an identity matrix due to both sampling and bias errors, but in applications we find that it is a well conditioned, sparse matrix. Our algorithms for vector field decomposition and nonparametric forecasting are listed in Algorithms~\ref{algVDecomp} and~\ref{algStatisticalForecast} in~\ref{appAlgorithms}, respectively.

\subsection{\label{secSpectral}Spectral properties of the regularized generator}

The spectral properties of the regularized generator $ L_{\varepsilon} $ from~\eqref{eqL} are compounded by the facts that (1) the limit $ \varepsilon \to 0 $ is a singular limit of the corresponding eigenvalue equation; (2) apart from special cases (e.g., irrational flows on flat tori), the operators $ v $ and $ \upDelta_h $ do not commute, and as result $ L_{\varepsilon} $ is nonnormal with 
\begin{equation}
  \label{eqLComm}
  [ L_{\varepsilon}^*, L_{\varepsilon} ] = 2 {\varepsilon} [ v, \upDelta_h ].
\end{equation} 
If the generator has a complete set of smooth eigenfunctions, heuristic asymptotic expansions suggest that the influence of the diffusion term on these eigenfunctions should be benign. In particular, writing $ u_k = z_k + \varepsilon u_k' + O( \varepsilon^2 ) $ and $ \gamma_k = \lambda_k + \varepsilon \gamma_k' + O( \varepsilon^2) $ with $ v( z_k ) = \lambda_k z_k $, and inserting these asymptotic series in the eigenvalue equation~\eqref{eqLStrong}, we obtain the $ O( \varepsilon ) $ equation 
\begin{displaymath}
  ( v - \lambda_k ) u'_k = ( \gamma'_k  - \upDelta_h ) z_k.
\end{displaymath}
This equation can be solved by requiring that $ u'_k $ is orthogonal to $ z_k $ as a solvability condition, i.e., $ u'_k = \sum_{i\neq k} c_{ik} z_i $, giving
\begin{displaymath}
  c_{ik} = \frac{ \langle \grad_h z_i, \grad_h z_k \rangle }{ \lambda_i - \lambda_k }, \quad \gamma'_k = \langle \grad_h z_k, \grad_h z_k \rangle = E_h(z_k).
\end{displaymath}
We therefore see that, at $ O( \varepsilon ) $, the diffusion term perturbs the eigenvalues of $ v $ corresponding to smooth eigenfunctions by a purely real term equal to the Dirichlet energy of the unperturbed eigenfunctions. This provides a more quantitative estimate of the suppression of highly-oscillatory eigenfunctions of $ v $ from the spectrum of $ L_{\varepsilon} $ claimed in Section~\ref{secGalerkin}. Note that the imaginary part of the perturbation to $ \lambda_k $,  which is important for nonparametric forecasting, occurs at $ O( \varepsilon^2 ) $.  It also follows from these results that the non-orthogonality of the eigenfunctions of $ L_{\varepsilon} $ can be estimated by $ \langle u_i, u_k \rangle = \varepsilon c_{ik} + O( \varepsilon^2 ) $, and the coefficients $ c_{ik} $ vanish if $ z_i $ and $ z_k $ are eigenfunctions of both $ v $  and $ \upDelta_h $. 

In general, orthogonality and completeness of the eigenfunctions of $ L_{\varepsilon} $ is guaranteed at all orders in $ \varepsilon $ if $ v $ and $ \upDelta_h $ are commuting operators. One situation that this occurs is when the dynamical flow $ \Phi_t $ preserves the Riemannian metric $ h $, i.e., $ v $ generates $ h $-isometries. A necessary and sufficient condition for $ v $ to generate isometries is that $ v $ is a Killing vector field satisfying the equation $ \mathcal{ L}_v h = 0 $, where $ \mathcal{ L }_v $ is the Lie derivative on type $ ( 0, 2 ) $ tensors with respect to $ v $. One can check that if this equation is satisfied then $ [ v, \upDelta_h ] = 0 $. Moreover, when $ v $ generates isometries of a smooth Riemannian metric, then the generator $ \tilde v $  must necessarily have a pure point spectrum \citep[][\S7.1.c]{HasselblattKatok02}.  This is because the group of diffeomorphisms of a Riemannian manifold is a compact Lie group containing $\{ \Phi_t  \} $ as an Abelian subgroup, and translations on compact Abelian groups have pure point spectra (see Section~\ref{secErgodicity}). 

A particularly important property that holds if $ [ v, \upDelta_h ] = 0 $ is that $ v $ and $ \upDelta_h $ have joint smooth eigenfunctions. This means that (1) the eigenvalue problem for $ L_{\varepsilon} $ yields exact eigenfunctions of $ v $; (2) the Galerkin method for approximating these eigenfunctions becomes highly efficient in the eigenbasis of $ \upDelta_h $; (3) the eigenfunctions of $ v $ are extrema of the Rayleigh quotient $ R_h $ for the metric $ h $. In particular, (3) implies that the leading generating eigenfunction $ \zeta_1 $ identified with respect to the Dirichlet energy $E_h( \zeta_1 ) $ lies entirely in the subspace of $ L^2(M,\mu) $ spanned by eigenfunctions of $ \upDelta_h $ corresponding to its smallest nonzero eigenvalue. These facts suggest that for a system with pure point spectrum it would be preferable to regularize the generator with diffusion in a Riemannian metric $ h $ preserved by the dynamics. In Section~\ref{secTakens}, we will present an approach for approximating such a metric using delay-coordinate maps, but for the rest of this Section we will continue to work with  $ h $ as defined in Section~\ref{secGalerkinPrelim}. As expected from the asymptotics (and verified in the experiments in Section~\ref{secTorus} ahead), the effects of noncommutativity of $ v $ and $ \upDelta_h $ should have minimal impact on the quality of the numerical eigenfunctions in this case. Moreover, it is important to establish that accurate Koopman eigenfunctions can be computed without having to perform delay-coordinate maps.          

What about the behavior of $ L_{\varepsilon} $ in more general ergodic systems where $ \tilde v $ has non-smooth eigenfunctions and/or  continuous spectrum, and in particular in weak-mixing systems where it has no nonconstant eigenfunctions? Mathematically, $ L_{\varepsilon} $ has the same structure as  a class of advection-diffusion operators arising in viscous, incompressible fluid dynamics on compact manifolds (recall that $ v $ generates an incompressible flow with respect to the invariant measure), for which theoretical results are available in the literature \cite{Franke04,ConstantinEtAl08,FrankeEtAl10}. The latter references study the dynamical and spectral properties of operators of the form $ \tilde L_\alpha = \alpha v + \upDelta $ in the advection dominated regime, $ \alpha \to \infty $, and these operators are equivalent to $ \varepsilon^{-1} L_{\varepsilon} $ with $ \varepsilon = 1 / \alpha $. Franke et al.\ \cite{FrankeEtAl10} show that (a closed extension of) $ \tilde L_\alpha $, and hence $ L_{\varepsilon} $, has only point spectrum irrespective of the mixing properties of $ \tilde v $. For our purposes, this is a positive result as it eliminates an important source of numerical instability in the discrete problem. However, the physical significance of the eigenfunctions of $ L_{\varepsilon} $ (and their utility in the dimension reduction and forecasting schemes of Sections~\ref{secPurePointSpectra} and~\ref{secGalerkinImplementation}) becomes questionable, especially in weak-mixing systems where $ \tilde v $ has no nonconstant eigenfunctions. 

In \cite{ConstantinEtAl08}, Constantin et al.\ establish necessary and sufficient conditions for $ \tilde L_\alpha $ to have the so-called relaxation-enhancing property as $ \alpha \to \infty $, i.e., the property that the associated Kolmogorov semigroup (the advection-diffusion analog of the Koopman group) will produce relaxation of any function in $ L^2(M , \mu ) $ to its mean value in arbitrarily small time. They show that $ \tilde L_\alpha $ is relaxation enhancing if and only if $ v $ has no eigenfunctions in $ H_1( M, \mu ) $; a condition that includes but is not necessarily limited to weak-mixing systems. The spectral manifestation of the relaxation-enhancing property is that the spectral gap of $ \tilde L_\alpha $ diverges as $ \alpha \to \infty $ \cite{FrankeEtAl10}, meaning that $ L_{\varepsilon} = \varepsilon \tilde L_{1/\varepsilon} $ will generally exhibit complicated spectral behavior, including the possibility of no eigenvalues with negligible real parts at arbitrarily small $ \varepsilon $. On the other hand, if $ \tilde v $ has eigenfunctions in $ H_1( M, \mu ) $, then $ L_{\varepsilon} $ must necessarily have a vanishing spectral gap, and the asymptotic behavior described earlier applies. Due to these considerations, rather than attempting to regularize weak-mixing systems by diffusion alone, in Section~\ref{secTimeChange} we will put forward an alternative approach which also involves a time change \citep[][]{KatokThouvenot06}, implemented with variable bandwidth kernels of the same class as~\eqref{eqKVB}, attempting to reduce the mixing properties of the system while preserving its orbits.

\subsection{\label{secTorus}Applications to variable-speed flows on the 2-torus}

We apply the dimension reduction, vector field decomposition, and nonparametric forecasting techniques described above to the dynamical system on $ \mathbb{ T }^2 $ with the vector field (cf.\ \eqref{eqTorusIrrational})
\begin{equation}
  \label{eqTorusDynamical}
  v = \sum_{i=1}^2 v^i \frac{ \partial \ }{ \partial \theta^i }, \quad  v^1 = 1 + ( 1 - \beta )^{1/2} \cos \theta^1, \quad   v^2  = \alpha ( 1 - ( 1 - \beta )^{1/2} \sin \theta^2 ),
\end{equation}
where $ \theta^1, \theta^2 \in [ 0, 2 \pi ) $ are canonical angles on the torus, $ \alpha $ is a frequency parameter set to an irrational number, and $ \beta \in ( 0, 1 ] $ is a parameter controlling the speed of the flow. This system, which was also studied in \cite{Giannakis15}, is ergodic and has pure point spectrum as it can be transformed to an irrational linear flow with appropriate changes of coordinates. Its unique Borel invariant measure has density 
\begin{equation}
  \label{eqTorusSigma}
  \sigma( \theta^1, \theta^2 ) \propto 1 / [ ( 1 + ( 1 - \beta )^{1/2} \cos( \theta^1 ) )( 1 - ( 1 - \beta )^{1/2} \sin\theta^2 ) ]
\end{equation}
relative to the Haar measure on $ \mathbb{ T }^2 $, and the orbit $ ( \theta^1( t ), \theta^2( t ) ) $ passing through $ ( 0, 0 ) $ at time $ t = 0 $ is given by
\begin{equation}
  \label{eqTorusOrbit}
  \tan ( \theta^1( t ) / 2 ) = [ 1 + ( 1 - \beta )^{1/2} ] \beta^{-1/2}  \tan( \beta t / 2 ), \quad \cot( \theta^2( t )  / 2 ) = ( 1 - \beta )^{1/2} + \beta^{1/2} \cot( \beta^{1/2} \alpha t / 2 ).
\end{equation} 

Here, we use the same frequency parameter,  $ \alpha = 30^{1/2} $, as in the irrational flow example of Section~\ref{secIrrational}, and set the parameter $ \beta $ to $ 1/2 $. We also use the same observation map $F$ as in Section~\ref{secIrrational}, i.e., the standard embedding of the 2-torus into $ \mathbb{ R }^3 $ from~\eqref{eqTorusEmbedding} with radius parameter $ R = 1 / 2 $. As illustrated in Fig.~\ref{figTorusX},  with this choice of parameters the system ``slows down'' at $ (\theta^1, \theta^2 ) \sim ( \pi, \pi/2 ) $ and ``speeds up'' at $ ( \theta^1, \theta^2 ) \sim ( 0, -\pi/2 ) $, resulting in a large density contrast $ \sigma( \pi, \pi/ 2 ) / \sigma( 0, - \pi / 2 ) \simeq 34 $ relative to the Haar measure. Moreover, the system exhibits two timescales which are mixed together in the components of the observation map. Despite the apparent complexity of these time series, the underlying dynamical system is completely integrable and our method is able to detect this structure using no information other than time-ordered data.    

\begin{figure}
  \centering\includegraphics{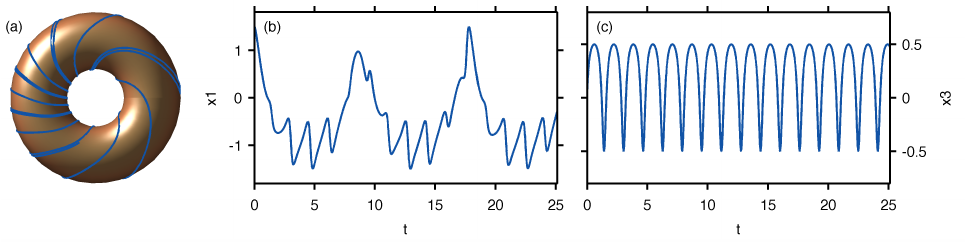}
  \caption{\label{figTorusX}Observed time series for the variable-speed dynamical system on the 2-torus. (a) Embedding in $ \mathbb{ R }^3 $ via $ F $; (b, c) components $ x^1 $ and $ x^3 $ of the embedding, respectively.}
\end{figure}

In our experiments, we used as training data a time series consisting of $ N = \text{64,000} $ samples $ \{ x_i \}_{i=0}^{N-1} $, $ x_i = F( \theta( t_i ) ) $, $ t_i = ( i - 1 ) T $,  taken from the orbit in~\eqref{eqTorusOrbit} at a timestep $ T = 2 \pi / 500 $. Using this data, we computed $ n = 1000 $ diffusion eigenvalues and eigenfunctions $ \{ ( \eta_i, \phi_i ) \}_{i=0}^{n-1} $ through Algorithm~\ref{algBasis}, and constructed the generating set $ \{ \zeta_1, \zeta_2 \} $ of Koopman eigenfunctions and the corresponding basic frequencies $ \{ \Omega_1, \Omega_2 \} $ using Algorithm~\ref{algGenerators} with $ \varepsilon = 3 \times 10^{-4} $. We used the true manifold dimension, $ m = 2 $, as an input to these algorithms, but in this case an accurate dimension estimate can also be obtained from the kernel density estimation procedure in Algorithm~\ref{algBasis}.   The eigenvalues of the approximate generator $ L_{\varepsilon} $ corresponding to the basic frequencies are $ \gamma_1 = 0.0005 + 0.707 \ii $ and $ \gamma_3 = 0.0016 + 3.871 \ii $ ($ \gamma_2 $ is the complex conjugate of $ \gamma_1 $), and $ \Omega_i = \Imag \gamma_i $. According to our convention, these eigenvalues are ordered in order of increasing Dirichlet energy; in this case, $ E_1 = 1.52 $ and $ E_3 = 5.34 $. Moreover, the computed eigenvalues have $ \Real \gamma_1 / \varepsilon = 1.51 \approx E_1 $     and $ \Real \gamma_3 / \varepsilon = 5.34 \approx E_2 $, which is in good agreement with~\eqref{eqDirichletGamma}. 

The generating eigenfunctions from these calculations are displayed in Fig.~\ref{figTorusZ}. There, it is evident that the projection maps from~\eqref{eqPiProj} based on the eigenfunctions send the dataset to a near-exact circle in the complex plane, and the time series $ t_j \mapsto \zeta_i( \theta_j ) $ describe simple harmonic oscillations as expected from~\eqref{eqZSHO}. The frequencies of these oscillations (measured, e.g., through FFT) are in very good agreement with $ \Omega_i $. As illustrated by the scatterplots in Fig.~\ref{figTorusZ}(a, b), eigenfunctions $ \zeta_1 $ and $ \zeta_2 $ vary purely along angles $ \theta^1 $ and $ \theta^2 $, respectively. This is an outcome of the fact that in the induced Riemannian geometry from $ F $ these eigenfunctions have the smallest Dirichlet energies, but in other geometries we could have mixtures of the form $ \zeta_1^{q_1} \zeta_2^{q_2} $ appearing as the least oscillatory eigenfunctions. Compared to their counterparts for the irrational flow in Section~\ref{secIrrational} (which are pure sinusoids in the $ \theta^1 $ and $ \theta^2 $ angles), the generating eigenfunctions for the variable-speed flow exhibit ``compressed''  (``expanded'') waveforms in the regions of anomalously high (low) phase-space speed. As discussed in the SOM, the results in Fig.~\ref{figTorusZ} are in good agreement with results obtained via EDMD using a dictionary of functions consisting of time-lagged components of the observation vector in $ \mathbb{ R }^3 $. 

\begin{figure}
  \centering\includegraphics[scale=.9]{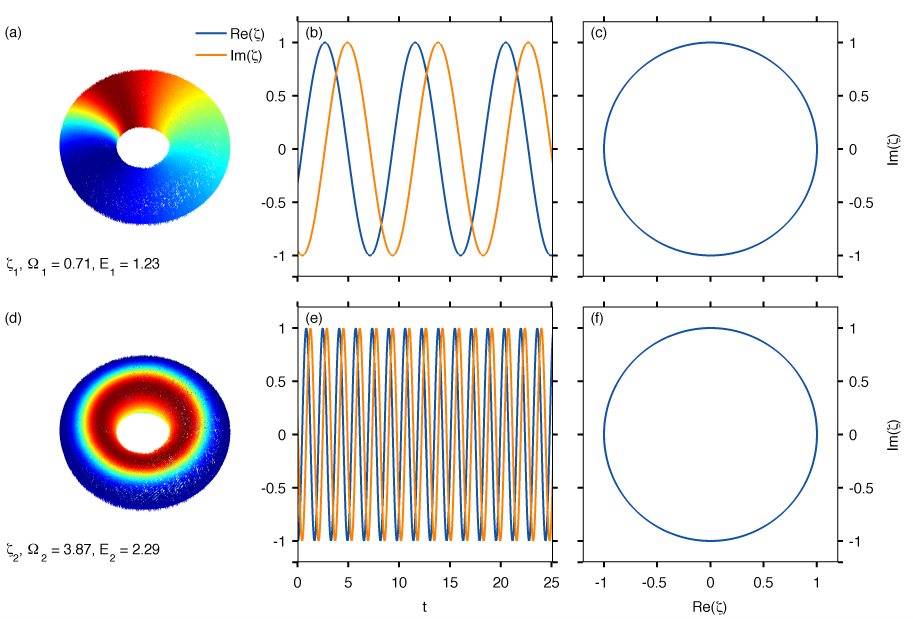}
  \caption{\label{figTorusZ}The generating Koopman eigenfunctions $ \zeta_1 $ (a--c) and $ \zeta_2 $ (d--f) for the variable-speed dynamical system on the 2-torus. (a, d) Scatterplots of $ \Real( \zeta_i ) $ on the torus; (b, e) time series of $ \Real \zeta_i $ and $ \Imag \zeta_i  $; (c, f) scatterplots of $ (\Real \zeta_i, \Imag \zeta_i ) $. The eigenfunctions shown here have not been rescaled to the unit circle---they are the result of Step~5 of Algorithm~\ref{algGenerators}.}
\end{figure}

With these results for the generating eigenfunctions and basic frequencies, we use Algorithm~\ref{algVDecomp} to decompose $ v $ into the mutually commuting vector fields $ \{ v_i \}_{i=1}^2 $ from Theorem~\ref{lemmaVDecomp}, and reconstruct these vector fields in data space. Figure~\ref{figTorusV} shows the reconstructions obtained with  the spectral order parameter $ l = 30 $, corresponding to $ n = ( 2 l + 1 )^2 = 3721 $ eigenfunctions. The vector fields exhibit the qualitative features expected from the corresponding generating eigenfunctions in Fig.~\ref{figTorusZ}; that is, the reconstructed vector fields, $ V_1 $ and $ V_2 $, describe flows purely along the $ \theta^ 1 $ and $ \theta^2 $ directions, respectively. As discussed in Section~\ref{secVectorField}, these flows preserve the equilibrium measure of the dynamics (but not the Riemannian measure of the torus embedded in $ \mathbb{ R }^3 $), and are clearly not ergodic as $ v_1 $ ($v_2$) has a non-trivial nullspace spanned by functions with no $ \theta^2 $ ($\theta^1$) dependence. With this large number of basis functions, the time-averaged (root mean square) reconstruction error in $ V_1 + V_2 $ is 0.2\% of the time-averaged norm of the full vector field $ V $ in $ \mathbb{ R }^3 $. Note that high-quality reconstructions are possible with substantially fewer basis functions (e.g., $ l = 5 $, $ ( 2 l + 1 )^2 = 121 $ yields a 5.9\% error). Moreover, the expansion coefficients of the $ V_i $ in the $ \{ z_k \} $ basis exhibit strong sparsity, suggesting that the reconstructed patterns can be efficiently compressed using subsets of the $ \{ z_k \} $ basis as dictionaries.       

\begin{figure}
  \centering\includegraphics[scale=1]{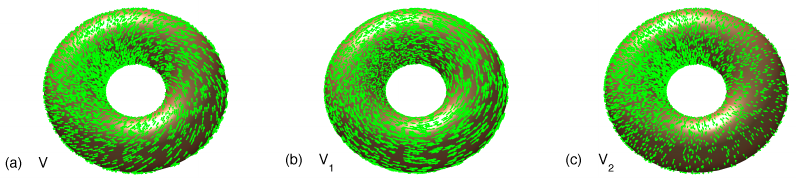}
  \caption{\label{figTorusV}Vector field decomposition for the variable-speed dynamical system on the 2-torus. (a) Full vector field $ v $ embedded in $ \mathbb{ R }^3 $; (b, c) mutually commuting components $ v_i $ from Theorem~\ref{lemmaVDecomp} corresponding to the  generating eigenfunctions $ \zeta_i$ in Fig.~\ref{figTorusZ}, reconstructed in $ \mathbb{ R }^3 $ via the pushforward map in~\eqref{eqVPushZeta}.}
\end{figure}

Next, we consider statistical forecasting of the components $ F^1( \theta ) = x^1 $ and $ F^3( \theta ) = x^3 $ of the  observation map using the nonparametric technique in Algorithm~\ref{algStatisticalForecast}. As with the irrational-flow example in Section~\ref{secIrrational}, we assign an initial probability measure $ \mu_0 $ with a von Mises density relative to the Haar measure.  In this case, we set the location and concentration parameters $ ( \bar \theta^1, \bar \theta^2 ) = ( \pi, \pi ) $ and $ \kappa = 30 $, respectively. Note that the equilibrium measure $ \mu $ of the variable-speed system differs from the Haar measure on the torus, so the initial density $ \rho_0 = d \mu_0 / d \mu $ is given by the density in~\eqref{eqVonMises} divided by $ \sigma $ from~\eqref{eqTorusSigma}. We normalize the initial density on the dataset so that $ \sum_{i=0}^{N-1} \rho_{0 i } / N = 1 $, where  $ \rho_{0,i} = \rho_0( a_i ) $. A scatterplot of the initial density on the torus is shown in Fig.~\ref{figTorusRhoT}(a). We advance this density and the corresponding expectation values and standard deviations for $ x^1 $ and $ x^3 $ for the lead times $ \{ t_i \}_{i=1}^{1000} $ with $ t_i = i \, T $ using Algorithm~\ref{algStatisticalForecast}. Figure~\ref{figTorusRhoT} shows representative snapshots of the time-dependent density for these lead times, which is is also visualized as a video in Movie~1. Figure~\ref{figTorusForecast} displays the means and standard deviations of $ x^1 $ and $ x^3 $ obtained via Algorithm~\ref{algStatisticalForecast} and an ensemble forecast with the perfect model using 10,000 samples drawn independently from $ \mu_0 $.

\begin{figure}
  \centering\includegraphics{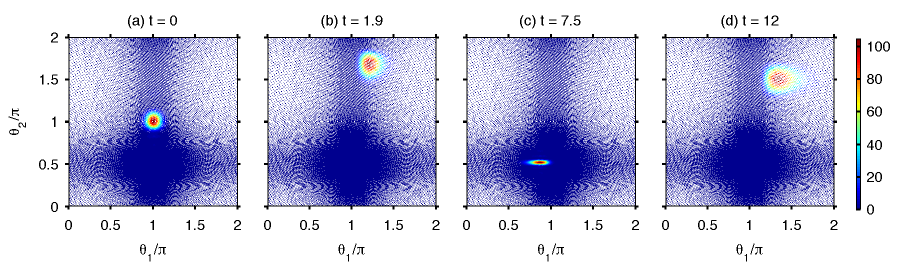}
  \caption{\label{figTorusRhoT}Time-dependent density $ \rho_t $ for the variable-speed dynamical system on the 2-torus relative to its equilibrium measure computed via the nonparametric method in Algorithm~\ref{algStatisticalForecast}. (a) Initial von Mises density with location and concentration parameters $ ( \pi, \pi ) $ and  $ 30 $, respectively; (b--d) snapshots of $ \rho_t $ illustrating phases of high uncertainty (b, d) and a phase of low uncertainty (c). The dynamic evolution of $ \rho_t $ is also shown in Movie~1.}
\end{figure}

\begin{figure}
  \centering\includegraphics[scale=1]{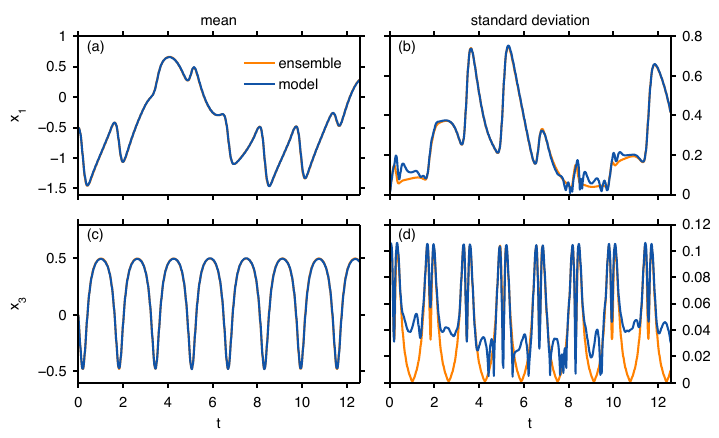}
  \caption{\label{figTorusForecast}Statistical forecasts via the nonparametric method in Algorithm~\ref{algStatisticalForecast} and a 10,000-member ensemble from the perfect model for the variable-speed system on the 2-torus. The forecast observables are the components $ x^1 $ (a, b) and $ x^3 $ (c, d) of the embedding of $\mathbb{T}^2$ into $ \mathbb{ R }^3 $ from~\eqref{eqTorusEmbedding} with radius parameter $ R = 1/2 $. The initial probability measure has the density in Fig.~\ref{figTorusRhoT}(a).}
\end{figure}

Following an initial transient stage, the probability density in Fig.~\ref{figTorusRhoT} relaxes to an aperiodic pattern characterized by alternating phases of high (Fig.~\ref{figTorusRhoT}(b, d)) and low (Fig.~\ref{figTorusRhoT}(b)) uncertainty, corresponding to the regions with high speed (low density relative to the Haar measure) and low speed (high density relative to the Haar measure), respectively. Qualitatively, the time-dependent measure will eventually explore the full phase space on the torus by ergodicity, but because the system is not mixing, the density $ \rho_t $ retains its coherence and its Dirichlet energy is bounded in time. The results in Fig.~\ref{figTorusForecast}(a, c) indicate that the mean forecast with the nonparametric model is in very good agreement with the ensemble forecast. The standard deviation forecast in Fig.~\ref{figTorusForecast}(b, d) is also in good agreement with the ensemble, but in this case the nonparametric model overestimates the standard deviation in the low-uncertainty periods (especially for $ x^3 $). A factor that may be contributing to this discrepancy is that during those periods the density $ \rho_t $ exhibits small-scale behavior and correspondingly large bandwidth in the $ \{ z_k \} $ basis, limiting the accuracy of our finite-$ l $ truncation.         
 
\section{\label{secTakens}Galerkin method with delay-coordinate maps for systems with pure point spectra}

Despite the attractive numerical results in Section~\ref{secTorus}, in Section~\ref{secSpectral} we saw that due to non-commutativity of $ v $ and $ \upDelta_h $, the advection-diffusion operator $ L_{\varepsilon} $ constructed via~\eqref{eqL} may not be optimal for approximating Koopman eigenfunctions. In this Section, we present an approach for spectral decomposition of systems with pure point spectra that uses  delay-coordinate maps to construct a diffusion operator that commutes with the Koopman group. This approach leads to a highly efficient Galerkin scheme for solving the Koopman eigenvalue problem, and provides a natural way of denoising data corrupted by i.i.d.\ observational noise. 

\subsection{\label{secDelay}Relationship between Koopman and Laplace-Beltrami operators in delay-coordinate space}

Delay-coordinate maps \cite{PackardEtAl80,Takens81,SauerEtAl91,Robinson05,DeyleSugihara11} were originally developed as a state-space reconstruction technique that maps a time-ordered signal into a higher-dimensional space of sequences where, under mild assumptions, the attractor of the dynamical system generating the data is recovered. Fixing an integer parameter $ s $ (the number of delays), we replace the observed time series $ \{  x_i \}_{i=0}^{N-1} $ in $ \mathbb{ R }^d $ by the time series $ \{ X_i \}_{i={s-1}}^{N-1} $, where $  X_i = ( x_i, x_{i-1}, \ldots, x_{i-s+1} ) \in \mathbb{ R }^{sd} $. This procedure implicitly defines a new observation map, $ F_s : M \mapsto \mathbb{ R }^{sd} $, with 
\begin{equation}
  \label{eqDelay}
  F_s( a_i ) = ( F( a_i ), F( \hat \Phi_{i-1}( a_i ) ), \ldots, F( \hat \Phi_{i-s+1}( a_i ) ) ), 
\end{equation} 
where $ F $ and $ \hat \Phi_i $ are the observation map and discrete-time flow map introduced in Section~\ref{secErgodicity}, respectively. (Recall that the dynamical system is assumed to be invertible so we can evaluate $ \hat \Phi_k $ for both positive and negative $k$.) Here, we are interested in the behavior of diffusion maps for this class of observation maps at large numbers of delays $ s $.    

First, we examine the Riemannian metric $ g_s $ induced on $ M $ via the embedding in~\eqref{eqDelay} when $ R^{sd} $ is equipped with the ``time-average'' inner product $ \langle X_i, X_j \rangle_s =  \sum_{k=0}^{s-1} \langle x_{i-k},  x_{j-k} \rangle / s $, where $ \langle \cdot, \cdot \rangle $ denotes the canonical inner product on $\mathbb{ R }^d $. The Riemannian inner product of two tangent vectors $ u_1, u_2 \in T_a M $ with respect to this metric becomes
\begin{equation}
  \label{eqGS}
  g_s ( u_1, u_2 ) = \frac{ 1 }{ s } \sum_{k=0}^{s-1} g( \hat \Phi_{-k*}  u_1, \hat \Phi_{-k*}  u_2 ) =  \frac{ 1 }{ s } \sum_{k=0}^{s-1} \langle F_* \Phi_{-k*}  u_1,  F_* \Phi_{-k*}  u_2 \rangle,
\end{equation}
where $ g $ is the induced Riemannian metric associated with $ F $, and $  \hat \Phi_{k*} : T M \mapsto TM $ and $ F_* : TM \mapsto T \mathbb{ R }^d $  are the pushforward maps on tangent vectors associated  $ \hat \Phi_k $ and $ F $, respectively. The limiting behavior of the sequence of smooth Riemannian metrics $ g_1, g_2, \ldots $  is controlled by the spectrum of Lyapunov exponents of the system. For systems with nonzero Lyapunov exponents, $ g_s $ will generally fail to converge to a smooth tensor. In particular, according to Oseledets' multiplicative ergodic theorem \cite{Oseledets68,Arnold98}, there exists a splitting of the tangent bundle into the direct sum $ TM = \bigoplus_{i=1}^m E_i $ where the subspaces $E_i $ are invariant under $ \hat \Phi_{k*} $, and for $ \mu $-a.e.\ $ a \in M $ and $ u \in E_i\rvert_a \setminus 0 $, 
\begin{displaymath}
  \lim_{k \to \infty} \frac{ 1 }{ 2k } \log\left(  g ( \Phi_{-k*}  u, \Phi_{-k*} u ) \right) = - \Lambda_i, 
\end{displaymath}
where $ \Lambda_i $ is the Lyapunov exponent corresponding to $ E_i $.  This means that if $ \Lambda_i < 0 $, the inner product $ g_s( u, u) $ from~\eqref{eqGS} will exhibit exponential growth with $ s $, and similarly if $ \Lambda_i > 0 $ $ g_s( u, u) $ will converge to zero. Thus, in systems with nonzero Lyapunov exponents the limit metric $ \lim_{s\to\infty} g_s $ will be non-smooth. Note that had we used forward-looking instead of backward-looking delays, $ g_s $ would expand (contract) the unstable (stable) subspaces. However, as shown below, in the case of systems with pure point spectra and smooth Koopman eigenfunctions, $ \bar g = \lim_{s\to\infty} $ is a smooth, flow-invariant metric tensor.
\begin{thm}
  \label{thmGBar}
  In systems with pure point spectra, $ C^\infty $ Koopman eigenfunctions, and $ m $ basic frequencies,
  \begin{displaymath}
    \bar g = \sum_{i=1}^m B_{ij} \beta_i \otimes \beta_j, 
  \end{displaymath}
  where $ \beta_i $ are smooth dual vector fields to the vector fields $ v_i $ from Theorem~\ref{lemmaVDecomp}, defined uniquely through the relations $ \beta_i( v_j ) = \delta_{ij} $, and $ B_{ij} = \int_M g( v_i, v_j ) \, d\mu $ are the Hodge inner products of the $ v_i $ on  $ L^2( TM, g, \mu ) $. Moreover, $ \bar g $ has the following properties:
  
  (i) It is invariant under each of the flows $ \Phi_{i,t} $ generated by $ v_i $; i.e., $ \Phi_{i,t*} \bar g = \bar g $, where $ \Phi_{i,t*} $ is the pullback map on $ (0,2) $ tensors associated with $ \Phi_{i,t*} $.  

  (ii) It is flat.  
  
  (iii) Its associated volume form has uniform density relative to the invariant measure; i.e., $ \dvol_{\bar g} / d \mu = \Gamma $, where $ \Gamma $ is a positive constant.  

\end{thm}
A proof of Theorem~\ref{thmGBar} can be found in~\ref{appGBar}. 
\begin{cor} \label{corGBarInv}  The inverse metric associated with $ \bar g $ is given by $ \bar g^{-1} = \sum_{i=1}^m B^{-1}_{ij} v_i \otimes v_j $, where $ B^{-1}_{ij} $ are the elements of the $ m \times m $ inverse Gramm matrix $ [ B_{ij} ]_{ij}^{-1} $.
\end{cor}

According to Theorem~\ref{thmGBar}, the $ v_i $ are Killing vector fields of $ \bar g $, which in turn implies that $ v_i $ (and hence $ v $) commute with the Laplace-Beltrami operator $ \upDelta_{\bar g} : C^\infty( M ) \mapsto C^\infty( M) $ associated with $ \bar g $; a result which can be also verified by explicit calculation of $ \upDelta_{\bar g} $.
\begin{lemma} 
  \label{lemmaDeltaBar}The Laplace-Beltrami operator $ \upDelta_{\bar g}$ associated with the metric $ \bar g $ from Theorem~\ref{thmGBar}  is
  \begin{displaymath}
    \upDelta_{\bar g} = - \sum_{i,j=1}^m B_{ij}^{-1} v_i \circ v_j.
  \end{displaymath}
\end{lemma}
\begin{proof}
  It follows from Corollary~\ref{corGBarInv} that the gradient of a function $ f \in C^\infty(M) $ with respect to $ \bar g$ is given by $ \grad_{\bar g} f = \bar g^{-1}( \cdot, df ) = \sum_{i=1}^m B^{-1}_{ij} v_i( f ) v_j $. Moreover, according to Theorem~\ref{lemmaVDecomp}, the $ v_i $ have vanishing $ \mu $-divergence, and since $ \vol_{\bar g} = \Gamma \mu $ by Theorem~\ref{thmGBar}(iii), the $ v_i $ have vanishing divergence with respect to $ \vol_{\bar g} $ too (in fact, $ \divr_{\mu} = \divr_{\bar g} $ as can be seen from the expression for the divergence in local coordinates in \eqref{eqGradDiv}). Using the Leibniz rule for the divergence, $ \divr_{\bar g}( u f ) = f \divr_{\bar g} u + u( f ) $ where $ u $ is an arbitrary smooth vector field, we obtain
  \begin{displaymath}
    \upDelta_{\bar g} f = - \divr_{\bar g} \grad_{\bar g} f = - \sum_{i,j=1}^m B_{ij}^{-1} \divr_{\bar g}( v_i( f ) v_j )  = - \sum_{i,j=1}^m B_{ij}^{-1} ( v_i \circ v_j )( f ). \qedhere
  \end{displaymath}
\end{proof}

The expression for $ \upDelta_{\bar g} $ in Lemma~\ref{lemmaDeltaBar}  in conjunction with the fact that $ [ v_i, v ] = 0 $ manifestly shows that $ v $ and $ \upDelta_{\bar g} $ are commuting operators. Thus, a Koopman eigenfunction $ z_k $ with corresponding eigenvalue $ \lambda_k = \ii \sum_{i=1}^m k_i \Omega_i $ is also an eigenfunction of $ \upDelta_{\bar g} $. In particular, $ \{ z_k, z_k^* \} $, or, equivalently $ \{ \Real z_k, \Imag z_k \} $, are orthogonal eigenfunctions of $ \upDelta_{\bar g} $ at the corresponding eigenvalue $ \sum_{i=1}^m B^{-1}_{ij} \Omega_i \Omega_j k_i k_j  $.  As stated in Section~\ref{secSpectral}, the leading generating eigenfunction $ \zeta_1 $ selected on the basis of the Dirichlet energy $ E_{\bar g}( \zeta_1 ) $ lies entirely in the eigenspace of $ \upDelta_{\bar g } $ corresponding to its smallest nonzero eigenvalue, $ B_{11}^{-1} \Omega_1^2 $. Moreover, the Galerkin approximation space associated with the eigenfunctions of $ \upDelta_{\bar g} $ is efficient for approximating the eigenfunctions of $ v $ since each $ z_k $ is expressible as a finite linear combination of eigenfunctions of $ \upDelta_{\bar g} $.   

To approximate eigenfunctions of $ \upDelta_{\bar g} $ using a finite number of delays, we apply diffusion maps with a variable-bandwidth kernel on $ \mathbb{ R }^{sd} \times \mathbb{ R }^{sd} $ analogous to~\eqref{eqKVB}, viz.
\begin{equation}
  \label{eqKVBE}K_\epsilon(  X_i, X_j ) = \exp\left( - \frac{ \lVert X_i -  X_j \rVert^2 }{ \epsilon \hat \sigma^{-1/m}_{s,\epsilon}(  X_i ) \hat \sigma_{s,\epsilon}^{-1/m }(  X_j )  } \right),
\end{equation}
where $ \hat \sigma_{s,\epsilon} $ are estimates of the sampling density $ \sigma_s = d\mu/\dvol_{g_s} $, accurate at $ O( \epsilon ) $. Note that as $ s \to \infty $, $ \sigma_s $ tends to a constant $ \bar \sigma = 1 / \Gamma $ in accordance with Theorem~\ref{thmGBar}(iii), but at finite $ s $ it will exhibit fluctuations, and the variable-bandwidth kernel ensures that the numerical eigenfunctions are orthogonal with respect to the invariant measure despite fluctuations in the sampling density. In particular, at finite $ s $, diffusion maps approximates the Laplace-Beltrami operator $ \upDelta_{h_s} $ associated with the conformally transformed metric $ h_s = g_s \sigma_s^{2/m} $, which has analogous invariance  properties and stable behavior as the metric $ h $ introduced in Section~\ref{secGalerkinPrelim}.

The procedure for tuning the kernel bandwidth and computing $ \hat\sigma_{s,\epsilon} $ and the Laplace-Beltrami eigenfunctions follows Algorithm~\ref{algBasis} with the kernel in~\eqref{eqKVB} replaced by~\eqref{eqKVBE}. The computational cost of pairwise-kernel evaluations associated with delay-coordinate maps scales linearly with $ s $. In applications, we sometimes found that the automatic tuning procedure somewhat underestimates appropriate values for $ \epsilon $ at large $ s $, and a modest bandwidth inflation was necessary to produce stable solutions to the diffusion maps eigenvalue problem. However, the results of the Galerkin method in Algorithm~\ref{algGenerators} were not too sensitive after $ \epsilon $ exceeded a threshold. We have checked that the numerical results presented in this Section and in Section~\ref{secNoise} are robust for different values of bandwidth inflation in the range $ 4 \times $ to $ 10 \times $. In practice, there are obvious limitations  on how large $ s $ can be since delay embedding reduces the $ N $ samples originally available for analysis to $ N - s $ (the first $ s $ samples are ``used up'' to create $ \tilde X_1 $), and $ N - s $ must be large to ensure convergence.     
   
To illustrate the behavior of the diffusion maps basis obtained via this approach, in Fig.~\ref{figTorusC} we compare the moduli $ \lvert \langle \phi_k, \zeta_i \rangle \rvert $ of the expansion coefficients of the Koopman eigenfunctions computed via Algorithm~\ref{algGenerators} in the diffusion maps basis $ \{ \phi_k \} $ for $ s = 1 $ (i.e.,  the case studied in Section~\ref{secTorus}) and $ s = 800 $. In the latter case, the ambient space dimension is $ sd = 2400 $, but the dataset lies in a set $ F_s( M ) $ of intrinsic dimension $ m = 2 $ which has the manifold structure of a 2-torus equipped (according to Theorem~\ref{thmGBar}) with an approximately flat metric. To compute Koopman eigenvalues and eigenfunctions, we used the same dataset (prior to delay-coordinate mapping) as in Section~\ref{secTorus}, and executed Algorithm~\ref{algGenerators} using the same parameters as in that Section. While the spectra of the $ \zeta_i $ in Fig.~\ref{figTorusC} are concentrated in the leading ($\lesssim 30$) diffusion eigenfunctions in both cases, it is evident that in the $ s=800 $ case the spectra are significantly sparser, especially for $ \zeta_1 $ which projects almost entirely onto the diffusion eigenfunctions in the leading two-dimensional eigenspace with corresponding eigenvalue $ \eta_1 $, as expected theoretically. Eigenfunction $ \zeta_2 $ also exhibits a tight spectral expansion for $ s = 800 $ but includes significant contributions from three pairs of diffusion eigenfunctions. 
       
\begin{figure}
  \centering\includegraphics[scale=.9]{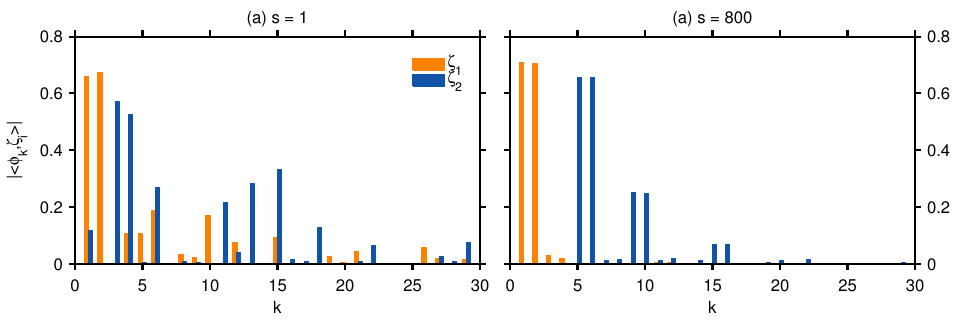}
  \caption{\label{figTorusC}Moduli  $ \lvert \langle \phi_k, \zeta_i \rangle \rvert $ of the expansion coefficients of the generating Koopman eigenfunctions $ \zeta_1 $ and $ \zeta_2 $ in the Laplace-Beltrami eigenfunction basis $ \{ \phi_k \} $ from diffusion maps for the variable-speed flow on the 2-torus. (a) No delay coordinate maps, $ s= 1 $; (b) delay-coordinate maps with $ s = 800 $ delays.}
\end{figure}

\begin{rk}[Diffusion maps and timescale separation] Delay-coordinate maps have previously been employed in conjunction with diffusion maps in methods for extraction of spatiotemporal patterns in complex systems \cite{GiannakisMajda12a,BerryEtAl13}. In these works and related applications (e.g., \cite{SlawinskaGiannakis17}), a behavior that has typically been observed is that as the number of delays increases, the time series formed by the diffusion eigenfunctions $ \phi_k $ become increasingly monochromatic, and are able to isolate distinct frequencies from broadband input signals. In~\cite{BerryEtAl13} this behavior was justified under the assumption that, for a suitable normalization, the diffusion operator approximated by diffusion maps approximates the long-time relaxation to equilibrium of the true (deterministic) system. Under that assumption, the time series of the leading $ \phi_k $ should evolve independently at the characteristic timescales determined by the corresponding diffusion eigenvalues, and the fact that the diffusion eigenvalues have finite spacings leads to timescale separation. Here, we have shown that in the case of systems with pure point spectra and smooth Koopman eigenfunctions, the timescale separation in diffusion maps can be explained from the connection between the generator of the Koopman group and the Laplace-Beltrami operator established in Lemma~\ref{lemmaDeltaBar}. That is, in such systems and in the limit of infinitely many delays  the diffusion eigenfunctions converge to Koopman eigenfunctions, which provide ``maximal'' timescale separation due to the fact that they evolve at a single frequency. Of course, the analysis presented here applies only in a rather restrictive setting, but our results should nevertheless provide a useful reference point to study connections between diffusion maps and Koopman operators in more complex systems. \end{rk}

As stated below~\eqref{eqGS}, in systems with nonzero Lyapunov exponents the induced metric $ g_s $ becomes ill-behaved as $ s \to \infty $ since it exponentially expands the stable Oseledets subspaces while shrinking the stable subspaces (in a manner that preserves the invariant measure).  The consequence of this behavior in the context of diffusion maps is that in the limit $ s \to \infty $ and $ \epsilon \to 0 $ the operator $ ( I - \mathcal{ P }_\epsilon) / \epsilon  $ (see Section~\ref{secDataDrivenBasis}) fails to converge to a Laplace-Beltrami operator associated with a smooth metric. Operationally, this means that as $ s $ increases the eigenfunctions from diffusion maps become increasingly biased towards the subspace of $ L^2( M, \mu ) $ consisting of functions with vanishing directional derivatives along all but the most stable subspace (in other words, eigenfunctions with appreciable directional derivatives along the unstable subspaces become increasingly rough in the kernel induced geometry). This effect was first identified in~\cite{BerryEtAl13}, who used a weighted version of the delay-coordinate map in~\eqref{eqDelay}  to regularize the induced metric and establish connections with Lyapunov metrics of dynamical systems \cite{Arnold98}. While biasing the eigenfunctions towards stable subspaces may actually be desirable in certain cases (e.g., by providing an effective means for intrinsic dimension reduction in systems with many positive Lyapunov exponents), in other cases, delay-coordinate maps with many delays may hinder the performance of approximation methods for Koopman eigenfunctions (e.g., if approximate eigenfunctions associated with unstable directions are desired). For these reasons, the availability of schemes such as those presented in Section~\ref{secGalerkinImplementation} and Section~\ref{secTimeChange} ahead which are able to accurately approximate Koopman eigenfunctions without performing delay-coordinate maps is important.    

\subsection{\label{secNoise}Koopman eigenfunctions from noisy data}

In this section, we demonstrate that besides being useful for improving the efficiency of spectral Galerkin schemes for the Koopman eigenvalue problem, delay-coordinate maps are also effective when dealing with data generated by dynamical systems with pure point spectra which are corrupted by i.i.d.\ observational noise. Specifically, we consider that instead of the noise-free time series $ \{ x_i \} $ we observe a noisy time series $ \{ \tilde x_i \} $, where $ \tilde x_i = x_i + \xi_i $, and the $ \xi_i $ are i.i.d.\ random variables in $ \mathbb{ R }^d $ such that $ \mathbb{ E } ( \xi_i ) = 0 $, $ \mathbb{ E }(  \lVert \xi_i \rVert^2 ) = R^2 < \infty $, and the third and fourth moments of $ \xi_i $ are finite (note that the $ \xi_i $ should not to be confused with the time change function $ \xi $ introduced in Section~\ref{secTimeChange} ahead). Here, the assumption that the $ \xi_i $ have vanishing expectation leads to no loss of generality since a nonzero constant expectation can be absorbed by a shift of the data, which leaves the pairwise distances in the diffusion maps kernel unchanged. Besides the above requirements, we do not make specific assumptions about the distribution of the $ \xi_i $, but in the numerical experiments below we use Gaussian noise.    

Following the approach presented in Section~\ref{secDelay}, we fix an integer parameter $ s $ and replace the observed time series $ \{ \tilde x_i \}_{i=0}^{N-1} $ by the time series $ \{ \tilde X_i \}_{i={s-1}}^{N-1} $, where $ \tilde X_i = ( \tilde x_i, \tilde x_{i-1}, \ldots, \tilde x_{i-s+1} ) \in \mathbb{ R }^{sd} $. Note that we can also write $ \tilde X_i = X_i + \Xi_i $, where   $ X_i = F_s( a_i ) = (  x_i,  x_{i-1}, \ldots, x_{i-s+1} ) $ are  the noise-free samples in delay-coordinate space, and $ \Xi_i = ( \xi_i, \xi_{i-1}, \ldots, \xi_{i-s+1} ) $. To compute diffusion eigenfunctions from noisy data, we replace the kernel in~\eqref{eqKVBE} by 
\begin{equation}
  \label{eqKVB2}K_\epsilon( \tilde X_i, \tilde X_j ) = \exp\left( - \frac{ \lVert \tilde X_i - \tilde X_j \rVert^2 }{ \epsilon \tau^{-1/m}_{s,\epsilon}( \tilde X_i ) \tau_{s,\epsilon}^{-1/m }( \tilde X_j )  } \right),
\end{equation}
where the functions $ \tau_{s,\epsilon}( \tilde X_i ) $ are modified bandwidth functions which estimate the sampling density up to a proportionality constant in the presence of noise. Details on these functions and the asymptotic properties of diffusion maps with the kernel in~\eqref{eqKVB2} are included in~\ref{appNoise}. 
 
Kernels in delay-embedding space have previously been used in conjunction with diffusion maps for correction of timing uncertainties in time-ordered data \cite{FungEtAl16}. As shown in Theorem~\ref{thmNoise} in~\ref{appNoise}, such kernels are also useful for denoising data corrupted by i.i.d.\ noise. In particular, the effect of i.i.d.\ noise is to introduce a random bias with positive expectation to the pairwise squared distances $ \lVert \tilde X_i - \tilde X_j \rVert^2 $. By the law of large numbers, as the number of delays $ s $ increases with a suitable scaling $ s(N) = o( N )$, the effect of that bias cancels in the diffusion maps normalization. As a result, the diffusion maps matrix, $ \tilde P $, constructed from the noisy data provides a consistent approximation of the integral operator $ \mathcal{ P }_\epsilon $ (see Section~\ref{secDataDrivenBasis}) up to an $ O( \epsilon^2) $ correction which does not affect the pointwise consistency of the Laplace-Beltrami operator approximated through $ ( I - \tilde P_\epsilon ) / \epsilon $. Note that the flatness of the metric $ \bar g $ established in Theorem~\ref{thmGBar} is important in deriving this result.  In summary, for sufficiently large $ s $ we can employ our Galerkin scheme for the Koopman eigenvalue problem in Algorithm~\ref{algGenerators} using clean basis functions determined through the eigenfunctions of $ \tilde P $. Theoretically, this approach allows for the removal of i.i.d.\ noise of arbitrarily large variance provided that arbitrarily many delays $ s $ can be used. In practice, the amount of variance that can be feasibly handled is limited by the total number of samples $ N $ (as we must have  $s \ll N $). Furthermore, if the system has nonzero Lyapunov exponents (a case that we are not treating here) adding delays leads to the eigenfunction bias issues discussed in Section~\ref{secDelay}.

Figures~\ref{figNoisyTorusZ1} and~\ref{figNoisyTorusZ2} show the generating Koopman eigenfunctions computed via the approach described above and in~\ref{appNoise} for the same variable-speed torus system as in Section~\ref{secTorus}, but with Gaussian noise added to the observed data in $ \mathbb{ R }^3 $. We consider two cases; one with moderate noise of standard deviation 0.1 added to each component of the observation vector (Fig.~\ref{figNoisyTorusZ1}), and one with strong noise of standard deviation 1 (Fig.~\ref{figNoisyTorusZ1}). In the moderate-noise case, we used a dataset of 64,000 samples sampled at the same time interval as in Section~\ref{secTorus}. In the strong-noise case the number of samples was 128,000 though reasonably good results can also be obtained using 64,000 samples. For ease of comparison, we used $ s = 800 $ delays in both cases, though in the moderate-noise the case as few as 20 delays are sufficient for denoising. Using the diffusion eigenfunctions from each case, we computed the generating Koopman eigenfunctions using Algorithm~\ref{algGenerators} with the same parameters as in Section~\ref{secTorus} except that in the strong-noise case we used $ n = 401 $ (as opposed to 1001) diffusion eigenfunctions to build the Galerkin approximation space.  

\begin{figure}
  \centering\includegraphics[scale=.8]{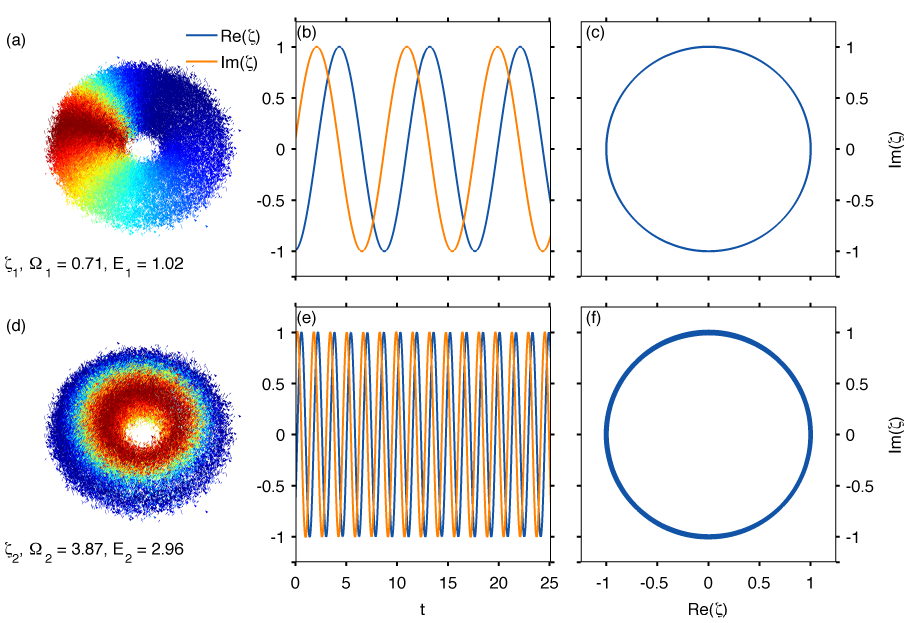}
  \caption{\label{figNoisyTorusZ1}The generating Koopman eigenfunctions $ \zeta_1 $ (a--c) and $ \zeta_2 $ (d--f) for the variable-speed dynamical system on the 2-torus computed using data corrupted by i.i.d.\ Gaussian noise with standard deviation 0.1 in each component of the observation vector in $ \mathbb{ R }^3 $. (a, d) Scatterplots of $ \Real( \zeta_i ) $ on the noisy torus; (b, e) time series of $ \Real \zeta_i $ and $ \Imag \zeta_i  $; (c, f) scatterplots of $ (\Real \zeta_i, \Imag \zeta_i ) $. The eigenfunctions shown here have not been rescaled to the unit circle---they are the result of Step~5 of Algorithm~\ref{algGenerators} executed with the modified kernel in~\eqref{eqKVB2} for $ s = 800 $ lags. The number of samples in this dataset is 64,000.}
\end{figure}
      
\begin{figure}
  \centering\includegraphics[scale=.8]{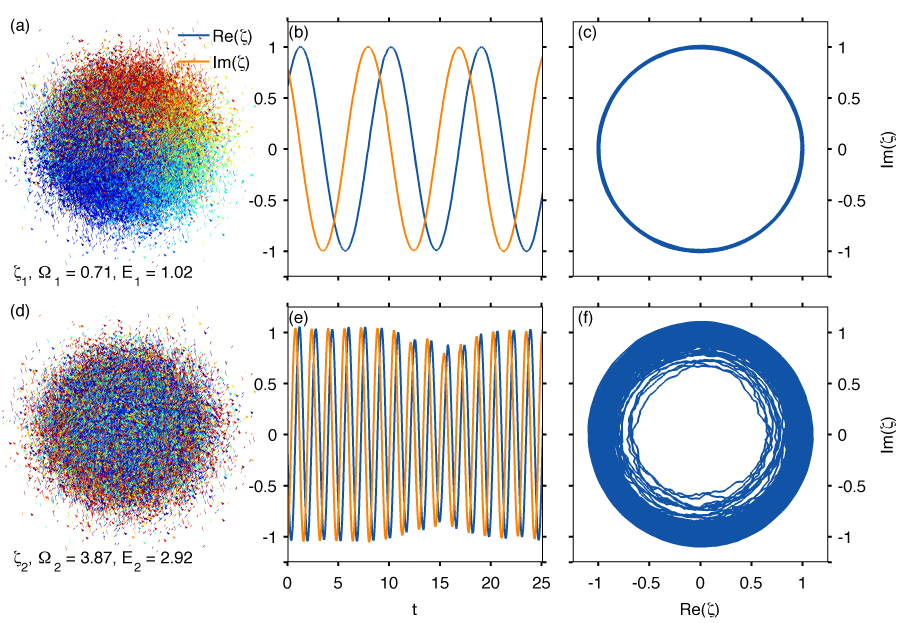}
  \caption{\label{figNoisyTorusZ2}As in Fig.~\ref{figNoisyTorusZ1}, but for noise standard deviation equal to 1 and number of observed samples equal to 128,000.}
\end{figure}

As is evident from Figs.~\ref{figNoisyTorusZ1} and~\ref{figNoisyTorusZ2}, with the modifications for noisy data described above, Algorithms~\ref{algBasis} and~\ref{algGenerators} successfully recover the generating Koopman eigenfunctions shown for the noise-free data in Fig.~\ref{figTorusZ}, in both the moderate- and strong-noise cases. In both cases, the Koopman frequencies $ \Omega_i $ agree to their counterparts computed from the noise-free data to within 3 significant figures. As expected from its small Dirichlet energy, eigenfunction $ \zeta_1 $ is particularly well recovered in both cases. Eigenfunction $ \zeta_2 $ exhibits a moderate amount of amplitude modulation in the strong-noise case, but it retains a nearly-monochromatic oscillation at the correct theoretical frequency. Koopman eigenfunction results analogous to Figs.~\ref{figNoisyTorusZ1} and~\ref{figNoisyTorusZ2} obtained via EDMD are shown in Figs.~2 and~3 in the SOM, respectively. There, it can be seen that while our approach and EDMD perform comparably  in the moderate-noise case, in the strong-noise case the quality of the numerical Koopman eigenfunctions obtained via our approach is higher.   

\section{\label{secTimeChange}Regularization by time change}

\subsection{\label{secTimeChangeOverview}Time change in dynamical systems}

Time change (e.g., \citep[][]{KatokThouvenot06}) is a technique in continuous-time dynamical systems which involves multiplying the vector field of a dynamical system by a positive function to create a new, orbit-equivalent system. For our purposes, we consider   a smooth ergodic dynamical system $ ( M, \mathcal{ B }, \nu, \Psi_t ) $ with vector field $ w $ and invariant probability measure $ \nu $ satisfying the assumptions stated in section~\ref{secErgodicity}, and a smooth, positive time-change function $ \psi $, bounded away from zero, yielding the smooth vector field $ v = \psi w $. The vector field $ v $ generates a flow $ \Phi_t $ on $ M $ having the same orbits as the original flow $ \Psi_t $ (since $ v $ and $ w$ are parallel), but time flows differently with respect to $ \Phi_t $ than it does with respect to $ \Psi_t $. In particular, for $ \Psi_\tau a = \Phi_t a $ we have $ \tau = \int_0^t \psi( \Phi_{t'} a ) \, d t' $. Moreover,  $ ( M, \mathcal{ B }, \mu, \Phi_t ) $ is ergodic for the invariant probability measure $ \mu $, absolutely continuous with respect to $\nu$, with $ C^\infty $ density 
\begin{equation}
  \label{eqNu}
  \varrho = \frac{ d \mu }{ d \nu } = \frac{  1 /\psi  }{  \int_M  d \nu / \psi }.
\end{equation}  
 
While time change preserves ergodicity, the same is not true for mixing---a theorem due to Ko\v{c}ergin \cite{Kocergin73} states that for any ergodic flow there exists a time change by a smooth function which renders it mixing. Heuristically, we can think of mixing by time change in the following way. Suppose that $ \Psi_t $ is non-mixing, and $ A \subset M $ is a measurable set that we picture as a coherent globular object. In the original  system, $ \Psi_t( A ) $ maintains its coherence in the course of dynamical evolution, but in the time-changed system parts of $ A $ will lag behind others. For a suitable time-change function, $ \Phi_t(A ) $ will become increasingly stretched and twisted, and its volume will become asymptotically equidistributed in $ M $ in accordance with the condition for mixing in Section~\ref{secErgodicity}. As a concrete example, which we will study with numerical experiments in Section~\ref{secMixing3Torus} ahead, Fayad \cite{Fayad02} shows that an irrational flow on the 3-torus becomes mixing through the analytic time-change function 
\begin{equation}
  \label{eqMixing3Torus}
  \psi( \theta^1, \theta^2, \theta^3 ) = 1 \\ + \Real \sum_{k=1}^\infty \sum_{\lvert l \rvert \leq k} \frac{ e^{-k }}{ k }   \left( e^{\ii k \theta^1} + e^{\ii k \theta^2 } \right) e^{\ii l \theta^3 },
\end{equation}
where $ \theta^\mu \in [ 0, 2 \pi ) $ are  canonical angle coordinates on $ \mathbb{ T }^3 $. Figure~\ref{figMixing3TorusX} shows $ \psi $ as a time series for the underlying irrational flow $ w = \sum_{i=1}^3 w^i \frac{\partial^i }{ \partial\theta^i} $ with frequencies $ ( \alpha_1, \alpha_2, \alpha_3 ) =  (w^1, w^2, w^3) = ( 1, 5^{1/2}, 10^{1/2} ) $. 
   
\subsection{\label{secTimeChangeForecasting}Dimension reduction and nonparametric forecasting with time change}

Traditionally, time change is used as a technique to improve mixing and stochasticity of dynamical systems. Here, our approach is to use time change as a regularization tool to reduce mixing and improve the spectral properties of the Koopman operators for dimension reduction and nonparametric forecasting. In essence, we will use the eigenfunctions of the time-changed system to decompose the original dynamical system into a collection of non-autonomous oscillators with variable frequency (as opposed to simple harmonic oscillators in pure point spectrum systems). We will also use these eigenfunctions to perform a vector field decomposition analogous to Theorem~\ref{lemmaVDecomp}, but the vector field components in this case will be non-commuting. Finally, we will generalize the forecasting methods of Section~\ref{secForecasting} to the time-changed framework. 

In what follows, $ ( M, \mathcal{ B }, \mu, \Phi_t ) $ will be an ergodic dynamical system with vector field $ v $ generating the observed data, and related to the system $ ( M, \mathcal{ B }, \nu, \Psi_t ) $ with  vector field $ w = \frac{ 1 }{ \psi } v $ via the smooth time-change function $ \psi $, bounded away from zero. As previously, we denote the Koopman operator and skew-adjoint generator of the first system by $ U_t $ and $ \tilde v $, respectively, and we also introduce the skew-adjoint generator $ \tilde w : D( \tilde w ) \mapsto L^2(M,\nu) $, $ D( \tilde w ) \subset L^2( M,\nu) $, associated with the second system. We use the notations $ \langle \cdot, \cdot \rangle_\nu $ and $ \lVert \cdot \rVert_\nu $ to distinguish the inner product and corresponding norm on $ L^2( M,\mu) $. Note that by the properties of $ \psi $, $ L^2( M,\mu) $ and $ L^2(M, \nu) $ (resp.\ $ D( \tilde v) $ and $ D( \tilde w ) $) are canonically isomorphic. As a result, we can identify $ \tilde v $ with  $ T_\psi \circ \tilde w $, where $ T_\psi $ is the bounded multiplication operator on $ L^2(M,\mu) $ by $ \psi $.

In what follows, our interest is in the case that $ \tilde w $ has pure point spectrum with basic frequencies $ \{ \Omega_i \}_{i=1}^m $ and the associated smooth  orthonormal eigenfunctions $ \{ \zeta_i \}_{i=1}^m $ on $ L^2(M,\nu) $. We denote the eigenvalues and eigenfunctions of $ \tilde w $ by $ \omega_k = \sum_{i=1}^m k_i \Omega_i $ and $ z_k = \prod_{i=1}^m \zeta_i^{k_i} $ with $ k = ( k_1, \ldots, k_m ) \in \mathbb{ Z}^m $. Because $ L^2( M, \mu ) \simeq L^2( M, \nu ) $, we can expand any $ f \in L^2( M,\mu) $ as $ f = \sum_k \hat f_k z_k $, where $ \hat f_k = \langle z_k, f \rangle_\nu = \langle  z_k, f / \varrho \rangle $. Equivalently, we have a Fourier operator $ \mathcal{ S } : L^2( M,\mu ) \mapsto \ell^2 $ such that $ \mathcal{ S } f = \hat f := ( \hat f_k )_k $. This operator is bounded, invertible, and with a bounded inverse, but unlike $ \mathcal{ U } $ from Section~\ref{secErgodicity}, it is not unitary. In particular, given $ f $ as above and $g \in L^2(M,\mu) $ with $ \hat g = \mathcal{S} g $, we have $ \langle f, g \rangle = \langle \hat f, \mathcal{ G } \hat g \rangle_{\ell^2} $, where $ \mathcal{ G } = ( \mathcal{ S} \mathcal{ S}^* )^{-1} $. In the canonical orthonormal basis of $ \ell^2 $, denoted here by $ \{ e_k \} $, the matrix elements of $ \mathcal{ G } $ are $ \langle e_k, \mathcal{ G } e_l \rangle_{\ell^2} = \langle z_k, z_l \rangle = \langle z_k, \varrho z_l \rangle_\nu $.  

Under $ \mathcal{ S } $, $ \tilde w $ becomes the skew-adjoint unbounded multiplication operator $  T_{\ii \omega} =  \mathcal{ S } \tilde w \mathcal{ S }^{-1} $ with $ T_{\ii \omega} ( \hat f_k )_k = \ii ( \omega_k \hat f_k )_k $, defined on the dense domain $ D( T_{\ii \omega} ) \subset \ell^2 $ such that $ D( T_{\ii\omega} ) = \{ ( \hat f_k )_k \in \ell^2 \mid (\omega_k f_k )_k \in \ell^2\} $. However, $ \tilde v $ is no longer transformed into a multiplication operator; instead, we have $ \hat v = \mathcal{ S } \tilde v \mathcal{ S }^{-1} = \mathcal{ H } T_{\ii\omega} $, where $ \mathcal{H} : \ell^2 \mapsto \ell^2 $ is the bounded operator $ \mathcal{H} = \mathcal{ S } T_\psi \mathcal{ S}^{-1} $. This result in conjunction with~\eqref{eqUEvolve} implies that the  Fourier representation $ \hat f_t = \mathcal{ S } U_t f $ of an observable $ f \in D( \tilde v)  $ is governed by the equation
\begin{equation}
  \label{eqUEvolveFourierTimeChange}
  \frac{ d\ }{ dt } \hat f_t = \hat v  \hat f_t.
\end{equation}
The following Proposition characterizes the evolution of the $ z_k $ under the dynamical system $ ( M, \mathcal{ B }, \mu, \Phi_t ) $.      

\begin{prop} \label{propPiProjT}Assume that $ \tilde w $ has pure point spectrum with  eigenfrequencies  $ \{ \omega_k \} $ and corresponding orthonormal  eigenfunctions $ \{ z_k \} $ on $ L^2( M, \nu ) $. Then, under evolution by the Koopman group $ \{ U_t \}_{t\in\mathbb{R}} $ on $ L^2( M,\mu) $, the following hold: 

  (i) The $ z_k $ evolve according to the equation
  \begin{displaymath}
    \frac{ d\ }{ dt } U_t z_k= \ii \omega_k T_{\psi_t}  U_t z_k,
  \end{displaymath}
  where $ T_{\psi_t} $ is the bounded multiplication operator on $ L^2( M,\mu) $ by  $ \psi_t = \psi \circ \Phi_t $. 

  (ii) For $ \mu $-a.e.\ $ a \in M $, the time series $ t \mapsto \tilde z_k( t ) = z_k( \Phi_t ( a ) ) $ evolves as a nonautonomous, variable-frequency harmonic oscillator, 
  \begin{equation}
    \label{eqZSHOTimeChange}
    \frac{ d \tilde z_k( t ) }{ d t} = \ii \omega_k \tilde \psi( t ) \tilde z_k( t ), \quad \tilde \psi( t ) = \psi_t(  a ), \quad \tilde z_k( t ) = e^{\ii \omega_k \int_0^t \tilde \psi( t' ) \, dt' } \tilde z_k( 0 ). 
  \end{equation} 
\end{prop} 

\begin{proof}
  (i) Since $ z_k \in D( \tilde w ) \simeq D( \tilde v ) $, it follows from~\eqref{eqUEvolve} that
  \begin{displaymath}
    \frac{ d\ }{ dt } U_t z_k = U_t v z_k  = U_t( \psi w z_k ) = U_t( \psi )  U_t( \ii \omega_k  z_k ),
  \end{displaymath}
  leading to the claim.
  
  (ii) The result follows directly from (i).

\end{proof}

Proposition~\ref{propPiProjT} indicates that if the dynamical system generating the data is related to a system with pure point spectrum by a time change, we can still use the Koopman eigenfunctions of the pure point spectrum system for dimension reduction and forecasting, retaining a number of the desirable features of the techniques of Section~\ref{secPurePointSpectra}. In particular, we can still define projection maps $ \pi_i : M \mapsto \mathbb{ C } $ as in~\eqref{eqPiProj} using a generating set of eigenfunctions $ \{ \zeta_i \}_{i=1}^m $ of $ \tilde w $ with minimal Dirichlet energy. Under these maps, the image of the state in $ M $ still evolves as as an oscillator with constant amplitude, but now the frequency is nonconstant; compare, in particular, \eqref{eqZSHO} with~\eqref{eqZSHOTimeChange}. Due to the time-dependence of the oscillatory frequency, the vector field $ v $ will in general not be projectible under $ \pi_{i*} $. However, the composite map $ \pi = ( \pi_1, \ldots \pi_m ) $ provides an embedding  of $M $ into $ \mathbb{ C }^m $, and $ v $ is projectible under $ \pi_* $.  The following Lemma, which we state without proof, describes the analog of the vector field decomposition in Theorem~\ref{lemmaVDecomp} for time-changed systems.

\begin{lemma} \label{lemmaVDecompT} Let $ \tilde w$ be as in Proposition~\ref{propPiProjT}, and let $ w = \sum_{i=1}^m w_i $ be the associated vector field decomposition from Theorem~\ref{lemmaVDecomp}. Then, $ v $ admits the decomposition $ v = \sum_{i=1}^m v_i $ into the nowhere-vanishing, linearly independent, $ \mu $-preserving vector fields $ v_i = \psi w_i $.
\end{lemma} 

Unlike the $ w_i $, the vector fields $ v_i $ are  non-commuting due to the presence of $ \psi $, which means that in general we cannot view them as generators of independent dynamical processes. Nevertheless, the $ v_i $ can still be reconstructed in data space using a pushforward operation analogous to~\eqref{eqVPushZeta}, with the difference that the observation map $ F $ is now expressed using the Fourier operator $ \mathcal{ S }$ instead of $ \mathcal{ U }$; that is, 
\begin{equation}
  \label{eqVPushT}
   V_i := F_* v_i = v_i( F ) = \psi \mathcal{ S }^{-1} T_{\ii \omega^{(i)} } \hat F =\sum_k \ii k_i \Omega_i \psi \hat F_k z_k, \quad V = F_* v = \sum_{i=1}^m V_i,
\end{equation}
where $ \hat F = ( \hat F_k )_k = \mathcal{ S } F $, and $ T_{\ii \omega^{(i)} } : D(T_{\ii \omega^{(i)} }) \mapsto \ell^2 $ is the unbounded multiplication operator by $ ( \ii k_i \Omega_i )_k $.   

We proceed similarly to modify the forecasting scheme of Section~\ref{secForecasting}. As in that Section, we are interested in  the density $ \rho_t = d \mu_t / d \mu $ of the time-dependent measure $ \mu_t = \Phi_{t*} \mu $ relative to $ \mu $ given the initial data $ \rho_0 $. The evolution of this density is still governed by the Perron-Frobenius operator, i.e., we have $ \rho_t =U^*_t \rho_0 = U_{-t} \rho_0 $, but in this case we compute the non-unitary transform $ \hat \rho_t = \mathcal{ S } \rho_t $, and seek the solution of 
\begin{equation}
  \label{eqRhoTimeChange}
  \frac{ d \hat \rho_t }{ d t } = -\hat v \hat \rho_t.
\end{equation}
This equation is the analog of~\eqref{eqUEvolveFourierTimeChange} for the Peron-Frobenius operator. Given a solution $\hat \rho_t $, the time-dependent expectation value of an observable $ f = \mathcal{S }^{-1} \hat f  $ can be computed from an analogous expression to~\eqref{eqMeanForecast0}, with some modification due to the lack of unitarity of $ \mathcal{S } $, viz.
\begin{equation}
  \label{eqMeanForecastT}
  \bar f_t = \mathbb{ E }_{\mu_t} f = \int_M f \rho_t d\mu = \langle \rho_t, f \rangle = \langle \hat \rho_t, \mathcal{ G } \hat f \rangle_{\ell^2}.
\end{equation}

Unlike Section~\ref{secForecasting}, a closed-form solution of~\eqref{eqRhoTimeChange} is generally not available due to the presence of the non-multiplicative operator $ \mathcal{H} $ in $ \hat v $. Nevertheless, we can obtain a sequence $ \rho_{0,t}, \rho_{1,t}, \ldots $ of approximate solutions by forming the sequence of bounded operators $ \hat v_l = \Pi_l \hat v \Pi_l $, where $ \Pi_l $ are orthogonal projectors to the $ ( 2 l + 1 ) $-dimensional subspaces of $ \ell^2 $ spanned by $ \mathcal{ S } z_{k_1}, \ldots, \mathcal{ S }_{k_l} $, where the $ z_{k_i} $ are ordered in order of increasing Dirichlet energy as in the schemes of Section~\ref{secGalerkinImplementation} (note that we use $ 2 l + 1 $ eigenfunctions since the Dirichlet energies of $ z_k $ and $ z^*_k = z_{-k} $ are equal). We then compute 
\begin{equation}
  \label{eqRhoLT}
  \hat \rho_{l,t} = e^{-t \hat v_l } \hat \rho_{l,0}, \quad e^{-t \hat v_l } = \sum_{k=0}^\infty \frac{ t^k }{ k! } \hat v_l^k, 
\end{equation}
where the series expansion for the exponential is well-defined since $ \hat v_l $ is bounded. Note that~\eqref{eqRhoLT} describes the evolution of $ l $ linearly coupled harmonic oscillators. If the vector field $ - v $ is replaced by the advection diffusion operator $ L'_\varepsilon = - v - \varepsilon \upDelta_h $ (i.e., the dual to $ L_\varepsilon $ in~\eqref{eqL}) for some $ \varepsilon > 0 $, then it follows from classical results in semidiscrete approximation schemes for linear parabolic partial differential equations \citep[][]{FujitaSuzuki91} that the corresponding solutions $ \rho_{l,t} = \mathcal{ S }^{-1} \hat \rho_{l,t} $ converge as $ l \to \infty $ to the contraction semigroup solution $ e^{t L'_\varepsilon} \rho_0 $ associated with $ L'_\varepsilon $. A study on prediction with $ L'_\varepsilon $, including issues related to the choice of diffusion operator $ \upDelta_h $ and/or regularization parameter $ \varepsilon $, is beyond the scope of this work. In Section~\ref{secApplicationsT}, we will see that predictions with the raw $ \rho_{l,t} $ from~\eqref{eqRhoLT} perform well at least over short to medium lead times, though there are biases on longer lead times which may be related to the absence of diffusion in our scheme.      

\subsection{\label{secImplementationT}Data-driven implementation}

The analysis in Section~\ref{secTimeChangeForecasting} shows that time-change transformations can extend the applicability of Koopman eigenfunction techniques to certain systems that do not posses nonconstant eigenfunctions. Of course, the systems in question have special structure, namely they are related to pure point spectrum systems via time change. However, even in such cases, taking advantage of this special structure from data is challenging without prior information about the time-change function $ \psi $. In this section, we present a time-change scheme that employs  an empirically accessible time-change function. This method can recover $ \psi $ in special cases, and thus can transform the system into a new system which is more amenable to Koopman eigendecomposition. 

Let $ \xi = \lVert v \rVert_g =\sqrt{ g( v, v ) } $ be the norm of the dynamical vector field with respect to the ambient space Riemannian metric $ g $. By the assumptions stated in section~\ref{secErgodicity}, this quantity is smooth and non-negative, and here we also assume that it is bounded away from zero so that it can be used as a time change function to obtain the smooth vector field $ \hat w = \xi^{-1} v  $. In the special case that $ v $ is indeed related to a system with pure point spectrum by a time change with function $ \psi $, then it follows from Lipschitz equivalence of Riemannian metrics on compact manifolds that there exists a constant $ 0 < C < \infty $ such that $ \xi/ C \leq \psi \leq C \xi $. Moreover, there exits a metric $ g $ (e.g., the flat metric on the torus) such that $ \xi = \psi $. 

A key property of $ \xi $ is that it can be approximated from time-ordered data using finite differences. In particular, we have $ \xi^2 = V \cdot V $, where $ V = F_* v $ is the pushforward of $ v $ in data space, and we can approximate $ V $ via finite differences as described in Section~\ref{secErgodicity}. Our approach is to use $ \xi $ as an empirically accessible time-change function, and consider $ \hat w $ as a candidate vector field with well-behaved eigenfunctions on the Hilbert space $ L^2( M, \hat \nu ) $ with $ d \hat \nu/ d\mu = \xi  / ( \int_M \xi \, d\mu ) $. In what follows, we will carry out the vector field decomposition in Theorem~\ref{lemmaVDecomp} and the statistical forecasts in~\eqref{eqMeanForecastT} and~\eqref{eqRhoLT} using $ \xi $ as the time-change function instead of $ \psi $.   

Proceeding in direct analogy with the methods of Section~\ref{secGalerkin}, we compute approximate eigenfunctions of $ \hat w $ by approximating this operator in a basis of eigenfunctions of a Laplace-Beltrami operator for a suitable Riemannian metric $ \hat h $. In this case, the metric should have volume measure equivalent to $ \hat \nu $, so we consider the conformal transformation $ \hat h = ( \xi \sigma  )^{2/m} g $, where, as before, $ \sigma $ is the density of the invariant measure $ \mu $ relative to the Riemannian measure $ \vol_g $ of the ambient space metric. We approximate the Laplace-Beltrami operator for this geometry using diffusion maps with the modified kernel (cf.~\eqref{eqKVB}) 
\begin{equation}
  \label{eqKVBT}\hat K_\epsilon( x, y ) = \exp\left( - \frac{ \lVert x - y \rVert^2 }{ \epsilon  ( \xi( x ) \hat\sigma_\epsilon( x ) )^{-1/m} ( \xi( y) \hat\sigma_\epsilon( y ))^{-1/m }( y )  } \right)
\end{equation}
and the same normalization as Algorithm~\ref{algBasis}. Note that in this case the sampling density $ q = d\mu / d\hat \nu$ (i.e., the density of the measure $ \mu $ from which data are sampled relative to the measure $ \hat \nu $ with respect to which we seek to build an orthonormal basis) is nonuniform, but the effects of $ q $ are removed as $ \epsilon \to 0 $ via normalization as described in Section~\ref{secDataDrivenBasis}. Applying Algorithm~\ref{algBasis} with the kernel in~\eqref{eqKVBT}, we obtain an orthonormal basis of $ L^2( M, \hat \nu ) $ from the Laplace-Beltrami eigenfunctions $ \{ \hat \phi_i \} $ associated with $ \hat h $. We then construct the associated basis $ \{ \hat \varphi_i \} $ of $ H_1( M, \hat \nu ) $ following~\eqref{eqBasisH1}, and compute approximate eigenvalues and eigenfunctions of $ \hat w $ from the eigenvalue problem for the advection-diffusion operator 
\begin{equation}
  \label{eqLHat}
  \hat L_{\varepsilon} = \hat w -  \varepsilon \upDelta_{\hat h} 
\end{equation}
with an analogous algorithm to Algorithm~\ref{algGenerators}. 

Our vector field decomposition and statistical forecasting techniques with time change can then be carried out using Algorithms~\ref{algVDecomp} and~\ref{algStatisticalForecast} with appropriate modifications to take into account the nonunitarity of the Fourier operator $\mathcal{S}$.  Instead of giving full listings for these modified algorithms, we indicate below the required changes.   
\begin{itemize}
\item The eigenfunctions $ \{ u_i \} $ and eigenvalues $ \{ \gamma_i \} $ are for the time-changed operator $ \hat L_{\varepsilon} $. Similarly, the discrete inner product weights $ w_i $ are for the diffusion eigenfunction basis $ \{ \hat \phi_i \} $.   
  \item In Step~3 of Algorithm~\ref{algVDecomp}, the vector field components $ V_i $ are now given by $ V_i = \ii \hat F G^{-1} \bar \Omega_i z^\top \diag \xi $ in accordance with~\eqref{eqVPushT}, where $ \diag \xi $ is an $ N \times N $ diagonal matrix whose $ i $-th diagonal entry is set to the value of the time-change function at the $ i $-th data sample, $ \xi(x_i) $. 
  \item In Algorithm~\ref{algStatisticalForecast}, we also compute the $ n \times n $ matrices $ G' = z^\dag z / N $  and $ H =z^\dag \diag w \diag \xi z $  approximating the bounded operators $ \mathcal{ G}  $ and  $ \mathcal{ H } $, respectively, in the $ \{ e_k \} $ basis of $ \ell^2 $. 
  \item In Step~3(a) of Algorithm~\ref{algStatisticalForecast}, the expansion coefficients for the density are now given by $ \tilde \rho( t_i ) = \exp( - \ii G^{-1} H \diag \omega t_i ) \tilde \rho_0 $. Similarly, in Step~3(c), the expectation value of the prediction observable now becomes $ \bar f( t_i ) = \hat f ( G^{-1} G' \tilde \rho( t_i ) )^\dag $.
\end{itemize}
  
\subsection{\label{secApplicationsT}Applications}

\subsubsection{\label{secMixing3Torus}Mixing flow on the 3-torus}

As a first application of our dimension reduction and nonparametric forecasting techniques with time change, we study the mixing system on $ M = \mathbb{ T }^3 $ in \cite{Fayad02} with the time-change function $ \psi $ in~\eqref{eqMixing3Torus}, using the underlying irrational flow with angular frequencies $ ( \alpha_1, \alpha_2, \alpha_3 ) = ( 1, 5^{1/2}, 10^{1/2} ) $. For this set of experiments, we generated a time series consisting of $ N = \text{512,000} $ samples in $ \mathbb{ R }^6 $ taken at a timestep $ T = 0.01 $ with the canonical embedding $ F( a ) = ( x^1, \ldots, x^6 ) = ( \cos \theta^1, \sin \theta^1, \cos \theta^2, \sin \theta^2, \cos \theta^3, \sin \theta^3 ) $. With this embedding, the torus inherits the flat metric $ g $, and we have $ \lVert v \rVert_g = \psi ( \alpha_1^2 + \alpha_2^2 + \alpha_3^2 )^{1/2} $. Therefore, in this case, $ \lVert v \rVert_g = \xi $ is  proportional to the true time-change function $ \psi $, and the method described in Section~\ref{secImplementationT} is guaranteed to recover an irrational flow on $ \mathbb{ T }^3 $. This application is therefore a ``best-case scenario''  for the techniques of Section~\ref{secImplementationT} to perform well, but nevertheless it remains challenging for data-driven eigendecomposition and forecasting techniques due to mixing dynamics.   

Figure~\ref{figMixing3TorusX} shows representative time series from this system together with the vector field  norm $ \lVert v \rVert_g = \xi \propto \psi  $. Qualitatively, $ \xi $ exhibits a series of intermittent spikes whose amplitude is modulated by a low-frequency envelope. These variations in the speed of the flow, whose maxima and minima differ by more than a factor of 10, produce phase-modulated wavetrains in the observed time series (see Fig.~\ref{figMixing3TorusX}(a, b)). As a result, the invariant measure $ \mu $ has highly nonuniform density relative to the Haar measure on the 3-torus.

\begin{figure}
  \centering\includegraphics{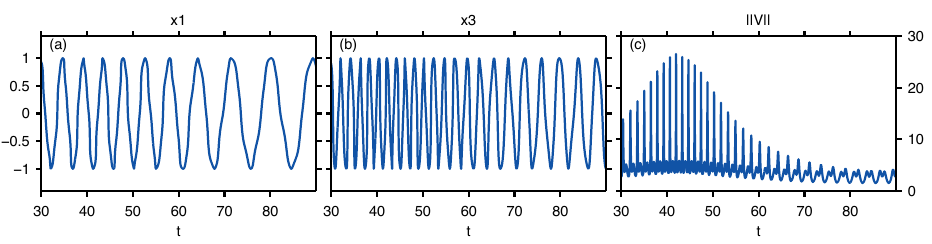}
  \caption{\label{figMixing3TorusX}Time series from the mixing system on the 3-torus. (a, b) Observation map components $ x^1 $ and $ x^3 $, respectively, in the flat embedding of the 3-torus in $ \mathbb{ R }^6 $; (c) magnitude $ \lVert V \rVert = \lVert v \rVert_g $ of the dynamical vector field with respect to the ambient-space metric. In this case, $ \lVert V \rVert $ is proportional to the time-change function $ \psi $ producing mixing.}  
\end{figure}

For comparison, we begin by solving the eigenvalue problem for the generator using the method in Section~\ref{secGalerkin} which does not involve time change. Because the system is mixing and $ \tilde v $ has no eigenfunctions, the behavior of the eigenfunctions of the regularized operator $ L_{\varepsilon} $ is expected to deviate significantly from the simple harmonic oscillator patterns arising in systems with pure point spectra, and as shown in Fig.~\ref{figMixing3TorusX} this is indeed the case. The eigenfunctions in Fig.~\ref{figMixing3TorusX}, which were computed for the diffusion regularization parameter $ {\varepsilon} = 3 \times 10^{-4} $, include one with a near-periodic time series (Fig.~\ref{figMixing3TorusZNoT}(a)) that appears to capture the low-frequency modulating envelope of the time change function (Fig.~\ref{figMixing3TorusX}(c)). However, other eigenfunctions (Fig.~\ref{figMixing3TorusZNoT}(b, c)) exhibit strong amplitude modulation---a manifestation of the fact that advection and diffusion interact non-trivially in $ L_{\varepsilon} $ in the presence of mixing, even for small $ \varepsilon $.  Note that applying no regularization ($\varepsilon=0$) results in significantly less coherent patterns than those in Fig.~\ref{figMixing3TorusZNoT}.   

\begin{figure}
  \centering\includegraphics[scale=.95]{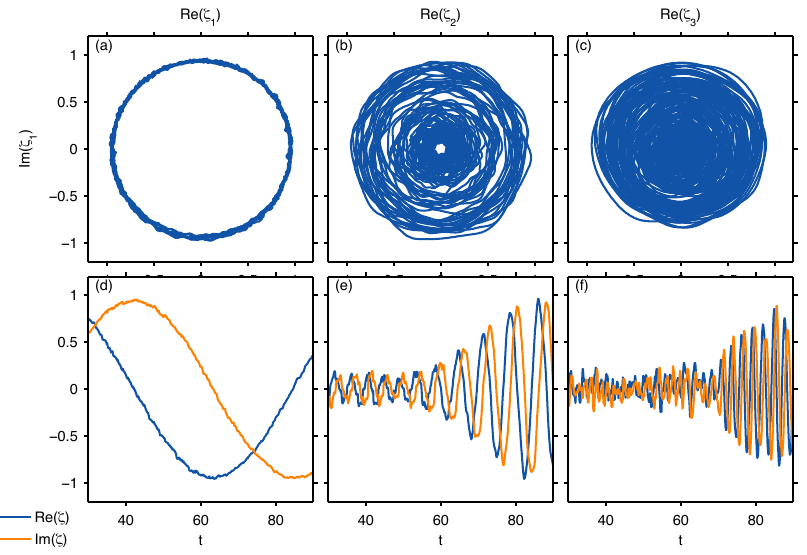}
  \caption{\label{figMixing3TorusZNoT}Eigenfunctions of the regularized generator $ L_\varepsilon $ for the mixing system on the 3-torus without time change. (a) Near-periodic pattern; (b, c) amplitude-modulated patterns.}
\end{figure}

Next, we discuss the results obtained via the time-change technique of  Section~\ref{secImplementationT}. As shown in Fig.~\ref{figMixing3TorusZ}, the generating eigenfunctions $ \{ \zeta_1, \zeta_2, \zeta_3 \} $ from this method lie on the unit circle to a good approximation (though with somewhat less fidelity than the eigenfunctions for the 2-torus system in Fig.~\ref{figTorusZ}), and, as expected from Proposition~\ref{propPiProjT}, the corresponding time series have the structure of phase-modulated wavetrains. The basic frequencies associated with the generating eigenfunctions are $ \{ \Omega_1, \Omega_2, \Omega_3 \} = \{ 0.5587, 0.2498, 0.7902 \} $; these values agree to within three significant figures with the frequencies $ ( \hat \alpha_1, \hat \alpha_2, \hat \alpha_3 ) = ( 1, 5^{1/2}, 10^{1/2} ) / 4 \approx ( 0.2500, 0.5590, 0.7906 ) $, of the irrational flow $ \hat w = \lVert v \rVert_g^{-1} v = \sum_{i=1}^3 \hat \alpha_i \frac{ \partial \  }{ \partial \theta^i } $. Based on this identification, we expect the eigenfunctions corresponding to the $ \Omega_i $ to have the structure $ \zeta_i = e^{\ii \theta_i} $ and to have equal Dirichlet energies $ E(\zeta_i) $ in the flat metric. Indeed, the numerical Dirichlet energies for the eigenfunctions in Fig.~\ref{secImplementationT} are $ E(\zeta_1) = 1.264 $, $ E(\zeta_2) = 1.266 $, and $ E(\zeta_3) = 1.28 $. The real parts of the numerical eigenvalues are $ \Real \gamma_i \approx 5 \times { 10 }^{- 4} $ in all three cases, indicating that the effects of diffusion in the regularized operator $ \hat L_{\varepsilon} $ are minimal after time change.

\begin{figure}
  \centering\includegraphics[scale=.9]{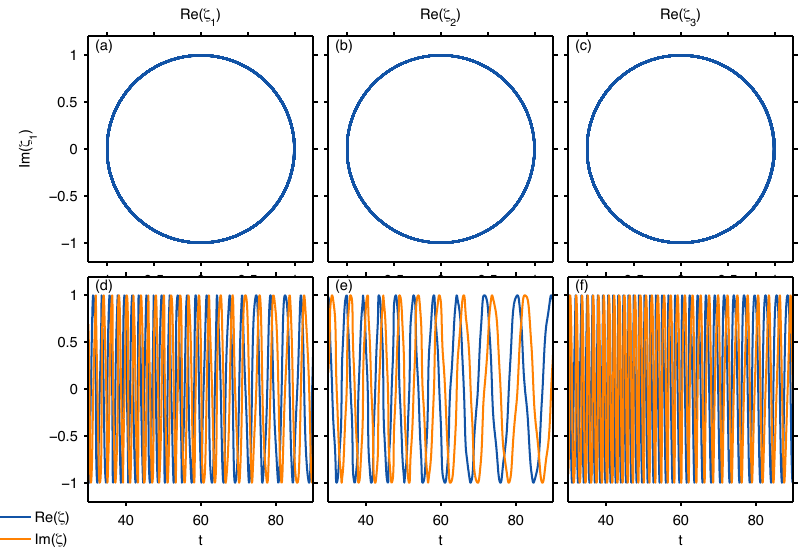}
  \caption{\label{figMixing3TorusZ}Generating eigenfunctions $ \zeta_i $ for the mixing system on the 3-torus with time change. (a--c) Mapping of the 3-torus to the circles $ ( \Real \zeta_i, \Imag \zeta_i ) $ on the complex plane; (d--f) time series of the real and imaginary parts of the eigenfunctions exhibiting time-dependent frequency (cf.\ the eigenfunctions in Fig.~\ref{figTorusZ} for the 2-torus system with pure point spectrum).}
\end{figure}

Figure~\ref{figMixing3TorusV} displays the vector field decomposition from Lemma~\ref{lemmaVDecompT} for this system using a moderately small spectral order parameter $ l = 5 $ (see Algorithm~\ref{algVDecomp}), i.e., $ n = ( 2 l + 1 )^3 = 1331 $ Koopman eigenfunctions in total. The vector field components are reconstructed in the six-dimensional data space, and then projected to three-dimensional periodic boxes for visualization. At this spectral order of approximation, the reconstruction error in $ \sum_{i=1}^3 V_i $ is $ 1 \% $ in $ \mathbb{ R }^6 $, but note that the visualizations in Fig.~\ref{figMixing3TorusV} are subject to additional projection errors, and therefore appear noisier than the native vector fields in $ \mathbb{ R }^6 $. Nevertheless, the plots in Fig.~\ref{figMixing3TorusV}(b--d) illustrate clearly that $ V_1 $, $ V_2 $, and $ V_3 $ generate flows along the $ \theta^2 $, $ \theta^1 $, and $ \theta^3 $ directions, respectively, which is consistent with the identification with the $ \hat w $ system made earlier.
  
\begin{figure}
  \centering\includegraphics{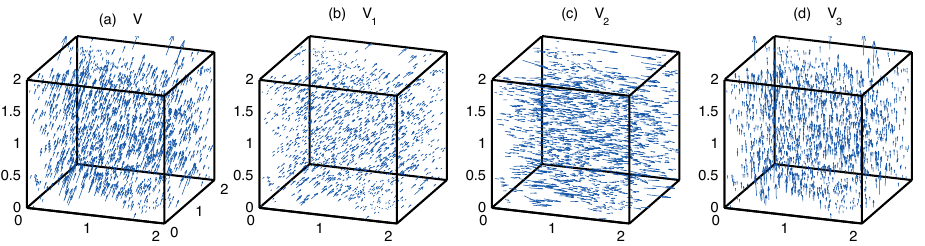}
  \caption{\label{figMixing3TorusV}Vector field decomposition for the mixing dynamical system on the 3-torus. (a) Full vector field $ v $; (b--d) vector field components $v_i $ from Lemma~\ref{lemmaVDecompT}. All vector fields are visualized in a periodic box with the $ x $, $ y $, and $ z $ axes corresponding to the normalized angles $ \theta^1 / \pi $, $ \theta^2/\pi $, and $ \theta^3 / \pi $, respectively, on $ \mathbb{ T }^3 $. To create these plots, the vector fields were first reconstructed in data space, $ \mathbb{ R }^6 $, via~\eqref{eqVPushT} and then projected to the periodic box via the inverse derivative map $ \partial \theta^i / \partial x^j $. This projection step is needed because~\eqref{eqVPushT}  cannot be applied to reconstruct the vector fields directly in the periodic box (in that case, the map $ F : \mathbb{ T }^3 \mapsto \mathbb{ R }^3 $ would be discontinuous). Because $ \partial \theta^i / \partial x^j $ has singularities, the reconstructed arrow plots are more noisy than the natively reconstructed vector fields in $ \mathbb{ R }^6 $.}
\end{figure}

Next, we turn to nonparametric forecasting of this system. We set the initial probability measure $ \mu_0 $ to a measure with isotropic circular Gaussian density $ \sigma $ relative to the Haar measure on $ \mathbb{ T }^3 $, viz.\ $ \sigma( \theta^1, \theta^2, \theta^3 ) = \exp( \kappa \sum_{i=1}^3 \cos( \theta^i - \bar \theta^i ) ) / ( I_0( \kappa ) )^3 $ which is the three-dimensional analog of~\eqref{eqVonMises}. Taking into account the density of the invariant measure relative to the Haar measure, our initial probability measure has  density $ \rho_0 = d \mu_0 / d \mu = \sigma  \psi / ( \int_M \psi \, d\mu) $ relative to $ \mu $. Here, we use the location and concentration parameters $ \bar \theta^i = \pi $ and $ \kappa = 30 $. Also, we take the components $ \{ x^1, x^3, x^5 \} $ of the observation map as our forecast observables.  

Figure~\ref{figMixing3TorusForecast} displays forecast results for the mean and standard deviation of these observables obtained via the method in Section~\ref{secImplementationT} for $ l = 13 $ ($n=\text{19,683}$) and an ensemble forecast with 10,000 particles using the perfect model. Because this system is mixing, expectation values with respect to the time-dependent probability measures $ \mu_t $ converge to expectation values with respect to $ \mu $. This late-time behavior is clearly evident in the ensemble forecast results, but it is not well emulated by the nonparametric model consisting of a finite collection of coupled oscillators. That is, the late-time behavior of this mixing deterministic system involves the generation of arbitrarily small lengthscales in the densities $ \rho_t $, and as a result any diffusion-free algorithm that attempts to replicate this behavior via a finite collection of deterministic oscillators will ultimately fail at late times. In Fig.~\ref{figMixing3TorusForecast}, the nonparametric model is able to accurately track the mean and standard deviation of the ensemble forecast over times less than the the equilibrium relaxation time of the true model, but eventually develops spurious oscillations. The latter are possibly due to aliasing effects in the finite-bandwidth representation of $ \rho_t $, and could be suppressed at a fixed time $ t $ by increasing $ l $. However, unless diffusion is added, failure to relax to the correct equilibrium should always occur at long-enough times for fixed $ l $.  Note that the consequences of this shortcoming may not be particularly detrimental if the timescale at which the nonparametric model deviates from the perfect model is comparable or exceeds the intrinsic predictability limits for the given initial distribution and forecast observable (as is the case in Fig.~\ref{figMixing3TorusForecast}).

\begin{figure}
  \centering\includegraphics[scale=.75]{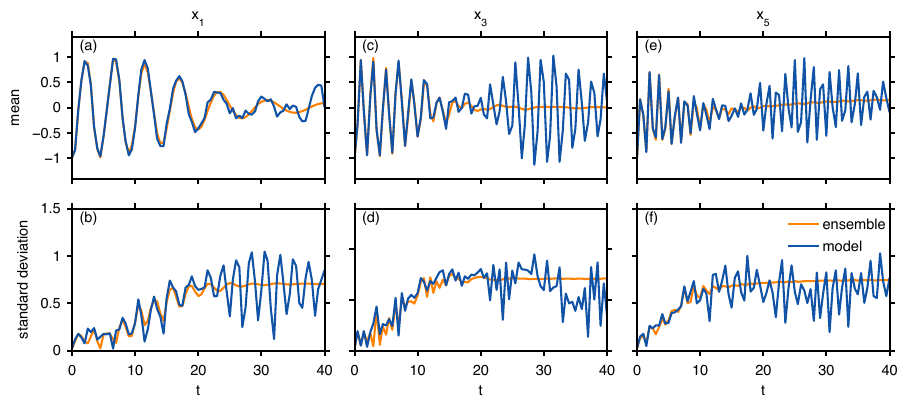}
  \caption{\label{figMixing3TorusForecast}Nonparametric and ensemble forecasts of the components $x^1$  (a, b), $x^3 $ (c, d), and $ x^5 $ (e, f) of the flat embedding into $ \mathbb{ R }^6 $ for the mixing dynamical system on the 3-torus. The initial probability measure has circular Gaussian density relative to the Haar measure with location and concentration parameters $ ( \bar \theta^1, \bar \theta^2, \bar \theta^3 ) = ( \pi, \pi, \pi ) $ and $ \kappa = 30 $, respectively. The nonparametric forecast is performed via Algorithm~\ref{algStatisticalForecast}, modified for time change as described in Section~\ref{secImplementationT}. The ensemble forecast is based on 10,000 independent samples drawn from the initial probability measure and evolved using the perfect model.}
\end{figure}

\subsection{\label{secFixedPointTorus}Ergodic flow on the 2-torus with a fixed point}

In this set of experiments, we consider the dynamical system on the 2-torus with the vector field $ v = \sum_{i=1}^2 v^i \frac{ \partial \  }{ \partial \theta^i } $, where
\begin{equation}
  \label{eqFixedPointTorus}
  v^1 = v^2 + ( 1 - \alpha ) ( 1 - \cos \theta^2 ), \quad v^2 = \alpha( 1 - \cos( \theta^1 - \theta^2 ) ),
\end{equation}
and $ \alpha $ is an irrational frequency parameter. Originally introduced by Oxtoby \cite{Oxtoby53},  flows of this class have a fixed point $ b $ with coordinates $ \theta^1 = \theta^2 = 0 $,  and are ergodic for the Haar measure on the 2-torus.\footnote{The convention in \cite{Oxtoby53} is that the frequency parameter lies in the unit interval; the cases with $ \alpha > 1 $ are equivalent to systems with frequency $ 1/\alpha $ up to an unimportant change of sign in $ v^1 $ and a constant time change.} Indeed, it is straightforward to check that $ \divr_\mu v $ vanishes everywhere on $ \mathbb{ T }^2 $ for the Haar measure, which is a necessary and sufficient condition for the flow to be $ \mu $-preserving. Due to the presence of the fixed point, this system is not uniquely ergodic, but the Haar measure is its only invariant Borel probability measure for which $ \mu( b ) = 0 $. Besides the Haar measure, its only other ergodic Borel probability measure is the trivial measure with $ \mu( b ) = 1 $. The system is topologically conjugate to a class of dynamical systems on the 2-torus called Stepanoff flows. The latter have the structure of a time-changed linear flow as in Section~\ref{secTimeChangeOverview}, but contain a single fixed point where the time change function diverges. Systems of this type are known to be topologically mixing, but we have not been able to find results on their measure-theoretic mixing properties apart from a conjecture in \cite{Oxtoby53} that a class of Stepanoff flows that includes~\eqref{eqFixedPointTorus} is mixing. In summary, the system specified in~\eqref{eqFixedPointTorus} does not in itself have the structure of a time changed flow, but is connected to a (singular) time changed-system system via a continuous transformation. Our objective in this section is to demonstrate how the time-change methods of Section~\ref{secImplementationT} behave in this more general context.     

Qualitatively, the orbit of a point $ a \neq b $ under~\eqref{eqFixedPointTorus} will pass by the fixed point at arbitrarily small distances, but by incompressibility, the trajectory develops ``bumps'' and circumvents the fixed point when it comes close to it. Because the speed of the flow can become arbitrarily slow near the fixed point, the time series of observables of this system exhibit complex behavior over a broad range of timescales, with intervals of rapid evolution separated by quasi-stationary periods. This behavior is illustrated in Figure~\ref{figFixedPointTorusX} for the frequency $ \alpha = 20^{1/2} $ and the standard (flat) embedding of the 2-torus in $ \mathbb{ R }^4 $, $ F( a ) = ( \cos \theta^1, \sin \theta^1, \cos\theta^2, \sin \theta^2 ) $. In what follows, we discuss dimension reduction and forecasting results for this system and observation map using a time series of 128,000 samples initialized at $ ( \theta^1, \theta^2 ) = ( \pi/2, \pi/2 ) $ and sampled at a timestep $ T = 0.01 $.

\begin{figure}
  \centering\includegraphics[scale=.9]{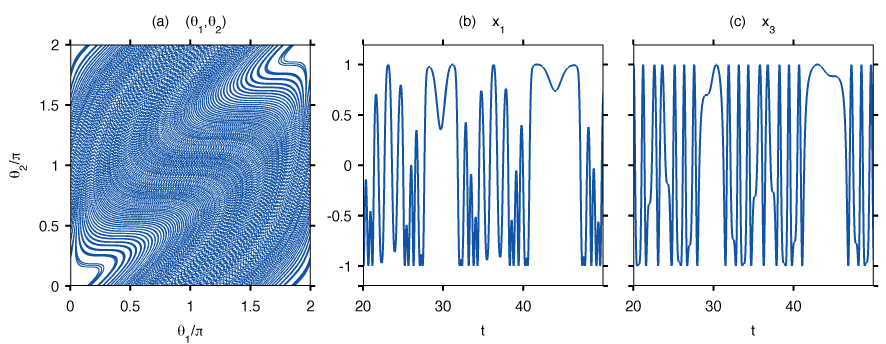}
  \caption{\label{figFixedPointTorusX}Time series for the fixed-point system on the 2-torus. (a) Phase space diagram on a periodic box for 64,000 samples; (b, c) components $ x^1 = \cos\theta^1$ and $ x^3 = \cos\theta^2 $ of the standard (flat) embedding $ F : \mathbb{T}^2 \mapsto \mathbb{R}^4 $. }
\end{figure}

First, we note that the methods of Section~\ref{secGalerkinImplementation} with no time change fail in the initial diffusion maps stage, as the Laplace-Beltrami eigenfunctions computed via Algorithm~\ref{algBasis} are corrupted by series of spikes (a hallmark of ill conditioning of the heat kernel matrix $ P$). We experimented with different kernels, tuning procedures, and normalizations (including the standard $ \alpha = 1/2 $ normalization of diffusion maps that requires no density estimation), but in all cases the quality of the eigenfunctions was poor. This ill-conditioning is likely caused by the behavior of the system near the fixed point, where the sampling density through finite-time trajectories has a singular, ``one-dimensional'' structure (see Fig.~\ref{figFixedPointTorusX}(a)), even though the asymptotic sampling density is uniform with respect to the Haar measure. On the other hand, after time change by the empirically accessible phase-space speed function $ \xi $, the quality of the eigenfunctions from diffusion maps improves markedly. We attribute this improvement to the modified Riemannian metric $ \hat h $ from Section~\ref{secImplementationT}. This metric becomes degenerate near the fixed point where $ \xi $ vanishes, assigning arbitrarily small norm to all tangent vectors (this can also be seen from the fact that the kernel in~\eqref{eqKVBT} assigns near-maximal affinity to all pairs of points with small corresponding $ \xi$). Therefore, the heat kernel associated with $ \hat h $ produces stronger averaging (smoothing) near the fixed point resulting in a well-behaved eigenfunction basis which is crucial to the success of the techniques of Section~\ref{secImplementationT}. 

In what follows, we work with the approximate eigenfunctions for the time-changed vector field computed using the advection-diffusion operator $ \hat L_{\varepsilon} $ from~\eqref{eqLHat} for the regularization parameter $ \varepsilon = 0.02 $. We selected this value as a reasonable compromise between bias error and smoothness of the computed eigenfunctions after testing various candidate values of $ \varepsilon $ in the interval $ 10^{-4} $ to $ 10^{-1} $. As shown in Fig.~\ref{figFixedPointTorusZ}, the generating eigenfunctions $ \{ \zeta_1, \zeta_2 \} $ for this value of $ \varepsilon $ do not lie on the unit circle with the same accuracy as the earlier results in Figs.~\ref{figTorusZ} and~\ref{figMixing3TorusZ}. Nevertheless, the eigenfunctions lie on a narrow annulus about the unit circle, and the corresponding time series have the structure of phase-modulated waves with weak amplitude modulation. The basic frequencies and Dirichlet energies are $ \{ \Omega_1, \Omega_2 \} = \{ 0.735, 0.165 \} $ and $ \{ E(\zeta_1), E(\zeta_2) \} =  \{ 1.54, 2.42 \} $. The eigenfunction time series exhibit timescale separation, with $ \zeta_1 $ evolving at faster timescales than $ \zeta_2 $, but this timescale separation has a time-dependent nature in the sense that both time series evolve slowly near the fixed point. The timescale separation between $ \zeta_1 $ and $ \zeta_2 $ is also evident from the scatterplots on the torus in Fig.~\ref{figFixedPointTorusZ}(a, d). There, it can be seen that the level sets of $ \zeta_2 $ are aligned with the orbits of the dynamics, whereas the level sets of $ \zeta_ 1$ are transverse to the dynamics resulting to rapid oscillations due to frequent level-set crossings. As discussed in the SOM, EDMD implemented with a dictionary consisting of lags of the state vector fails to recover Koopman eigenfunctions of comparable quality to those in Fig.~\ref{figFixedPointTorusZ}. In particular, as shown in Fig.~4 in the SOM, the EDMD spectrum contains an eigenfunction that somewhat resembles eigenfunction $ \zeta_2 $ in Fig.~\ref{figFixedPointTorusZ}, but is significantly more noisy. Moreover, we did not find evidence of an EDMD eigenfunction analogous to $ \zeta_1 $ which varies predominantly in the direction along the flow as opposed to directions transverse to the flow.

\begin{figure}
  \centering\includegraphics[scale=.8]{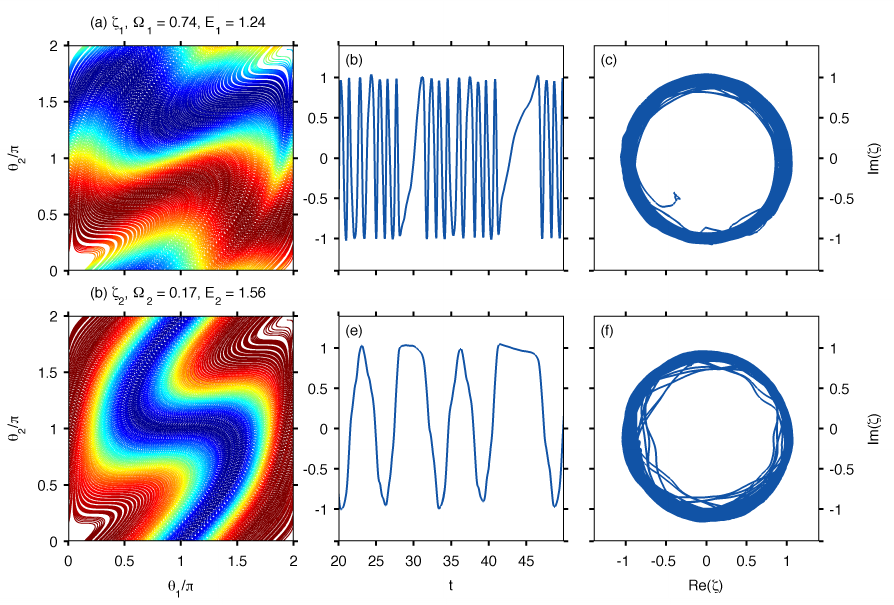}
  \caption{\label{figFixedPointTorusZ}Generating eigenfunctions $ \zeta_1 $ (a--c) and $ \zeta_2 $ (d--f) for the fixed point system on the 2-torus obtained via the time-change technique. (a, d) Scatterplots of $ \Real( \zeta_i ) $; (b, e) time series of $ \Real( \zeta_i ) $; (c, f) scatterplots of $ ( \Real( \zeta_i ), \Imag( \zeta_i ) ) $. The imaginary parts of $ \zeta_i $ are not shown in (b, e) for clarity, but to a good approximation they are $ 90^\circ $ phase-shifted versions of  $ \Real \zeta_i $ apart from near fixed points where that phase relationship holds less accurately.}   
\end{figure}

Figure~\ref{figFixedPointTorusZ} shows the vector field decomposition from Lemma~\ref{lemmaVDecompT} associated with the generating eigenfunctions from Fig.~\ref{figFixedPointTorusV}. This decomposition was performed using the bandwidth parameter $ l = 30 $, and similarly to the 3-torus example of Section~\ref{secMixing3Torus}, the results were projected to a periodic box on the plane for visualization. With this choice of bandwidth, the time-averaged reconstruction error of the full vector field was 7\%. Qualitatively, the vector field component $ V_1 $ corresponding to $ \zeta_1 $ describes a flow which is primarily directed along the $ \theta^2 $ direction, apart from a band centered around the $ \theta^1 = \theta^ 2 $ line where the flow turns towards the negative $ \theta^1 $ direction and its magnitude is diminished. On the other hand, the flow described by the component $ V_2 $ has significant magnitude along that band, together with inflow (outflow) from the $ \theta_2 > \theta_1 $ ($ \theta_2 < \theta_1 $) portions of the torus.       

\begin{figure}
  \centering\includegraphics{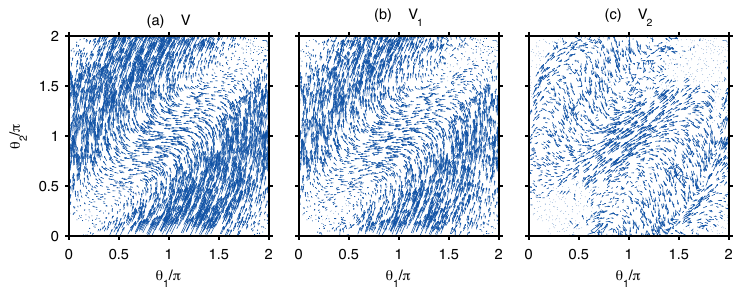}
  \caption{\label{figFixedPointTorusV}Vector field decomposition for the fixed-point system on the 2-torus. (a) Full vector field; (b, c) vector field components from Lemma~\ref{lemmaVDecompT}. All vector fields are visualized in a periodic box for the angles $ (\theta^1, \theta^2) $. As with Fig.~\ref{figMixing3TorusV}, the vector fields were first reconstructed in data space, $ \mathbb{ R }^4 $, and subsequently projected to the periodic domain. As a result, the arrow plots in (b, c) are noisier than the native vector fields in $ \mathbb{ R }^4 $.}
\end{figure}

Turning now to nonparametric forecasting, we use an initial probability measure $ \mu_0 $ from the same von-Mises family as the irrational flow example of Section~\ref{secIrrational}. That is, $ \mu_0 $ has the density function in~\eqref{eqVonMises} relative to the Haar measure, and in this case we set the location and concentration parameters to $ ( \bar \theta^1, \bar \theta^2 ) = ( \pi, \pi ) $ and $ \kappa = 30 $, respectively. The dynamic evolution of the density for the time interval $ [ 0, 10 ] $ is illustrated with snapshots in Fig.~\ref{figFixedPointTorusRho} and as a video in~Movie~2. With the chosen mean and concentration parameter, the initial density function is concentrated away from the fixed at point at $ ( 0, 0 ) $. In the course of dynamical evolution, the density function wanders close to the fixed point and becomes increasingly stretched. Eventually (around $ t =3.5 $), the density function appears to break up into disconnected components, but it is possible that these are artifacts caused by truncation to a finite number of basis functions and/or errors in the generating eigenfunctions and basic frequencies.

\begin{figure}
  \centering\includegraphics[scale=1]{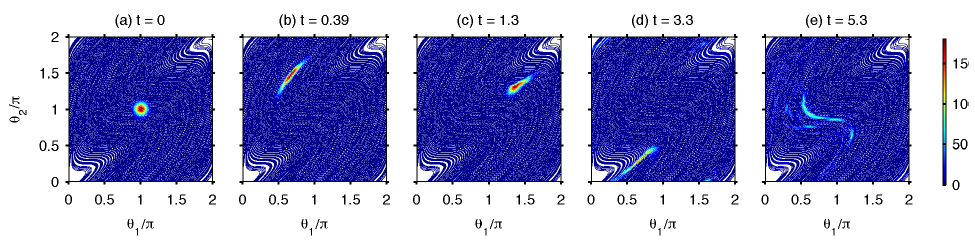}
  \caption{\label{figFixedPointTorusRho}Snapshots of the time-dependent probability density $ \rho_t $ relative to the equilibrium (Haar) measure for the fixed-point system on the 2-torus. (a) At initialization time, the density is strongly concentrated around the point $ ( \pi, \pi ) $; (b, c) at early times, most of the probability mass is located away from the fixed point at $ ( 0, 0 ) $; (d) the peak of the density function approaches the fixed point and becomes highly stretched; (e) eventually, the numerical density function breaks up into disconnected components.}
\end{figure} 
 
To assess the forecast skill of this nonparametric model, we compare its predictions against ensemble forecast obtained from 10,000 independent samples drawn from $ \rho_0 $ and evolved using the true model from~\eqref{eqFixedPointTorus}, taking the observation-map components $ x^1 = \cos \theta^1 $ and $ x^3 = \cos \theta^2 $ as our forecast observables. As shown in Fig.~\ref{figFixedPointTorusForecast}, over the time interval $ [ 0, 10 ] $ the mean forecast from the nonparametric model is in good agreement with the ensemble forecast. The nonparametric model also provides a reasonable uncertainty quantification through the predicted standard deviation, but as with the previous experiments the deviations from the ensemble forecast are greater for the standard deviation than the mean. In summary, the time-change approach is particularly effective in this class of systems both in terms of the quality of the diffusion maps basis (which is useful in other contexts besides the Koopman operators studied here), but also as a regularization tool that transforms the dynamical system to an orbit-equivalent system with improved spectral properties.

\begin{figure}
  \centering\includegraphics[scale=.7]{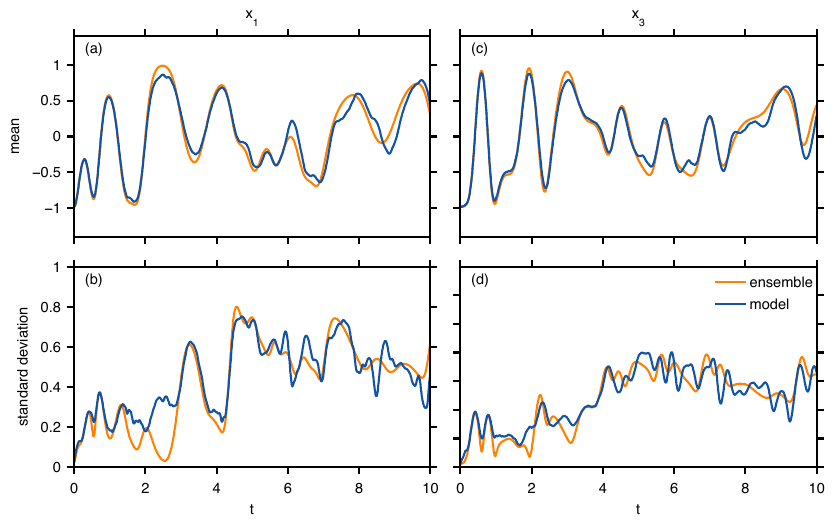}
  \caption{\label{figFixedPointTorusForecast}Nonparametric and ensemble forecasts of the mean and standard deviation of the components $ x^1 $ (a, b) and $ x^4 $ (c, d) of the canonical (flat) embedding of the 2-torus in $ \mathbb{ R }^4 $ for the fixed-point system. The initial probability measure has the circular Gaussian density in Fig.~\ref{figFixedPointTorusRho}(a). The nonparametric forecast is performed via Algorithm~\ref{algStatisticalForecast}, modified for time change as described in Section~\ref{secImplementationT}. The ensemble forecast is based on 10,000 independent samples drawn from the initial probability measure and evolved using the perfect model.}
\end{figure}

\section{\label{secConclusions}Concluding remarks}
 
In this work, we have developed a family of of data analysis techniques for dimension reduction, mode decomposition, and nonparametric forecasting of data generated by ergodic dynamical systems. Our approach is based on the Koopman and Perron-Frobenius formalisms for nonlinear dynamical systems, where the central objects of study are groups of unitary operators governing the evolution of observables and probability measures. In certain classes of systems (in particular, systems with pure point spectra), the spectral properties of these operators naturally lead to algorithms for dimension reduction and mode decomposition of spatiotemporal data, featuring timescale separation and strong invariance properties under changes of observation modality. The Perron-Frobenius operators also provide algorithms for equation-free forecasting of probability measures and expectation values of observables. Here, we develop these algorithms using a representation of the generator of the Koopman group in a complete orthonormal basis of the $ L^2 $ space of the dynamical system, acquired from time-ordered data  through the diffusion maps algorithm \cite{CoifmanLafon06}.  Placed in context, this work has connections with methods for mode decomposition and model reduction based on both the Koopman \cite{MezicBanaszuk04,Mezic05,Mezic13,SchmidSesterhenn08,RowleyEtAl09,ChenEtAl12,Schmid10,BudisicEtAl12,JovanovicEtAl14,TuEtAl14,HematiEtAl15,WilliamsEtAl15,KutzEtAl16,ProctorEtAl16,BudisicMezic12,BruntonEtAl17,ArbabiMezic16} and Perron-Frobenius \cite{DellnitzJunge99,DellnitzEtAl00,FroylandDellnitz03,Froyland05,Froyland07,Froyland08,FroylandPadberg09,FroylandEtAl10,FroylandEtAl10b,SchutteEtAl10,FroylandEtAl14,FroylandEtAl14b} perspectives, since the complete orthonormal basis learned from the data via diffusion maps allows us to pass between the two representations using standard linear algebra operations. Also, our approach has close connections with the nonparametric forecasting technique introduced in \citep[][]{BerryEtAl15}, which employs data-driven representations of the analogs of the Koopman and Perron-Frobenius operators in stochastic dynamical systems on manifolds (as opposed to the generator which is the object of focus in this work) in a similar basis from diffusion maps.

The ability to work in a complete orthonormal basis for the $ L^2 $ space of the dynamical system with a well-defined notion of smoothness has a number of advantages which enable us to (1) construct nonlinear dimension reduction maps based on Koopman eigenfunctions with projectible dynamics and small roughness on the data manifold; (2) decompose the dynamical vector field into a sum of mutually commuting vector fields, which we reconstruct in data space through a spectral representation of the pushforward map for vector fields on manifolds; (3) improve the efficiency and noise robustness of Galerkin methods for the Koopman eigenvalue problem through delay-coordinate maps; (4) predict the time-evolution of arbitrary probability densities and the expectation values of observables. These techniques perform best in the setting of systems with pure point spectra, where they lead to a decomposition of nonlinear dynamical systems into  uncoupled simple harmonic oscillators. We demonstrated the efficacy of these methods in numerical experiments with variable-speed flows on the torus having multiple timescales, large contrasts of the sampling density in ambient data space, and strong i.i.d.\ observational noise. We also established an explicit connection between Koopman operators for systems with pure point spectra and the Laplace-Beltrami operators approximated by diffusion maps applied to delay-coordinate mapped data, providing a  justification of the timescale separation seen in diffusion coordinates \cite{BerryEtAl13} for this class of systems.    

Another objective of this work has been to study and improve the regularity of numerically approximated eigenvalues and eigenfunctions of the Koopman group, particularly for mixing dynamical systems where the generator has no nonconstant eigenfunctions. In systems with pure point spectra, we demonstrated that adding a small amount of diffusion to the generator in an appropriate basis tailored to $ H_1 $ regularity eliminates oscillatory eigenfunctions with large Dirichlet energy for the given observation modality. This type of regularization is important even in simple systems such as irrational flows, since  generator may have no isolated eigenvalues, and thus highly oscillatory eigenfunctions may lie near (in the sense of the corresponding eigenvalues) eigenfunctions with low roughness.

In mixing systems, rather than regularizing the  generator by diffusion alone (which is known to impart singular changes \cite{ConstantinEtAl08,FrankeEtAl10} that are difficult to analyze in generality), we followed a different approach inspired by time-change methods in dynamical systems \cite{Fayad02,KatokThouvenot06}. In particular, we developed a strategy that involves rescaling the generator by the norm of the dynamical vector field in the ambient data space, and using the eigenvalues and eigenfunctions of the rescaled generator for dimension reduction and forecasting. This transformation is empirically computable from time-ordered data, and preserves the orbits of the dynamics while changing the flow of time along the orbits. In special cases, the transformation formally recovers a system with pure point spectrum from a class of time-changed mixing systems, but can be broadly used as an ad hoc regularization tool. We constructed analogs of our vector field decomposition and nonparametric forecasting techniques in the time-changed setting, where non-commuting vector fields and coupled oscillators encode the time-change function producing mixing. We demonstrated this approach in applications to a mixing flow on the 3-torus \cite{Fayad02}, and a challenging ergodic flow on the 2-torus with a fixed point \cite{Oxtoby53} where diffusion maps fails to produce a well-conditioned basis if no time change is applied.                
  
There is a number of areas for future research stemming from this work. First, the methods and applications discussed in the paper are heavily focused on dynamical systems on tori, so they should be investigated in more general classes of systems. In particular, we expect the techniques developed here to be useful in dynamical systems with mixed spectra, i.e., systems possessing invariant subspaces of their full $ L^2 $ space spanned by a (non-complete) set of eigenfunctions of the Koopman group. Also, it would be useful to study connections between Koopman operators and diffusion maps applied in delay-coordinate space for systems with nonzero Lyapunov exponents. Similarly, the time-change regularization strategy based on the norm of the dynamical vector field could be extended to other transformations having the goal of mapping the system under study to a system with improved spectral properties. At a more operational level, the algorithms formulated in this paper are all based on spectral expansions in global  bases of $ L^2 $ spaces, and these bases are generally inefficient in representing localized objects such as probability densities. It would therefore be fruitful to explore applications of multiscale bases (e.g., \cite{CoifmanMaggioni06,AllardEtAl12}) as alternatives to the global $ L^2 $ bases used here and in \cite{BerryEtAl15}. We plan to study these topics in future work.  
        
\section*{Acknowledgments}

The author would like to thank Tyrus Berry, Peter Constantin, Shuddho Das, John Harlim, Fanghua Lin, Andrew Majda, Lai-Sang Young, and Zhizhen Zhao for stimulating conversations on this work. This research was supported by DARPA grant HR0011-16-C-0116, NSF grant DMS-1521775, ONR grant N00014-14-1-0150, ONR MURI grant 25-74200-F7112, and ONR YIP grant N00014-16-1-2649.  

\appendix

\section{\label{appAlgorithms}Algorithms}

In this Appendix, we list the spectral decomposition and forecasting algorithms developed in Sections~\ref{secGalerkinImplementation} and~\ref{secTimeChange}.

Algorithm~\ref{algBasis} summarizes the construction of our data-driven orthonormal basis of $ L^2( M, \mu ) $ via diffusion maps.  This algorithm also includes a summary of the kernel density estimation and bandwidth-parameter tuning procedures. These procedures require a set $ \{ \epsilon_l \} $ of candidate bandwidth parameters and a nearest-neighbor truncation parameter $ k_\text{nn} $ as additional inputs; see \cite{BerryHarlim16,BerryEtAl15} for complete descriptions. We note that the content of Algorithm~\ref{algBasis} and the other algorithms in the paper are  intended to be at a conceptual level, and some of the steps would be implemented differently in practice to ensure efficiency. In particular, in applications with large sample numbers $ N $, we can truncate $ P $  to $ k \ll N $ nearest neighbors (incurring a small loss of accuracy by virtue of the exponential decay of $ K_\epsilon $). Also, it is customary for efficiency and stability to obtain the eigenvalues and eigenvectors of $ P $ from the eigenproblem for the symmetric matrix $ S = D^{1/2} P D^{-1/2} $, where $ D $ is an $ N \times N $ diagonal matrix with $ D_{ii} = \tilde d_i $ defined in Step~3 of the main calculation phase of Algorithm~\ref{algBasis}. The eigenfunctions of $ P $ and the inner product weights $ w $ can be computed from the eigenvectors of $ S $ through the relationships in the last step of Algorithm~\ref{algBasis}. We refer the reader to \cite{CoifmanLafon06,BerryHarlim16,BerrySauer16b,BerryEtAl15,Giannakis15} for further details on the numerical implementation and error estimates for this class of kernel algorithms.

\begin{alg}[Data-driven orthonormal basis]\label{algBasis}\ 
  \begin{itemize}
  \item Inputs 
    \begin{itemize} 
    \item Observed time series $ \{ x_i \}_{i=0}^{N-1} $, $ x_i \in \mathbb{ R }^d $, at sampling interval $ T $
    \item Candidate bandwidth parameter values $ \{ \epsilon_l \} $ with $ \epsilon_l = 2^l $ 
    \item Number of nearest neighbors $ k_\text{nn} $ for kernel density estimation
    \item Desired number of eigenvalues and eigenfunctions $ n $
    \end{itemize}
  \item Outputs
    \begin{itemize}
    \item Sampling densities $ \hat \sigma = ( \hat\sigma_\epsilon( x_0 ), \ldots, \hat \sigma_{\epsilon}(x_{N-1} ) )  $ relative to the Riemannian measure
    \item Estimated manifold dimension $ \hat m $
    \item Laplace-Beltrami eigenfunctions $ \{ \phi_0, \phi_1, \ldots, \phi_{n-1} \} $, $ \phi_i \in \mathbb{ R }^N $, and the corresponding eigenvalues $ \{ \eta_0, \eta_1, \ldots, \eta_{n-1} \} $, $ \eta_i \geq 0 $
    \item Inner product weights $ w \in \mathbb{ R }^N $
    \end{itemize}
  \item Density estimation phase
    \begin{enumerate}
    \item For each $ x_i $, compute the ad-hoc bandwidth function $ r^2_i = \sum_{j=2}^{k_\text{nn}} \lVert x_i - x_{I(i,j)} \rVert^2 / ( k_\text{nn} -1 )  $, where $ I( i,j ) $ is the index of the $ j $-th nearest neighbor of $ x_i $ in the dataset. 
    \item For each $ \epsilon_l $, compute the sum $ \Sigma_l = \sum_{i,j=0}^{N-1} \tilde K_{\epsilon_l}( x_i, x_j ) / N^2 $ for the kernel
      \begin{displaymath}
        \tilde K_{\epsilon_l}( x_i, x_j ) = \exp( - \lVert x_i - x_j \rVert^2 / ( \epsilon_l r_i r_j ) ).
      \end{displaymath}
    \item Choose the bandwidth parameter $ \epsilon \in \{ \epsilon_l \}  $ that maximizes $ \Sigma'_l = ( \log \Sigma_{l+1} - \log \Sigma_l ) / ( \log \epsilon_{l+1} - \log \epsilon_l ) $. The estimated manifold dimension is $ \hat m = 2 \Sigma'_l $.
    \item With the $ \epsilon $ and $ \hat m $ from Step~3, compute the sampling density 
      \begin{displaymath}
        \hat\sigma_\epsilon(x_i) = \sum_{j=0}^{N-1} \tilde K_\epsilon( x_i, x_j ) / [ N ( \pi \epsilon  r_i^2 )^{\hat m/2} ].
      \end{displaymath}
    \end{enumerate}  
  \item Main calculation phase
    \begin{enumerate}
    \item Select the bandwidth parameter $ \epsilon $ for the kernel $ K_\epsilon $ in~\eqref{eqKVB} using the same method as Steps~2 and~3 of the density estimation phase.
    \item Compute the vector $ \hat q \in \mathbb{ R }^N $ with $ \hat q_i = \sum_{j=0}^{N-1} K_\epsilon( x_i, x_j ) $ and the $ N \times N $ matrix $ \tilde K $ with $ \tilde K_{ij} =K_\epsilon( x_i, x_j ) / \hat q_j $.
    \item Compute the vector $ \tilde d \in \mathbb{ R }^N $ with $ \tilde d_i = \sum_{j=0}^{N-1} \tilde K_{ij} $ and the $ N \times N $ symmetric matrix $ S $ with  $ S_{ij} = \tilde K_{ij} / ( \tilde d_i \tilde d_j )^{1/2} $. $ S $ is related to $ P $ via the similarity transformation $ S = D^{1/2} P D^{-1/2} $, where $ D $ is the $ N \times N $ diagonal matrix with $ D_{ii} = \tilde d_i $.   
    \item Solve the eigenvalue problem $ S \tilde \phi_i = \kappa_i \tilde \phi_i $ for $ i \in \{ 0, 1, \ldots ,n-1 \} $ and $ \tilde \phi_i = ( \tilde \phi_{0i}, \ldots, \phi_{N-1,i} ) \in \mathbb{ R }^N $. Normalize the $ \tilde\phi_i $ to unit 2-norm.
    \item Set the eigenvalues to $ \eta_j = \log \kappa_j / \log \kappa_1 $, the inner product weights to $ w = ( w_0, \ldots, w_{N-1} )  \in \mathbb{ R }^N $ with $ w_i = \tilde \phi^2_{i0} $, and the eigenfunctions to $ \phi_j = ( \phi_{0j}, \ldots, \phi_{N-1,j} ) $ with $ \phi_{ij} = \tilde \phi_{ij} / \tilde \phi_{j0} $.
    \end{enumerate}
  \end{itemize}
\end{alg}       

Note that Algorithm~\ref{algBasis} is not too sensitive to the choice of the nearest neighbor parameter $ k_\text{nn} $ and the candidate bandwidth parameter values. Hereafter, we will always work with the values $ k_\text{nn} = 8 $ and $ \epsilon_l = 2^l $ with $ l \in \{ - 30, -29.9, \ldots, 9.9, 10 \} $, which were also used in \cite{BerryEtAl15}.

Algorithm~\ref{algGenerators} summarizes the numerical procedure to compute the generating frequencies and eigenfunctions of the Koopman group.

\begin{alg}[Generating frequencies and eigenfunctions]\label{algGenerators}\ 
  \begin{itemize}
  \item Inputs
    \begin{itemize}
    \item Diffusion eigenvalues $ \{ \eta_i \}_{i=0}^{n-1} , $ eigenfunctions $ \{ \phi_i \}_{i=0}^{n-1}$, and inner product weights $ w $ from Algorithm~\ref{algBasis} 
    \item Number of basic frequencies $ m $ (may be set to the estimated  dimension $ \hat m $ from Algorithm~1)
    \item Number of Koopman eigenfunctions to be computed, $ n' \leq n $
    \item Regularization parameter $ \varepsilon $
    \end{itemize}
  \item Outputs
    \begin{itemize} 
    \item Generating frequencies $ \{ \Omega_i \}_{i=1}^m $, $ \Omega_i \in \mathbb{ C } $, and eigenfunctions $ \{ \zeta_i \}_{i=1}^m $, $ \zeta_i \in \mathbb{ C }^N $ 
    \end{itemize}
  \item Execution steps
    \begin{enumerate}
    \item Compute the $ H_1 $-normalized eigenfunctions $ \{ \varphi_i \}_{i=0}^{n-1} $, $ \varphi_i \in \mathbb{ R }^ N $, using~\eqref{eqBasisH1}.
    \item Using the $ \{ \varphi_i \}_{i=0}^{n-1} $, compute the $ n \times n $ matrices $ D $, $ B $, and $ V $ from~\eqref{eqMatEig}, and form the matrix $ A = V  - \varepsilon D $. 
    \item Solve the generalized eigenvalue problem in~\eqref{eqGEV}; obtain the eigenvalues $ \{ \gamma_i \}_{i=1}^{n'} $, $ \gamma_i \in \mathbb{ C } $, and eigenvectors $ \{ c_i \}_{i=1}^{n'} $, $ c_i \in \mathbb{ C }^{n} $. Compute the Koopman eigenfunctions $ \{ u_i \}_{i=1}^{n'} $, $ u_i = \sum_{j=1}^n c_{ji} \varphi_i $. 
    \item Normalize each eigenfunction to unit norm $ \lVert u_i \rVert_{w}$, where $  \lVert u_i \rVert^2_{w}= \lvert c_{0i} \rvert^2 + \sum_{j=1}^{n}  \lvert c_{ji} \rvert^ 2 / \eta_j $. 
    \item Compute the Dirichlet energies $ E(u_i) = \sum_{j=1}^n  \lvert c_{ji} \rvert^2 $, and order the eigenfunctions and eigenvalues in order of increasing $ E(u_i) $. 
    \item Set $ \{ \Omega_i \}_{i=1}^m $ to the first $ m $ rationally independent frequencies $ \Imag \gamma_i $ and $ \{ \zeta_i \}_{i=1}^m $ to the corresponding eigenfunctions.    
    \end{enumerate}
  \end{itemize}
\end{alg}  

Next, we summarize the main steps in the numerical implementation of the vector field decomposition from Theorem~\ref{lemmaVDecomp} and the statistical forecasting scheme in Section~\ref{secForecasting} in Algorithms~\ref{algVDecomp} and~\ref{algStatisticalForecast}, respectively. In these algorithms, the notation $ \diag w $ refers to an $ N \times N $ diagonal matrix with diagonal elements equal to the components $ w_i $ of the inner product weight vector $ w $. Moreover, in Algorithm~\ref{algStatisticalForecast} we present our nonparametric forecast scheme for general vector-valued observables in $ \mathbb{ R }^s $ instead of the scalar-valued observables discussed in the main text. 

\begin{alg}[Vector field decomposition]\label{algVDecomp} \ 
  \begin{itemize}
  \item Inputs
    \begin{itemize}
    \item Observed time series $ \{ x_i \}_{i=0}^{N-1} $, $ x_i \in \mathbb{ R }^d $, at sampling interval $ T $
    \item Inner product weights $ w \in \mathbb{ R }^N $ from Algorithm~\ref{algBasis}
    \item Basic frequencies $ \{ \Omega_i \}_{i=1}^m $ and generating eigenfunctions $ \{ \zeta_i \}_{i=1}^m $ from Algorithm~\ref{algGenerators}
    \item Spectral order parameter $ l $
    \end{itemize}
  \item Outputs
    \begin{itemize}
    \item Vector fields $ \{ V_i \}_{i=1}^m $, where $ V_i = \{ V_{i0}, \ldots, V_{i,N-1} \}  $, and $ V_{ij} \in \mathbb{ R }^d $ is a tangent vector at $ x_j $
    \end{itemize}
  \item Preparatory steps
    \begin{enumerate}
    \item Arrange $ \{ x_i \}_{i=0}^{N-1} $ into a $ d \times N $ data matrix $ x = ( x_1 \cdots x_N ) $.
    \item Rescale the generators $ \zeta_i = ( \zeta_{0i}, \ldots, \zeta_{N-1,i} ) \in \mathbb{ C }^N $ to lie on the unit circle, $ \zeta_{ji} \leftarrow \zeta_{ji} / \lvert \zeta_{ji} \rvert $.
    \item Construct the index set $ K = \{ k_1, \ldots, k_n \} $ with $ k_i = ( q_1, \ldots, q_m ) $, $ \lvert q_i \rvert \leq l $ and $ n = ( 2 l + 1 )^m $. 
    \item For each $ i \in \{ 1, \ldots, n \} $ form the basis vector $ z_i = \prod_{j=1}^m \zeta_j^{q_j} $ with $ ( q_1, \ldots, q_m ) = k_i \in K $. Arrange the $ z_k $ in the $ N \times n $ matrix $ z = ( z_1 \cdots z_n ) $.
    \item Compute the Gramm matrix $ G \in \mathbb{ C }^{n\times n} $ with $ G_{ij} = \langle z_i, z_j \rangle_{w} $; in matrix notation, $ G = z^\dag \diag w z $.
    \end{enumerate}
  \item Execution steps
    \begin{enumerate}
    \item Compute the expansion coefficient matrix $ \hat F \in \mathbb{ C }^{N\times n } $ of the observation map s.t.\ $ \hat F = ( \hat F_1 \cdots \hat F_n ) $, $ \hat F_i = \langle z_i, F \rangle_{w} = \sum_{j=1}^N w_j x_j z^*_{ji} $. In matrix notation, $ \hat F = x \diag w z^* $.
    \item For $ i \in \{ 1, \ldots, m \} $ compute the diagonal matrix $ \bar \Omega_i \in \mathbb{ R }^{n\times n} $, where $ \bar \Omega_{i,jj} = q_i \Omega_i $ and $ q_i $ is the $ i $-th element of the index vector $ k_j \in K $. 
    \item Reconstruct the vector fields in~\eqref{eqVPushZeta} by forming the $ d \times N $ matrices $ V_i = \ii \hat F G^{-1 \top } \bar \Omega_i z^\top $. Set $ V_{ij} $ to the $ j $-th column of $ V_i $.     
    \end{enumerate}
  \end{itemize}
\end{alg}

\begin{alg}[Nonparametric forecast]\label{algStatisticalForecast} \ 
  \begin{itemize}
  \item Inputs
    \begin{itemize}
    \item Observable time series $ \{ f_i \}_{i=0}^{N-1} $, $ f_i \in \mathbb{ R }^s $, with sampling interval $ T $
    \item Inner product weights $ w \in \mathbb{ R }^N $ from Algorithm~\ref{algBasis}
    \item Discretely sampled initial probability density $ \rho_0 = (\rho_{00}, \ldots, \rho_{0,{N-1}} ) \in  \mathbb{ R }^N $, normalized such that $ \sum_{i=1}^N \rho_{0i} = 1 $
    \item Forecast times $ \{ t_i \}_{i=1}^{N'} $ in multiples of $ T $ 
    \item Basic frequencies $ \{ \Omega_i \}_{i=1}^m $ and generating eigenfunctions $ \{ \zeta_i \}_{i=1}^m $ from Algorithm~\ref{algGenerators}
    \item Spectral order parameter $ l $
    \end{itemize}
  \item Outputs
    \begin{itemize}
    \item Forecast densities $ \{ \rho( t_i ) \}_{i=1}^{N'} $, $ \rho( t_i ) \in \mathbb{ R }^N $
    \item Expectation values $ \{ \bar f( t_i ) \}_{i=1}^{N'} $, $ \bar f( t_i ) \in \mathbb{ R }^s $, of the observable  
    \end{itemize}
  \item Preparatory steps
    \begin{enumerate}
    \item Repeat the preparatory steps of Algorithm~\ref{algVDecomp}.
    \item Compute the frequency $ \omega_i = \sum_{j=1}^m q_j \Omega_j $ corresponding to each $ k_i = ( q_1, \ldots, q_m ) \in K $. 
    \end{enumerate}
  \item Execution steps
    \begin{enumerate}
    \item Compute the expansion coefficient matrix $ \hat f \in \mathbb{ C }^{s \times n } $ of the observable s.t. $ \hat f = f \diag w z^* $.
    \item Compute the expansion coefficients $ \hat \rho = ( \hat \rho_1, \ldots, \hat \rho_n ) $ of the initial density with $ \hat \rho = \rho_0 \diag w z^* $; pre-multiply with the inverse Gramm matrix to form the coefficients $ \tilde \rho_0 = G^{-1} \hat \rho $. 
    \item For each forecast time $ t_i $
      \begin{enumerate}
      \item Advance the expansion coefficients of the density using $ \tilde \rho_j( t_i ) = e^{-\ii \omega_j t_i } \tilde \rho_j $. Arrange the result in a row vector $ \tilde \rho( t_i ) = ( \tilde \rho_1( t_i ), \ldots, \tilde \rho_n( t_i ) ) \in \mathbb{ C }^n $.
      \item Reconstruct the density $ \rho( t_i ) $ using $ \rho( t_i ) = \tilde \rho( t_i ) z^\top $.
      \item Compute the expectation value of the observable $ \bar f( t_i ) = \hat f \tilde \rho^\dag( t_i ) $.
      \end{enumerate}
    \end{enumerate}
  \end{itemize}
\end{alg}
 
\section{\label{appGBar}Proof of Theorem~\ref{thmGBar}}

According to Theorem~\ref{lemmaVDecomp}, the vector fields $ v_i $ are nowhere-vanishing and linearly independent. Therefore, to verify the desired expression for $ \bar g $ it suffices to show that $ \bar g( v_i, v_j ) = B_{ij}  $. Indeed, using~\eqref{eqVPushZeta}, \eqref{eqGS}, and Theorem~\ref{lemmaVDecomp}, we find
  \begin{align*}
    \bar g( v_i, v_j ) &=  \lim_{s\to\infty} \frac{ 1}{ s } \sum_{k=0}^{s-1} \langle F_* \hat \Phi_{k*} v_i, F_* \hat \Phi_{k*} v_j \rangle \\
     &=  \lim_{s\to\infty} \frac{ 1}{ s } \sum_{k=0}^{s-1} \langle  v_i( F \circ \hat \Phi_k ), \hat  v_j( F \circ \hat \Phi_k ) \rangle \\
     & = \lim_{s\to\infty} \frac{ 1}{ s } \sum_{k=0}^{s-1} \left<  v_i\left(  \sum_p \hat F_p^* ( \zeta_1^{-p_1} \cdots \zeta_m^{-p_m} ) \circ \hat \Phi_k \right) , v_j\left(  \sum_q \hat F_q ( \zeta_1^{q_1} \cdots \zeta_m^{q_m} ) \circ \hat \Phi_k \right)  \right> \\
      & = \lim_{s\to\infty} \frac{ 1}{ s } \sum_{k=0}^{s-1} \left<  v_i\left(  \sum_p \hat F_p^* \zeta_1^{-p_1} \cdots \zeta_m^{-p_m}  e^{-\sum_{l=1}^m \ii k \Omega_l T p_l} \right) , v_j\left(  \sum_q \hat F_q \zeta_1^{q_1} \cdots \zeta_m^{q_m}  e^{\sum_{n=1}^m \ii k \Omega_n T q_n} \right)  \right> \\
      & = \lim_{s\to\infty} \frac{ 1}{ s } \sum_{k=0}^{s-1} \left<  \left(  \sum_p \hat F_p^* \Omega_i p_i \zeta_1^{-p_1} \cdots \zeta_m^{-p_m}  e^{-\sum_{l=1}^m \ii k \Omega_l T p_l} \right) , \left(  \sum_q \hat F_q \Omega_j q_j \zeta_1^{q_1} \cdots \zeta_m^{q_m}  e^{\sum_{n=1}^m \ii k \Omega_n T q_n} \right)  \right> \\
      & =    \sum_{p,q} \langle \hat F_p^*, \hat F_q \rangle \Omega_i \Omega_j p_i q_j \zeta_1^{q_1-p_1} \cdots \zeta_m^{q_m-p_m} \left( \lim_{s\to\infty} \frac{ 1}{ s } \sum_{k=0}^{s-1} e^{\sum_{l=1}^m \ii k \Omega_l T (q_l-p_l)} \right) \\
      & =  \sum_{p,q} \langle \hat F_p^*, \hat F_q \rangle \Omega_i \Omega_j p_i q_j \zeta_1^{q_1-p_1} \cdots \zeta_m^{q_m-p_m} \delta_{q_1p_1} \cdots \delta_{q_mp_m} \\
      & = \sum_{p} \lVert \hat F_p \rVert^2  \Omega_i \Omega_j p_i p_j. 
  \end{align*}  
Note that the sum in the last step is absolutely summable due to the exponential decay of the Fourier coefficients $ \hat F_p $ of the $ C^\infty $ function $ F $,  which allows us to interchange limits with respect to $ s $ and $ p,q $.  Also, we have used the result $ v_i\left(  \sum_k \hat f_k z_k \right) = \tilde v_i\left(  \sum_p \hat f_k z_k \right) = \sum_p \ii \Omega_i k_i \hat f_k $, which holds for any $ f = \sum_k f_k z_k \in C^\infty( M ) \subset D( \tilde v_i ) $ because the generators $ \tilde v_i $ (which are maximal skew-adjoint extensions of $ v_i $; see the text below Theorem~\ref{lemmaVDecomp}) are closed.  We have further made use of the relation $ \lim_{s\to\infty} \frac{ 1}{ s } \sum_{k=0}^{s-1} e^{\sum_{l=1}^m \ii k \Omega_l T (q_l-p_l)} = \delta_{q_1p_1} \cdots \delta_{q_mp_m} $, which follows by ergodicity of the discrete-time system $ ( M, \mathcal{ B }, \mu,  \hat \Phi_n ) $ (which implies in turn that $ \Omega_l T  $ is irrational for all $ l \in \{ 1, \ldots, m \} $). Similarly, we compute
\begin{align*}
  B_{ij} &= \int_M g( v_i, v_j ) \, d\mu \\
  &= \int_M  \langle F_* v_i, F_* v_j \rangle \, d\mu \\
  & = \int_M  \langle v_i( F ), v_j( F ) \rangle \, d\mu \\
  & = \int_M \left< v_i \left( \sum_p \hat F^*_p \zeta^{-p_1}_1 \cdots \zeta^{-p_m}_m  \right), v_j \left( \sum_q \hat F^*_q \zeta^{q_1}_1 \cdots \zeta^{q_m}_m  \right) \right> \, d\mu \\
  & = \int_M \left< \left( \sum_p \hat F^*_p \Omega_i p_i \zeta^{-p_1}_1 \cdots \zeta^{-p_m}_m  \right),  \left( \sum_q \hat F^*_q \Omega_j q_j \zeta^{q_1}_1 \cdots \zeta^{q_m}_m  \right) \right> \, d\mu \\
& = \sum_{p,q} \langle \hat F^*_p, F_q \rangle \Omega_i \Omega_j p_i q_i \langle \zeta_1^{p_1} \cdots \zeta_m^{p_m}, \zeta_1^{q_1} \cdots \zeta_m^{q_m} \rangle \\
& = \sum_{p,q} \langle \hat F^*_p, F_q \rangle \Omega_i \Omega_j p_i q_i \delta_{p_1 q_1 } \cdots \delta_{p_m q_m} \\
&= \sum_{p} \lVert \hat F_p \rVert^2 \Omega_i \Omega_j  p_i p_j. 
\end{align*}  
We therefore have $ \bar g( v_i, v_j ) = B_{ij}  $ as claimed above.

Next, to verify that $ \Phi_{n,t*} $ is an isometry of $ \bar g $ as claimed in part~(i) of the Theorem, it is sufficient to show that $ \bar g( v_i\vert_a, v_j\rvert_a) = \bar g(   \Phi_{n,t*}v_i \rvert_a,  \Phi_{n,t*} v_j \rvert_a ) $ for all $ i, j \in \{ 1, \ldots, m \} $ and $ a \in M $. This can be confirmed using Theorem~\ref{lemmaVDecomp}(iv), according to which
\begin{align*}
  \bar g(   \Phi_{n,t*} v_i \rvert_a,  \Phi_{n,t*}v_j \rvert_a ) &= \bar g(  v_i \rvert_{ \Phi_{n,t}(a)}, v_j \rvert_{ \Phi_{n,t}(a)} ) \\
  &= \sum_{k,l=1}^m B_{kl} ( \beta_k( v_i ) ) \rvert_{\Phi_{n,t}(a) }  ( \beta_l(  v_j ) )\rvert_{\Phi_{n,t}(a) } \\
  & = B_{ij} \\
  & = \bar g( v_i\vert_a, v_j\rvert_a). 
\end{align*}
Moreover, to show that $ \bar g $ is flat as stated in part~(ii) we note that each $ v_i $ is a Killing vector field since it generates a one-parameter group of isometries $ \Phi_{i,t} $, and therefore, by Theorem~\ref{lemmaVDecomp}, $ \{ v_i \}_{i=1}^m $ is a set of $ m $ linearly independent, mutually commuting, Killing vector fields. It is a standard result from differential geometry that whenever such vector fields exist in an open neighborhood of a point in $ M $ the metric is flat at that point, and the global flatness of $ \bar g $ follows from the fact that the $ v_i $ are globally defined on $M $. 

Finally, we check that $ \dvol_{\bar g}/d\mu  $ is a constant as stated in part~(iii) working in local coordinates $ \{ \theta^1, \ldots, \theta^m \} $ such that $ v_i = \frac{\partial\;}{ \partial \theta^i } $. In these coordinates $ \bar g\rvert_a $ has constant components $ \bar g_{ij} = B_{ij} $ at every $ a \in M $, and we have $ \dvol_{\bar g}\rvert_a = \sqrt{ \det B } \, d\theta^1 \wedge \cdots \wedge d\theta^m $, where $ B = [B_{ij}] $.  Similarly, we can expand $ d \mu\rvert_a = \gamma(\theta) \, d\theta^1 \wedge \cdots d\theta^m $, where $ \gamma $ is a smooth positive function. Now, by Theorem~\ref{lemmaVDecomp}(iii), the $ v_i $ have vanishing $ \mu $-divergence, and using the local-coordinate expression for the divergence in~\eqref{eqGradDiv} we obtain $ 0 = \divr_\mu v_i =  \gamma^{-1} \partial \gamma / \partial \theta^i $. Thus, $ ( \dvol_{\bar g} / d\mu )\rvert_{a} = \sqrt{ \det B } / \gamma(\theta) = \Gamma $ is locally a constant, and since the $ v_i $ are smooth and defined at every point in $ M $, $ \dvol_{\bar g} / d\mu $ is also globally constant. \qed

\section{\label{appNoise}Treatment of i.i.d.\ noise through delay embeddings and diffusion maps}

In this Appendix, we study the asymptotic properties of the procedure employed in Section~\ref{secNoise} to remove the effects of i.i.d.\ observational noise from the diffusion eigenfunctions computed through Algorithm~\ref{algBasis}. In particular, we are interested in the behavior of diffusion maps at small kernel bandwidth $ \epsilon $, and in the limit of infinitely many delays $ s $  and a suitable scaling $ N(s) \gg s $ of the number of samples $ N $. 
\begin{thm} 
  \label{thmNoise}
  Let $ f $ be an observable in $ L^2( M, \mu ) $ with the values $ ( f( a_{s-1} ), \ldots, f( a_{N-1} ) ) = \vec f  $ on the sampled states $ \{ a_i \}_{i=s-1}^{N-1} $. Under the assumptions on the noise stated in Section~\ref{secNoise} and for an ergodic dynamical system with a pure point spectrum and smooth Koopman eigenfunctions, for $ \mu $-a.e.\ starting state $ a_0 $ in the training data, a.s.\ with respect to the noise, and uniformly with respect to $a_i \in  M $,
  \begin{displaymath}
    \lim_{\substack{s\to\infty\\ N(s)/s \to \infty }}( \tilde P \vec f )_i = \mathcal{ P }_\epsilon f( a_i ) + O( \epsilon^2 ),
  \end{displaymath}
  where $ \mathcal{ P }_\epsilon $ is the averaging operator on $ L^2( M, \mu ) $ constructed through~\eqref{eqG} for the kernel in~\eqref{eqKVB2}, and $ \tilde P $ is the corresponding Markov matrix constructed from the noisy data via~\eqref{eqPNoise}.     
\end{thm}

To prove Theorem~\ref{thmNoise}, note that the noise variables in delay embedding space are sequences $ \Xi_i = ( \xi_i, \xi_{i-1}, \ldots, \xi_{i-s+1} ) $ of the i.i.d.\ noise variables $ \xi_i $, and for fixed $i $ and $ N(s) \gg s $, the fraction of indices $ j $ such  that $ \Xi_i $ and $ \Xi_j $ have at least one overlapping element (which is $ O( s/N(s) ) $) vanishes as $ s \to\infty$. Thus, it suffices to consider the case $ \lvert i - j \rvert > s  $, where
\begin{displaymath}
  \mathbb{ E }( \lVert \Xi_i - \Xi_j \rVert^2 ) = \frac{ 1 }{ s } \sum_{k=0}^{s-1} \mathbb{ E }(\lVert \xi_{i-k} - \xi_{j-k} \rVert^2 ) = \frac{ 1 }{ s }  \sum_{k=0}^{s-1}( \mathbb{ E }(\lVert \xi_{i-k} \rVert^2) + \mathbb{ E }( \lVert \xi_{j-k} \rVert^2 )  ) = 2 R^2,
\end{displaymath}
and we have used the facts that $ \mathbb{ E }( \xi_i ) = 0 $ and $ \mathbb{ E }( \lVert \xi_i \rVert^2 ) = R^2 $. Similarly, writing  $ \tilde X_i - \tilde X_j = X_i - X_j + \Xi_i - \Xi_j $, it follows that the squared pairwise distances $  \lVert \tilde X_i - \tilde X_j \rVert^2 $ between the noisy data have the expectation 
\begin{displaymath}
  \mathbb{ E }( \lVert \tilde X_i - \tilde X_j \rVert^2 ) = \lVert X_i - X_j \rVert^2 + \mathbb{ E }( \lVert \Xi_i - \Xi_j \rVert^2 ) = \lVert X_i - X_j \rVert^2 + 2 R^2. 
\end{displaymath}

Define now the fixed bandwidth kernel $ \bar K_\epsilon( \tilde X_i, \tilde X_j ) = e^{-\lVert \tilde X_i - \tilde X_j \rVert^2 / \epsilon }$ and the normalized kernel
\begin{equation}
  \label{eqKappa}
  \kappa_\epsilon( \tilde X_i, \tilde X_j ) = \frac{ \bar K_\epsilon( \tilde X_i, \tilde X_j ) }{ ( N(s) -s + 1 )^{-1} \sum_{k,l=s-1}^{N(s)-1} \bar K_\epsilon( \tilde X_k, \tilde X_l ) }.
\end{equation}
Note that the sum in the denominator in~\eqref{eqKappa} runs from $ s- 1 $ to $ N(s) -1 $ (instead of 0 to $ N(s) -1$) since the first $ s - 1 $ samples $ \tilde x_i $ are used for delay embedding. We will use the kernel in~\eqref{eqKappa} to estimate the sampling density $ \sigma_s = d \mu/ \dvol_{g_s} $  of the invariant measure relative to the Riemannian measure at $ s $ delays via the quantity
\begin{displaymath}  
\tau_{s,\epsilon}( \tilde X_i ) = \frac{ 1 }{ ( N(s) - s + 1 ) } \sum_{j=s-1}^{N(s)-1} \kappa_\epsilon( \tilde X_i, \tilde X_j ).
\end{displaymath}
Because $ N(s ) \gg s  $, the contribution of the delay sequences $ \tilde X_j $ in the sum with overlapping samples with $ \tilde X_i $ vanishes as $ N\to\infty $. Thus, we may compute $ \lim_{s\to\infty} \tau_{s,\epsilon}( \tilde X_i ) $ by formally assuming that $ \tilde X_i $ and $ \tilde X_j $ are non-overlapping. In particular, since the random variables $ \xi_i  $ have finite moments up to order 4, $ \lVert \xi_i \rVert^2 $ have finite variance and by the law of large numbers, as $ s \to \infty $,  $ \lVert \Xi_i - \Xi_j \rVert^2 $ converges almost surely to $ 2 R^2 $ in the non-overlapping case, and moreover
\begin{equation}
  \label{eqKLim}
  \bar K_\epsilon( \tilde X_i, \tilde X_j ) \xrightarrow[]{\text{a.s.}} e^{-2R^2/\epsilon} \bar K_\epsilon( X_i, X_j ),
\end{equation}  
where we used the fact that $ \bar K_\epsilon $ is continuous and bounded. Hence, $ \bar K_\epsilon( \tilde X_i, \tilde X_j ) $ acquires a constant multiplicative bias $ e^{-2R^2/\epsilon} $, but this bias appears in both the numerator and denominator in~\eqref{eqKappa} and is therefore canceled. We therefore conclude from these arguments and the Birkhoff ergodic theorem that as $ s \to \infty $ and $ N(s) \gg s $, 
\begin{displaymath}
  \tau_{s,\epsilon}( \tilde X_i ) \xrightarrow[]{\text{a.s.}}  \bar \tau_\epsilon( X_i ) = \lim_{s\to\infty} \int_M \kappa_\epsilon( F_s( a_i ), F_s( a_j ) ) \, d\mu(a_j) = \lim_{s\to\infty} \frac{ \int_M \bar K_\epsilon( F_s( a_i ), F_s( a_j ) ) \, d\mu( a_j ) }{ \int_M\left( \int_M \bar K_\epsilon( F_s( a_k ), F_s( a_l ) ) \, d\mu( a_k ) \right) d\mu( a_l ) },
\end{displaymath}
where $ a_i $ the unique state in $ M $ underlying the noisy observation $ \tilde X_i $ (i.e., $ \tilde X_i = X_i + \Xi_i $ with $ X_i = F_s( a_i ) )$. Note that the cancellation of the multiplicative bias term $ e^{-2R^2/\epsilon} $ would not have taken place had we used the variable-bandwidth kernel $ \tilde K_\epsilon $ in Algorithm~\ref{algBasis} for density estimation, as the bias would depend on $ i, j $ in that case. 

Next,  using  asymptotics for Gaussian integrals on compact manifolds (e.g., \cite{CoifmanLafon06,BerrySauer16b}) it can be shown that, uniformly with respect to $ a_i \in M $,
\begin{align*}
  \frac{ 1 }{ \epsilon^{m/2} } \int_M \bar K_\epsilon( F_s(a_i), F_s(a_j) ) \, d\mu(a_j) &=  \frac{ 1 }{ \epsilon^{m/2} }  \int_M \bar K_\epsilon( F_s(a_i), F_s(a_j) ) \sigma_s( a_j ) \dvol_{g_s}( a_j) \\
  &= c_s \sigma_s ( a_i ) + \sigma'_s(a_i) \epsilon + O( \epsilon^2 ),
\end{align*}
where $ c_s $ is a constant that does not depend on $ \epsilon$, and $ \sigma'_s $ is a function that vanishes if $ g_s $ is flat. Indeed, according to Theorem~\ref{thmGBar}, $ \bar g $ is flat and has uniform volume form relative to $ \mu $, so that $ \lim_{s\to\infty} \sigma_s = 1 / \Gamma $, $ \lim_{s\to\infty} \sigma'_s = 0 $, and    
\begin{equation}
  \label{eqTauLim}
  \bar \tau_\epsilon( X_i ) = \lim_{s \to \infty} \frac{ \sigma_s( a_i ) + \sigma'_s( a_i ) \epsilon + O( \epsilon^2 ) }{ \int_M [ ( \sigma_s( a_l ) )^2 + \sigma_s( a_l ) \sigma_s'(a_l) \epsilon ] \, \dvol_{g_s}( a_l ) + O( \epsilon^2 ) } = C^m + O( \epsilon^2 ),
\end{equation}
where $ C $ is a positive constant. We therefore conclude that in the limit of large data and infinitely many delays $ \tau_\epsilon( \tilde  X_i ) $ converges up to $ O( \epsilon^2 ) $ to a constant on $ M $. Since in this limit the sampling density relative to the Riemannian measure is also a constant, $ \tau_\epsilon( \tilde X_i ) $ provides an $ O( \epsilon^2 ) $ estimate of the true sampling  density up to a proportionality constant, which is of sufficient accuracy for the asymptotics  in \cite{BerryHarlim16} to hold.

Consider now the diffusion maps normalization performed on the kernel $ K_\epsilon $ in~\eqref{eqKVB2} for the noisy data. Following the approach described in Section~\ref{secDataDrivenBasis}, we normalize $ K_\epsilon( \tilde X_i, \tilde X_j ) $ to construct a Markov matrix $ \tilde P $ such that 
\begin{equation}
  \label{eqPNoise}
  \tilde P_{ij} = \frac{ \tilde H_{ij} }{ \sum_{k=s-1}^{N-1} \tilde H_{ij} }, \quad \tilde H_{ij} = \frac{ K_\epsilon( \tilde X_i, \tilde X_j ) }{ \sum_{k=s-1}^{N-1} K_\epsilon( \tilde X_j, \tilde X_k ) }. 
\end{equation}
Then, we use $ \tilde P $ to approximate the action $ \mathcal{ P }_\epsilon f $ of the averaging operator $ \mathcal{ P }_\epsilon $ associated with the kernel on the observable $ f  $ through the matrix-vector product $ \tilde P \vec f $. By the same arguments used to derive~\eqref{eqTauLim}, we can compute 
\begin{displaymath}
  \lim_{s\to\infty} ( \tilde P \vec f )_i = \lim_{s\to\infty} \frac{ 1}{ N(s)-s+1} \sum_{j=s+1}^{N(s)-1} \tilde P_{ij} f_j
\end{displaymath}
treating the matrix elements $ \tilde P_{ij} $ as corresponding to non-overlapping sequences $ \tilde X_i, \tilde X_j $ in~\eqref{eqPNoise}. In that case, by the law of large numbers and~\eqref{eqTauLim}, 
\begin{displaymath}
  K_\epsilon( \tilde X_i, \tilde X_j ) = \exp\left( - \frac{ \lVert \tilde X_i - \tilde X_j \rVert^2 }{ \epsilon \tau^{-1/m}_{s,\epsilon}( \tilde X_i ) \tau_{s,\epsilon}^{-1/m }( \tilde X_j )  } \right) \xrightarrow[]{\text{a.s.}} \exp \left( - \frac{ 2 R^2 }{ \epsilon ( C^2 + O( \epsilon^2 ) ) } \right) K_\epsilon( X_i, X_j ).    
\end{displaymath}  
Using this result together with the pointwise ergodic theorem to convert sums in~\eqref{eqPNoise} to integrals with respect to $ \mu $, we conclude that as $ s \to \infty $, $ N(s ) \gg s  $, 
\begin{displaymath}
  ( \tilde P \vec f )_i \xrightarrow[]{\text{a.s.}} \mathcal{ P }_\epsilon f( a_i ) + O( \epsilon^2 ).
\end{displaymath}
 proving the Theorem. \qed

We therefore see that provided that sufficiently many delays are used,  the effect of i.i.d.\ noise is to produce an $ O( \epsilon^2 ) $ bias in the pointwise approximation of the action of  $ \mathcal{ P }_\epsilon $ on functions. However, this bias is of the same order as the error in approximating $ \upDelta_{\bar g} $ through $ \mathcal{ P }_\epsilon $ (see~\eqref{eqPDM});  that is,  in the limit of infinitely many delays, $ ( I - \tilde P \vec f_i )/\epsilon $ approximates $ \upDelta_{\bar g}f( a_i ) $ at the same accuracy as $ ( I - \mathcal{ P }_\epsilon )f( a_i ) / \epsilon $.

\end{document}